\documentclass[11pt]{article}
\usepackage{geometry}
\usepackage[latin1]{inputenc}
\usepackage{amsmath}
\usepackage{amsthm,amssymb,amsfonts}
\usepackage{graphicx}
\usepackage{verbatim}
\usepackage{color,cite}
\usepackage{latexsym}
\usepackage{lscape}
\usepackage[T1]{fontenc}
\usepackage[all,cmtip]{xy}
\usepackage{hyperref}
\usepackage{caption}
\usepackage{xspace}
\usepackage{relsize}
\usepackage{amsmath, bm}

\usepackage{textgreek}

\usepackage{enumitem}
\usepackage{tikz}
\usetikzlibrary{arrows.meta}
\usepackage{xcolor}
\usetikzlibrary{patterns}
\usetikzlibrary{shapes,arrows,spy,positioning}

\theoremstyle{plain}
\newtheorem{theorem}{Theorem}[section]
\newtheorem{lemma}[theorem]{Lemma}
\newtheorem{proposition}[theorem]{Proposition}
\newtheorem{corollary}[theorem]{Corollary}

\newtheorem{definition}{Definition}
\newtheorem{example}[theorem]{Example}

\newtheorem*{question}{Question}

\newtheorem{remark}{Remark}
\newtheorem*{convention}{Convention}

\newcommand{\nc}{\newcommand}
\newtheorem*{claim}{Claim}

\nc\bB{\mathbb{B}}
\nc\bC{\mathbb{C}}
\nc\bD{\mathbb{D}}
\nc\bE{\mathbb{E}}
\nc\bF{\mathbb{F}}
\nc\bG{\mathbb{G}}
\nc\bH{\mathbb{H}}
\nc\bI{\mathbb{I}}
\nc{\bJ}{\mathbb{J}}
\nc\bK{\mathbb{K}}
\nc\bL{\mathbb{L}}
\nc\bM{\mathbb{M}}
\nc\bN{\mathbb{N}}
\nc\bO{\mathbb{O}}
\nc\bP{\mathbb{P}}
\nc\bQ{\mathbb{Q}}
\nc\bR{\mathbb{R}}
\nc\bS{\mathbb{S}}
\nc\bT{\mathbb{T}}
\nc\bU{\mathbb{U}}
\nc\bV{\mathbb{V}}
\nc\bW{\mathbb{W}}
\nc\bY{\mathbb{Y}}
\nc\bX{\mathbb{X}}
\nc\bZ{\mathbb{Z}}

\nc\cA{\mathcal{A}}
\nc\cB{\mathcal{B}}
\nc\cC{\mathcal{C}}
\nc\cD{\mathcal{D}}
\nc\cE{\mathcal{E}}
\nc\cF{\mathcal{F}}
\nc\cG{\mathcal{G}}
\nc\cH{\mathcal{H}}
\nc\cI{\mathcal{I}}
\nc{\cJ}{\mathcal{J}}
\nc\cK{\mathcal{K}}
\nc\cL{\mathcal{L}}
\nc\cM{\mathcal{M}}
\nc\cN{\mathcal{N}}
\nc\cO{\mathcal{O}}
\nc\cP{\mathcal{P}}
\nc\cQ{\mathcal{Q}}
\nc\cR{\mathcal{R}}
\nc\cS{\mathcal{S}}
\nc\cT{\mathcal{T}}
\nc\cU{\mathcal{U}}
\nc\cV{\mathcal{V}}
\nc\cW{\mathcal{W}}
\nc\cY{\mathcal{Y}}
\nc\cX{\mathcal{X}}
\nc\cZ{\mathcal{Z}}

\nc\bl{\bold{l}}
\nc\bv{\bold{v}}
\nc\bj{\bold{j}}

\newcommand{\C}{\mathbb{C}}

\nc\blambda{\mathbf{\lambda}}
\newcommand{\om}{\omega}
\newcommand{\SC}{\operatorname{SC}}

\newcommand{\Cyl}{\operatorname{Cyl}}

\newcommand{\splin}{\text{SL}(2,\mathbb{R})}

\newcommand{\La}{\Lambda}
\newcommand{\brho}{\boldsymbol{\rho}}

\nc{\red}[1]{\textcolor{red}{#1}}
\nc{\blue}[1]{\textcolor{blue}{#1}}
\nc{\ann}[1]{\marginpar{\small{\red{#1}}}}

\title{Counting Saddle Connections on Hyperelliptic Translation Surfaces with a Slit}
\author{David Aulicino, Howard Masur, Huiping Pan, and Weixu Su\footnote{D.A. was partially supported by the National Science Foundation under Award No. DMS - 1738381, the Simons Foundation: Collaboration Grant under Award No. 853471, and several PSC-CUNY Grants.  H. P. is supported by National Natural Science Foundation of China NSFC 12371073, 12631004 and Scientific Research Innovation Capability Support Project for Young Faculty SRICSPYF-BS2025135. W. S. is  supported by National Natural Science Foundation of China NSFC No. 12371076 and 12631004.}}

\date{}

\begin{document}

\maketitle

\noindent

\begin{abstract}
We consider saddle connections on a translation surface in a hyperelliptic connected component of a stratum that do not intersect the interior of a distinguished saddle connection.  For this restricted set of saddle connections, we show that it satisfies an $L (\log L)^{d-2}$ growth rate, where $d$ is the complex dimension of the hyperelliptic stratum.  The upper bound holds for all translation surfaces in the hyperelliptic stratum while the lower bound holds for almost every surface in the hyperelliptic stratum.  The proof of the lower bound uses horocycle renormalization.
\end{abstract}

%Uncomment next line when we finish the paper, to clean up table of contents
\setcounter{tocdepth}{2}
\tableofcontents
%\red{Change TOC depth at end\ann{Don't forget!}}

\section{Introduction}

\subsection{Translation Surfaces and Strata}\label{sec:background}
%\paragraph*{\bf Translation surfaces, saddle connections, and counting} 
A \emph{translation surface} is a pair $(X, \omega)$ consisting of a compact Riemann surface  and a non-zero holomorphic $1$-form  $\omega$ on $X$ (also called an Abelian differential).  An Abelian differential on a Riemann surface of genus $g\geq 1$ has $2g-2$
zeroes counting multiplicity.
A translation surface $(X,\omega)$ induces a flat metric on $X$
with conical singularities corresponding to the zeros of $\omega$.
The area of a translation surface $(X,\omega)$ is given by
 $$\operatorname{Area}(X,\omega) = \frac{i}{2} \int_X \omega \wedge \bar \omega.$$

It is natural to study geodesics on a translation surface (with the induced metric). 
Outside the zero set, there are natural local 
charts such that $\omega=dz$, and the coordinate changes of those charts are translations of $\mathbb{R}^2$. As a result, 
any geodesic without a singularity in its interior is a straight line segment with respect to natural coordinates of $\omega$. A  geodesic joining a pair of zeros of $\omega$ (which need not be distinct) without singularities in its interior is called a \emph{saddle connection}.  A closed geodesic which does not pass through any singularities determines a maximal family of homotopic parallel closed trajectories that form a \emph{cylinder}.  The boundary of a cylinder is necessarily a union of saddle connections except for the case that $(X,\omega)$ is a flat torus without marked points.  Indeed, a flat torus with two marked points is a fundamental example in this work, and we extend the definition of saddle connection to include segments between the marked points. If each boundary component of a cylinder consists of a single saddle connection, then it is called a \emph{simple cylinder}.
For an introduction to this subject, we refer to the book \cite{AthreyaMasurTransSurfText}.

Let $\Omega \cM_g$ be the moduli space of translation surfaces on a closed surface of genus $g$.  
The space $\Omega \cM_g$ is naturally stratified by unordered integer partitions $\kappa=(k_1,\ldots, k_n)$ of $2g-2$, i.e., $\sum\limits_{i=1}^{n} k_i = 2g-2$.  
Let $\kappa$ be such a partition. The strata $\mathcal{H}(\kappa)$
consists of translation surfaces $(X,\omega)$ whose zeros have orders $k_1, \ldots, k_n$.
Each $\mathcal{H}(\kappa)$ has at most three connected components~\cite{KontsevichZorichConnComps}.

Let $\mathcal{H}(\kappa)$ be a stratum. 
Let $\cH_1(\kappa)$ be the hypersurface of  $\mathcal{H}(\kappa)$ given by  unit area translation surfaces.  It was proven in \cite{MasFinMeasErg, VeechFinMeasErg, MasurSmillieHausdDim} that each connected component of $\mathcal{H}_1(\kappa)$ carries a natural probability measure $\mu$ in the Lebesgue class, which is known as the \emph{Masur-Smillie-Veech measure}.  
Let us briefly describe this measure by introducing charts of $\mathcal{H}(\kappa)$ through the period map. 
For an arbitrary  $(X_0,\omega_0) \in \cH(\kappa)$  one can choose a triangulation of the underlying surface by saddle connections. There are $\gamma_1,\cdots,\gamma_d$ (where $d=2g+n-1$) edges of the triangulation that form a basis of the first relative homology group (first homology group of the surface relative to the zero set).
For a sufficiently small neighborhood $U\subset \cH(\kappa)$ of $(X_0,\omega_0)$,
there is the \emph{period map} defined by
\begin{eqnarray*}
\Theta:\quad  U\quad & \to & \mathbb{C}^{d} \\
 (X,\omega) &\mapsto& \left(\int_{\gamma_1} \omega,\cdots, \int_{\gamma_d} \omega \right).
\end{eqnarray*}
One pulls back the Euclidean measure on $\mathbb{C}^{d}$ locally with $\Theta$ to define a measure on $\cH(\kappa)$. For a measurable subset $V\subset \cH_1(\kappa)$, we define 
$\mu(V)$ to be the measure of the cone over $V$ consisting of translation surfaces with area at most one. It is known that the measure $\mu$ is finite \cite{MasFinMeasErg, VeechFinMeasErg}. Up to normalization, we can assume that $\mu$ is a probability measure on each connected component of 
$\mathcal{H}_1(\kappa)$.

There is a natural action on $\cH(\kappa)$ by $\operatorname{SL}(2, \mathbb{R})$ given by again locally representing each  surface via the period maps as points in $\mathbb{C}^{d}\cong \mathbb{R}^{2d}$ and then applying $\operatorname{SL}(2, \mathbb{R})$ as a linear action on each $\mathbb{R}^2$ factor.
The measure $\mu$ is invariant under this action.  Two important $1$-parameter subgroups of $\operatorname{SL}(2, \mathbb{R})$ are known as the \emph{Teichm\"uller flow}
$$g_t =\left( \begin{array}{cc} e^t & 0 \\ 0 & e^{-t} \end{array} \right),$$ 
and the \emph{horocycle flow}
$$u_s=\left( \begin{array}{cc} 1 & s \\ 0 & 1 \end{array} \right).$$ 
Both of these actions are volume preserving and ergodic. 
They play crucial roles in the study of translation surfaces including the present work because they serve as renormalization flows. There are many excellent surveys on this topic such as
\cite{ForniMatheusSurvey,MasTab,WrightSurvey,ZorichFlatSurfs}.
  
  \subsection{Counting Saddle Connections} 
  The study of counting problems for saddle connections on translation surfaces is very active, and it is connected to many different areas of mathematics, from low-dimensional dynamical systems to algebraic geometry.
  
  Let $(X,\omega)$ be a translation surface. 
   To each saddle connection $\gamma$, it is natural to associate its holonomy vector $z_\gamma\in \mathbb{C}$ given by integrating $\omega$ along $\gamma$; namely  $$z_{\gamma} = \int_{\gamma} \om \in \C.$$ 
   The \emph{length} of a saddle connection $\gamma$ is given by $|z_{\gamma}|.$ 
   This is exactly the geodesic length of $\gamma$ in the flat metric. 
  In what follows, we will write $|z_{\gamma}|$ simply as $|\gamma|$ when no confusion can arise.

  Denote by $\SC_{\omega}$ the set of saddle connections on $(X,\omega)$. 
  Let $$\La_{\om} = \{z_{\gamma}: \gamma \in \SC_{\om}\},$$ i.e., the set of holonomy vectors of saddle connections on $\om$. This is a countable discrete subset of the plane $\C$. 
  For $R>0$, let $\La_\om (R)$ be the collection of holonomy vectors in $\La_\om$ with length at most $R$.  

Motivated by the problems of counting periodic trajectories on billiards in rational polygons, \cite{MasurGrowthRateTrajsLowerBd,MasurGrowthRateTrajs} proved that the counting function $$N((X,\omega), R) = \#\La_{\om}(R) $$ has quadratic upper and lower bounds for all $\omega$.  In other words, there are constants $0 < c_1= c_1(\omega) < c_2 = c_2(\omega)$ so that for all $R$, 
\begin{equation}\label{eq:quadbounds} 
c_1 R^2  \le N((X,\omega), R) \le c_2 R^2 .
\end{equation}
Similar results are also obtained for the number of cylinders. 

There is a special class of translation surfaces known as \emph{Veech surfaces} or \emph{lattice surfaces}.  To each translation surface $(X,\omega)$, 
there is a \emph{Veech group} $\text{SL}(X, \omega) \subset \operatorname{SL}(2, \mathbb{R})$ generated by the derivatives of its affine
automorphisms. 
The Veech group $\operatorname{SL}(X, \omega) \subset \splin$ is a discrete subgroup of $\operatorname{SL}(2, \mathbb{R})$. For a typical surface $(X, \omega)$ in a stratum
of genus at least two,  
$\operatorname{SL}(X, \omega)$ is trivial.  However, if $\operatorname{SL}(X, \omega)$ forms a lattice subgroup of $\operatorname{SL}(2, \mathbb{R})$, then $(X, \omega)$ is called a \emph{Veech surface}.
It was proved in \cite{VeechTeichCurvEisen} that for a Veech surface $(X,\omega)$,
there exists some constant $c=c(X,\omega)>0$ such that 
\begin{equation} \label{eq:Veech}
\lim_{R \rightarrow \infty} \frac{N((X,\omega), R)}{R^2} = c.
\end{equation} 

%Quadratic asymptotic formula for the number of saddle connections is also true 
%for generic surfaces in each stratum. 
For a connected component $\mathcal{H}$ of a stratum $\cH_1(\kappa)$, \cite{VeechSiegelMeasures} showed there is a constant $c = c(\mathcal{H})>0$ such that for generic surfaces $(X, \omega) \in \cH$, $$\lim_{R \rightarrow \infty} \int_{\mathcal{H}} \left| \frac{N((X,\omega), R)}{R^2} - c\right|  \,d\mu= 0,$$ which is an $L^1$-quadratic asymptotic result. Inspired by Veech's approach, \cite{EskinMasurAsymptForms} adapted ideas from homogeneous dynamics (specifically, quantitative versions of Oppenheim's conjecture by \cite{EskinMargulisMozesQuantOppenheim}) and an ergodic theorem of \cite{nevo2017equidistribution} to improve this to a pointwise asymptotic result, showing that for $\mu$-almost every $(X,\omega) \in \mathcal{H}$, 
\begin{equation} \label{eq:EM01}
\lim_{R \rightarrow \infty} \frac{N((X,\omega), R)}{R^2} = c.
\end{equation} 
The constant $c=c(\mathcal{H})$ is called a Siegel-Veech constant and framework for computing these numbers explicitly was established in \cite{EskinMasurZorich}.

\subsection{The Main Theorems}
A translation surface $(X, \omega)$ is called \emph{hyperelliptic} if there is an hyperelliptic  involution $\tau:X\to X$ satisfying $\tau^{2}=\operatorname{Id}$ and $\tau^*(\omega)=-\omega$.  
A component of a stratum is {\em hyperelliptic} if every $(X,\omega)$ in the component
is a hyperelliptic translation surface.
There are two types of strata which have hyperelliptic connected components: $\cH(2g-2)$ and $\cH(g-1,g-1)$ \cite{KontsevichZorichConnComps}. Denote the hyperelliptic connected components by $\cH^{hyp}(2g-2)$ and $\cH^{hyp}(g-1,g-1)$.  
For instance, in the case of $g=2$ there are two hyperelliptic strata.   The first $\mathcal{H}(2)$  consists of $\omega$ with a single zero of order two and the second $\mathcal{H}(1,1)$ consisting of $\omega$  with a pair of simple zeros.

The present work focuses on the growth rate of saddle connections in the following setting.  We fix a hyperelliptic component   $\cH^{hyp}(\kappa)$ of a stratum and let $\cH^{hyp}_1(\kappa)$ be the subset of unit area translation surfaces. 
Let $(X, \omega)\in\cH_1^{hyp}(\kappa)$.  
Throughout, we let $\tau$ denote the hyperelliptic involution. 
Let $\beta$ be a saddle connection on $(X, \omega)$ that is invariant under $\tau$.  We say that two saddle connections are \emph{interiorly disjoint} if their interiors do not intersect. 
 We consider $(X, \omega) \setminus \beta$, where we implicitly excise only the interior of the saddle connection $\beta$.  In this way, saddle connections on $(X, \omega) \setminus \beta$ cannot cross $\beta$, and this is exactly the set of saddle connections of which we study growth rates.

\begin{theorem}
\label{MainSummaryThm}
Suppose $\beta$ is a  marked saddle connection on $(X, \omega)\in\cH_1^{hyp}(\kappa)$
satisfying $\tau(\beta)=\beta$.  Let $d = \dim_\bC \cH^{hyp}(\kappa)$.
Let $A((X, \omega) \setminus \beta,L)$ be the set of saddle connections interiorly disjoint from $\beta$ of length at most $L$.  
Then there exist $c$ and $c'$ depending only on the stratum such that 
\begin{enumerate}
\item[(1)] 	For  every $(X,\omega)\in \cH^{hyp}_1(\kappa)$, there exists  $L_0$ depending on
	 $(X, \omega)$  such that for $L\geq L_0$ 
$$ \# A((X, \omega) \setminus \beta,L) \leq c' \frac{L}{|\beta|} \left(\log \frac{L}{|\beta|}\right)^{d-2}. 
	$$ 
\item[(2)] For almost every $(X,\omega)\in \cH^{hyp}_1(\kappa)$, there exists  $L_0$ depending on $(X,\omega)$ and 
	 $\beta$  such that for $L\geq L_0$ 
\begin{equation*}
  \# A((X, \omega) \setminus \beta,L) \geq c \frac{L}{|\beta|} \left(\log \frac{L}{|\beta|}\right)^{d-2}.
\end{equation*}
\end{enumerate}
\end{theorem}

We also have the following result in the setting of non-invariant $\beta$.
Assuming Theorem \ref{MainSummaryThm}, 
the next theorem follows immediately from Lemma~\ref{TwoBdSC:Lemma}.  

%\ann{Add Theorem 1.2}

\begin{theorem}
\label{MainSummaryThm-noninvariant}
Suppose $\beta$ is a  marked saddle connection on $(X, \omega)\in\cH_1^{hyp}(\kappa)$
satisfying $\tau(\beta)\neq \beta$.  Let $A((X, \omega) \setminus (\beta\cup\tau(\beta)), L)$ be the set of saddle connections disjoint from $\beta\cup \tau(\beta)$ of length at most $L$. 
Then  for some exponent $d'<d$, the conclusions  (1) and  (2)  in
Theorem \ref{MainSummaryThm} hold for $A((X, \omega) \setminus (\beta\cup\tau(\beta)), L).$
\end{theorem}

Roughly speaking, Theorem \ref{MainSummaryThm} asserts that the growth rate
of saddle connections on every $(X,\omega)$ that are interiorly disjoint from a given invariant saddle connection $\beta$ is bounded above by $L \left(\log L\right)^{d-2}$, and for almost every $(X,\omega)$ the growth rate is  exactly $L \left(\log L\right)^{d-2}$ (the lower bound has the same order as the upper bound). 

Compared with the $L^2$ growth of all saddle connections, Theorem \ref{MainSummaryThm} also shows that the number of  saddle connections that do not  intersect the interior of $\beta$ is vanishingly small.   
That is, the ratio of those saddle connections that do not intersect $\beta$ to all saddle connections among those of length less than $L$ is bounded above by 
$$\frac{\left(\log L\right)^{d-2}}{L}.$$
Heuristically, this makes sense because the translation flow is ergodic in almost every direction and most long saddle connections will track an ergodic direction, which will necessarily intersect $\beta$.  In fact, \cite{DozierEquidisSC} proved  that the set of saddle connections becomes equidistributed on any translation surface.

We actually prove a stronger theorem which gives the upper bound in Theorem~\ref{MainSummaryThm} as a corollary. 
We first recall \cite[Lem.~6.5]{NguyenPanSu2020}, which will be explicitly restated in Lemma~\ref{TwoBdSC:Lemma} below: If $\rho$ is a saddle connection that is not invariant under the hyperelliptic involution $\tau$, then $\rho\cup\tau(\rho)$ divides the surface into two connected components.

\begin{definition}
Let $C>1$.  Given saddle connections $\alpha$ and $\beta$, we say the pair $\rho\cup\tau(\rho) $ \emph{$C$-separates $\alpha$ from $\beta$} if any path from $\alpha$ to $\beta$ crosses $\rho\cup\tau(\rho)$ and  $|\rho|<|\beta|/C$.  Alternatively, we will use the phrase \emph{$\alpha$ is $C$-separated from $\beta$ by $\rho$}.
\end{definition}

\begin{remark}
Following \cite{MasurGrowthRateTrajs}, we permit $\alpha$ to $C$-separate $\alpha$ from $\beta$.  Consequentially, if $\alpha$ is not $C$-separated from $\beta$, then there does \emph{not} exist $\rho$ such that $|\rho| < |\beta|/C$, and in particular, we necessarily have $|\alpha|\geq |\beta|/C$.
\end{remark}
\begin{remark}\label{rmk:sep:top}
	Suppose $\rho\neq \alpha$ and $\alpha $ is $C$-separated  from $\beta$ by $\rho\cup\tau(\rho)$, then it is necessary that $\alpha$ and $\beta$ are contained in  distinct components of $(X,\omega)\setminus (\rho\cup \tau(\rho))$. In particular, both $\alpha$ and $\beta$ are interiorly disjoint from $\rho\cup\tau(\rho)$.  
\end{remark} 

\begin{example}
One example immediate from the definition  is that every saddle connection that does not intersect the saddle connection systole (i.e. the shortest saddle connection on a translation surface) is not $C$-separated from the systole by any $\rho$.
\end{example}

The upper bound in Theorem~\ref{MainSummaryThm} is based on the following. 
Suppose  $\beta$ is a marked saddle connection on $(X, \omega)$ satisfying $\tau(\beta) = \beta$. 
For each $C>1$, let $A((X, \omega) \setminus \beta,L, C)$ be the set of saddle connections $\alpha\in A((X, \omega) \setminus \beta,L)$ such that $\alpha$ is not $C$-separated from $\beta$.

\begin{theorem}
\label{CountAlphaSC:Theorem}
For every  $C > 1$, there exist $c$ and $L_0$ depending on $C$ and the stratum $\cH^{hyp}(\kappa)$ such that for every $(X,\omega) \in \cH^{hyp}(\kappa)$ and  $L\geq L_0$, 
$$\# A((X, \omega) \setminus \beta,L, C) \leq c\frac{L}{|\beta|}\left(\log \frac{L}{|\beta|} \right)^{d-2}.$$
\end{theorem}

\medskip

As indicated above, in the hyperelliptic setting, there are two varieties of saddle connections - invariant and non-invariant - and both are natural for counting problems.  Moreover, cylinders are natural objects to count.  In fact, in the course of proving Theorem~\ref{MainSummaryThm}, we easily obtain growth rates for these other objects.  The following theorem summarizes Theorems~\ref{CountAlphaSC:Theorem3cyl} and \ref{thm:lower-general-cyls-noninv} below.

\begin{theorem}
\label{MainSummaryThm:Cyls:NonInvSC}
Suppose that $(X, \omega) \in \cH^{hyp}_1(\kappa)$ and $\beta$ is a marked saddle connection on $(X, \omega)$ satisfying $\tau(\beta) = \beta$. Let $A^{cyl}((X, \omega) \setminus \beta,L) $ and $A^{non-inv}((X, \omega) \setminus \beta, L)$ be the subsets of $A((X, \omega) \setminus \beta,L)$ consisting of cylinders and non-invariant saddle connections, respectively.  
Then for each of these two sets, the bounds in Theorem~\ref{MainSummaryThm} hold with the exponent $d-2$ replaced by $d-3$.
\end{theorem}

%\begin{remark}
%The assumption that $\beta$ is invariant is in fact not at all restrictive.   Lemma~\ref{TwoBdSC:Lemma}, says that when  $\beta$ is non-invariant we can cut along $\beta\cup\tau(\beta)$ and collapse the pair to build an invariant saddle connection on a hyperelliptic surface in a smaller dimensional stratum.  The details are left to the reader.
%\end{remark}

\paragraph{Homogeneous Dynamics and New Results} 
The stratum $\cH(0,0)$ consists of translation surfaces of genus one with two marked points. 
A translation surface with a saddle connection between the marked points can be viewed as a slit torus.  Indeed, our results even in that case are non-trivial, 
and we refer to Section~\ref{SlitTorusUB:subsect} and Section~\ref{LowerBound:Section} (see Remark \ref{rem:cylinder:genusone}) for details.

This is particularly interesting due to the close historical (and mathematical) relationship between homogeneous dynamics and dynamics on the moduli space of translation surfaces.  Dynamics on the moduli space $\cH(0)$ of tori with a marked point can be viewed from the homogeneous perspective as dynamics on $\operatorname{SL}(2,\mathbb{R}) /\operatorname{SL}(2,\mathbb{Z})$.
 Furthermore, the space of slit tori in $\cH(0,0)$ can still be viewed from the homogeneous space setting as the space $\operatorname{SL}(2,\mathbb{R}) \ltimes \bR^2 / \operatorname{SL}(2,\mathbb{Z}) \ltimes \bZ^2$. Therefore, our results in the case $\cH(0,0)$ represent new results in homogeneous dynamics as well. 
A natural question
is whether the results in $\mathcal{H}(0,0)$ generalize to a natural counting problem 
in homogeneous dynamics of higher-rank. 
 %This raises the open question:
%\begin{question}\ann{Not sure if we should pose this or if it's even a good question.}
%Do the results in $\cH(0,0)$ generalize to a natural counting problem in homogeneous dynamics?
%\end{question}

%IF WE DEFINE VEECH SURFACE ABOVE, DELETE Additionally, suppose $(X, \omega)$ is a Veech surface.  This means that the $SL(2,\mathbb{R})$ orbit of $(X,\omega)$ is closed.  Then  the almost every condition in the lower bound can be removed and our result holds for every Veech surface.  See the remark at the beginning of Section~\ref{sec:bad-measure}.

\paragraph{Veech Surfaces} 
Additionally, if $(X, \omega)$ is a Veech surface, then the almost every condition in the lower bound can be removed and our result holds for every Veech surface.  See Remark \ref{rem:veech} at the beginning of Section~\ref{sec:bad-measure}.

\paragraph{Motivation} 
There are several motivations for the present work.  Riemann surfaces with boundary are commonly studied alongside those without boundary.  From the hyperbolic perspective, \cite{MirzakhaniGrowthSCGeods} provides a count of simple closed curves in both settings demonstrating that one is not necessarily more difficult than the other.  
In the present work, we effectively consider a translation surface with a boundary component, corresponding to an excised saddle connection.  In the translation surface setting, the boundary changes the growth rate, and the situation is far more subtle than the setting without boundary.

Theorem~\ref{MainSummaryThm} can be used to study a variation of classical rational billiards.
We consider a rational billiard table such that one side of the table is marked.  We consider a billiard game where the ball is allowed to hit any side of the table \emph{except} for the marked side.  Then, as in the usual case, the rational billiard unfolds to a translation surface \cite{FK36}.  The marked side unfolds to a collection of line segments (possibly one)
through which the billiard is forbidden to pass.  Therefore, counting closed trajectories on a rational billiard table that avoid a marked side is equivalent to counting saddle connections or cylinders on a translation surface disjoint from a set of marked segments.  The present work considers exactly the case with a single segment on a hyperelliptic translation surface.  Moreover, there are billiard tables with a marked side that unfold to exactly the surfaces to which our theorem applies.

The other motivation for the present work is from the study of meromorphic Abelian (or infinite area quadratic) differentials on Riemann surfaces, which arise naturally in the boundary of the moduli space of translation and half-translation surfaces and more recently in mathematical physics \cite{GMNWallHitchinWKB, BridgelandSmithStabConds}. 
Indeed, there has been an increased interest in meromorphic differentials originating from mathematical physics due to their relation to Bridgeland-Smith stability conditions.  In those works, saddle connections on meromorphic quadratic differentials with higher order poles are a fundamental object of study.

The meromorphic setting can be transformed to a translation surface with marked segments by taking the convex core of the finite area portion of the surface \cite{HaidenKatzarkovKontsevichFlatSurfsStabStrcts}.  The marked segments in this setting are called \emph{slits} in \cite{AulicinoCBRank}, and $(X, \omega) \setminus \beta$ was called a \emph{slit translation surface}.  Again, we recall that $(X, \omega) \setminus \beta$ implicitly means we only exercise the interior of the saddle connection $\beta$.  With $\beta$ excised from the surface and the natural requirement that saddle connections be a subset of the surface, this notation implicitly forces us to restrict to saddle connections on the original surface that are interiorly disjoint from $\beta$. We use the terms slit and marked or distinguished saddle connection interchangeably going forward.  Finally, the setting was simplified to a translation surface with a single marked segment (slit) as in Theorem~\ref{MainSummaryThm} above.

In \cite{AulicinoCBRank}, the first named author considered the set of directions on a translation surface with a meromorphic differential and following \cite{GMNWallHitchinWKB, BridgelandSmithStabConds} showed that if the surface has infinite area, the set of directions is always a closed countable set in contrast to the holomorphic case where it is a countable dense subset of the circle of directions. 
Moreover, the Cantor-Bendixson rank (CB-rank) of this resulting set of saddle connection directions was considered in \cite{AulicinoCBRank} and an upper bound was given in terms of the dimension of the stratum containing the surface after forgetting the slit.  For hyperelliptic translation surfaces, the CB-rank and the complex dimension of the stratum align (see Appendix~\ref{HypCBRank:Appendix}).  However, work of \cite{TaharTopBdCBRk} indicates they do not align in general.  Consequentially, we believe that the quantity $d$ representing the complex dimension in Theorem~\ref{MainSummaryThm} should actually be replaced with the CB-rank in the general setting.

We also mention that in the setting of meromorphic differentials, the counting considered in the present work is complementary to that of \cite{TaharCountSCHighOrd}, which considers meromorphic differentials with a \emph{finite} number of saddle connections and proceeds to give a count for the number of such saddle connections.  Here we consider growth rates as the natural question in the setting of an infinite number of saddle connections.

\paragraph{Further discussion}
We can only prove the lower bound for $\mu$-almost every $(X,\omega)$ in  $\cH^{hyp}_1(\kappa)$, where $\mu$ denotes the Masur-Smillie-Veech measure. 
The problem of whether the lower bound holds for every $(X,\omega)$ in  $\cH^{hyp}_1(\kappa)$ or not remains open. Moreover, the proof of  Theorem~\ref{MainSummaryThm} relies on the assumption that $(X,\omega)$ is hyperelliptic.  It is possible that many ideas in this work can be applied to the non-hyperelliptic case, but the fundamental setting will need to be completely reworked to avoid the hyperelliptic assumption.

The question of the existence of exact asymptotics for
$\# A((X, \omega) \setminus \beta,L) $ in Theorem~\ref{MainSummaryThm}
 is dealt with in the case  $\cH(0,0)$ in \cite{AulicinoMasurSlitTorusAsymptotics}.  The argument there uses ideas from this paper; in particular the horocycle flow.  The question of asymptotics in general is open. 

\subsection*{Organization of the Paper} 
 Section~\ref{sec:sketch} is a short exposition of the heuristic method to prove  Theorem~\ref{MainSummaryThm}.
In Section~\ref{Background:Section}, we provide all of the necessary background on translation surfaces, including some elementary results, which are new to the best of our knowledge.  In Section~\ref{AuxSC:Section}, we prove a lemma that is fundamental to the upper bound results.  It organizes the collections of saddle connections we wish to count and produces sets of saddle connections that allow us to count the saddle connections in our original collection.  In Section~\ref{UpperBound:Section}, we prove the upper bound in its strongest form with the $C$-separated assumption, and 
the upper bound in Theorem~\ref{MainSummaryThm} is proved at the end of that section. 
Then Theorem \ref{MainSummaryThm-noninvariant} is proved by assuming Theorem~\ref{MainSummaryThm}.
In Section~\ref{LowerBound:Section}, we prove the lower bound in Theorem~\ref{MainSummaryThm}. 
 
 Finally, we conclude with Appendix~\ref{HypCBRank:Appendix}, which establishes the CB-rank of the set of saddle connection directions for all hyperelliptic translation surfaces.

\subsubsection*{Acknowledgements} The authors are grateful to Pat Hooper for discussions and suggestions during the course of this project.  They would also like to thank Jayadev Athreya, Jon Chaika, and Anton Zorich for helpful conversations.  Finally, they would like to thank the anonymous referee for their thoughtful feedback.

\section{A Heuristic}\label{sec:sketch}

Both of the Teichm\"uller flow and horocycle flow are used in this work.
The Teichm\"uller flow will be used to contract long vertical trajectories so that they can be studied as short trajectories on deformed surfaces, and the horocycle flow will be used to uncover saddle connections with desirable properties while fixing a saddle connection.  

There are several heuristics that explain the $d-2$ in the exponent in Theorem \ref{MainSummaryThm}.   The lower bound provides a quantitative version of the (qualitative) algorithm from work of the third and fourth named authors with Nguyen \cite{NguyenPanSu2020} for producing cylinders containing an invariant saddle connection.   We begin by applying the horocycle %$$\begin{matrix} 1,s 
 flow to find collections of  simple cylinders containing $\beta$ of a given length.  We excise each cylinder, identify the resulting boundaries to obtain a collection of translation surfaces each with a slit $\beta_1$ in a hyperelliptic stratum with dimension one less, and then find cylinders containing this new slit $\beta_1$.  Pulled back to the original surface under the horocycle flow, this gives families of saddle connections, again of a given length.   This can be done exactly $d-2$ times, i.e., one less than the number of parallelograms a hyperelliptic surface can be decomposed into.  (See Figure~\ref{fig:H4hypParallelogram} for a decomposition of a translation surface in $\cH^{hyp}(4)$ into $\dim_\bC \cH^{hyp}(4) - 1 = 5$ parallelograms.)

Then the count goes as follows.  As each cylinder with boundary  $\beta_1$ that contains $\beta$ winds around the surface, its length is roughly given by successive multiples of $\beta$.  For each sufficiently large $L_1$ and range of lengths between $L_1/2$ and $L_1$, we show that for almost all $(X,\omega)$, there are linear in  $L_1/|\beta|$ many such $\beta_1$ with lengths in that range.   Now let $L_2\geq L_1$ again be sufficiently large.  We show again that there are linear in $L_2/L_1$ cylinders $\beta_2$ containing the slit $\beta_1$ with lengths in the range between $L_2/c$ and $L_2$ for a constant $c$.   This gives  a total of $L_2/|\beta|$ many $\beta_2$  with the lengths of $\beta_1$ and $\beta_2$ in the given ranges.   
We continue this process until we reach the final cylinder $\beta_{d-2}$ at step $d-2$ of length $L_{d-2}$ satisfying $L/c\leq L_{d-2}\leq L$.
For a collection $\beta_1, \beta_2,\ldots, \beta_{d-2}$ with each $\beta_{i}$ having lengths in a given range, the number of $\beta_{d-2}$ is linear in $L/|\beta|$.
For each $1\leq i\leq d-3$, we partition the interval of lengths, which we call ranges, in the form $I_{i,j} = \left[\frac{L_i}{2^{j+1}}, \frac{L_i}{2^{j}}\right]$, so that $|\beta_{i}| \in I_{i,j}$ and $|\beta|\leq \frac{L_i}{2^{j}} \leq L$.  Then there are essentially $\log \frac{L}{|\beta|}$ such intervals in the partition.  Since there are $d-3$ of these $\beta_{i}$, the total number of cylinders $\beta_{d-2}$ will be of the order 
$$\frac{L}{|\beta|} \left(\log \frac{L}{|\beta|}\right)^{d-3}.$$
To count  all saddle connections,  we can take each cylinder $\beta_{d-2}$ and take  saddle connections that join opposite sides.  A similar analysis to the one above then raises the exponent of the logarithm to $d-2$.

This is all accomplished  through the use of horocycle renormalization.  The time parameter in the horocycle flow combined with the length parameter of the saddle connections under consideration provide an interplay between the two parameters that facilitate many elementary estimates that allow us to get a grasp on the estimates required e.g. Section~\ref{LowerBoundSetting:Section} and Lemma~\ref{lem:badtimes}.  

The upper bound techniques are now standard, such as Teichm\"uller renormalization and the focus on the concept of $C$-separated defined above.    Nevertheless, great care was taken to obtain exactly the desired power on the $\log L$ term, which was not required in previous works \cite{MasurGrowthRateTrajs, EskinMasurAsymptForms}.  
%We believe this further highlights the scope of their power.

\begin{figure}[htb]
\begin{center}
 \begin{tikzpicture}[scale=0.9]
		\draw (0,0) -- (-1.5,0) -- (-2.5,-3) -- (-1,-3) -- cycle;
	\node at (-1,-0.3) {$\beta$};
\draw[->, >= triangle 90] (0,0) -- (-0.8,0);
\draw[->, >= triangle 90] (-1,-3) -- (-1.8,-3);
\draw[->>] (-2.5,-3) -- (-2,-1.5);
\draw[->>] (-0,-4) -- (0.5,-2.5);
\node at (-1.7,-3.4) {$\beta$};
		\node at (-0.3,-1.6) {$\beta_1$};
        \node at (1.2,-2.6) {$\tau(\beta_1)$};
          \node at (0.8,-0.3) {$\beta_2$};
           \node at (-1,-3.8) {$\tau(\beta_2)$};
           \node at (-2.5,-1.2) {$\tau(\beta_1)$};
            \node at (3, 0.8) {$\tau(\beta_2)$};
		\draw (0,0) -- (-1,-3) -- (0,-4) -- (1,-1) -- cycle;
\draw[-latex] (-1,-3) -- (-0.5,-3.5);
\draw[-latex] (2,1) -- (2.5,0.5);
      \draw (0,0) -- (1,-1) -- (3,0) -- (2,1) -- cycle;
      \draw[-latex] (2,1) -- (2.5,0.5);
            \draw[->|] (1,-1) -- (2,-0.5);
            \draw[->|] (-0.5,2) -- (0.5,2.5);
      \draw (0,0) -- (2,1) -- (1.5,3) -- (-0.5,2) -- cycle;
       \draw[->] (2,1) -- (1.75,2);
       \draw[->] (-2,2) -- (-2.25,3);
      \draw (0,0) -- (-0.5,2) -- (-2.5, 4) -- (-2,2) -- cycle;
       \draw[-{> >[sep] >}] (0,0) -- (-1,1);
        \draw[-{> >[sep] >}] (-0.5, 2) -- (-1.5,3);
	\end{tikzpicture}
\end{center}
	\caption{A decomposition of a translation surface in $\cH^{hyp}(4)$ into five parallelograms cf. Lemma~\ref{Parallelogramulation:Lemma}.}
\label{fig:H4hypParallelogram}
\end{figure}
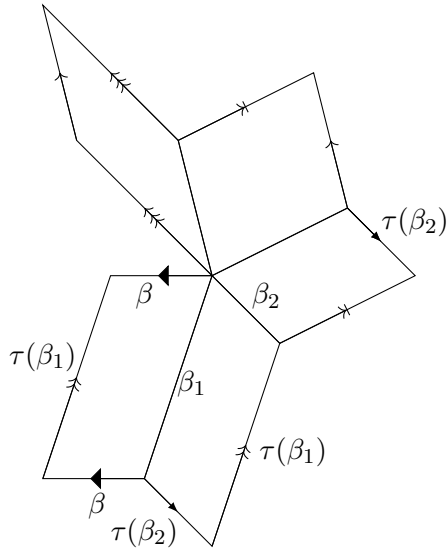

\section{Preliminaries on Translation Surfaces}
\label{Background:Section}

In this section, we establish some needed results about translation surfaces in general after recalling some basic definitions.  There are many references and surveys on translation surfaces and counting problems.  Most recently, \cite{AthreyaMasurTransSurfText} contains careful background and an extensive bibliography.   We prove Lemma~\ref{lem:smallangle}, which is an elementary result about a translation surface, but we were unable to find in the literature.  In Section~\ref{HypTransSurf:Section}, we recall and prove several results concerning hyperelliptic translation surfaces.  Indeed, the present work will focus exclusively on hyperelliptic translation surfaces, so these results are essential for everything that follows.

\subsection{The Flat Geometry of Translation Surfaces}

%A \emph{translation surface} is a Riemann surface $X$ carrying a non-zero Abelian differential $\omega$, and will be denoted $(X, \omega)$.  Translation surfaces are equipped with natural flat metrics with conical singularities.  Any geodesic without a singularity in its interior is a straight line with respect to natural coordinates of $\omega$. A  geodesic joining a pair of singularities of $\omega$ (which need not be distinct) without singularities in its interior is called a \emph{saddle connection}.  A closed geodesic which does not pass through any singularities determines a maximal family of homotopic parallel closed trajectories that form a \emph{cylinder}.  The boundary of a cylinder is necessarily a union of saddle connections.  If each boundary of a cylinder consists of a single saddle connection, then it is called a \emph{simple cylinder}.

%An Abelian differential induces a natural flat length on a translation surface given by integrating along $|\omega|$, and a natural area form given by $\omega \wedge \bar \omega$.  Given a saddle connection $\gamma$, we will write $|\gamma|$ for the flat length of $\gamma$ and $\text{Area}(X, \omega)$ for the area of $(X, \omega)$.  Furthermore, to each saddle connection $\gamma$, it is natural to associate its holonomy vector $\text{hol}(\gamma) \in \bR^2$ given by integrating $\omega$ along $\gamma$ and splitting it into its real and imaginary parts. 

Let $(X,\omega)$ be a translation surface. Recall that  $z_\gamma$ is the holonomy vector given by integrating $\omega$ over a saddle connection $\gamma$.  
The length of a saddle connection $\gamma$ is given by $|\gamma|=|z_\gamma|$.
We can write $z_\gamma=h_\gamma + i v_\gamma$, where $h_\gamma$ and $v_\gamma$
are called the horizontal component and vertical component of $\gamma$, respectively.

\begin{definition}[Cross-product]
The \emph{cross product of saddle connections} $\gamma$ and $\gamma'$, written as $|\gamma\times \gamma'|$, is defined to be  $|z_\gamma\wedge z_{\gamma'}|$,
where $z_\gamma\wedge z_{\gamma'}$ denotes the wedge product of the holonomy vectors
$z_\gamma$ and $z_{\gamma'}$.  This is the area of  the parallelogram in the complex plane spanned by the vectors $z_\gamma$ and $z_{\gamma'}$.
\end{definition}
If two saddle connections $\gamma$ and $\gamma'$ bound an angle $\theta$, then 
$\sin\theta=\frac{ |\gamma\times \gamma'|}{|\gamma||\gamma'|}$.
In this paper, we will think of an angle between two saddle connections that is close to $\pi$ as being close to zero and we will consider the angle in terms of its sine.
In particular, two saddle connections are parallel if and only if $|\gamma\times \gamma'|=0$.

\begin{definition}[Diameter]
The \emph{diameter} of a translation surface is the maximum flat distance between two points on the surface.
\end{definition}

A notational concept that will be useful in what follows is the logarithmic scale of a saddle connection.  Indeed, it will be more useful to use the log of the length rather than the length itself.

\begin{definition}[Logarithmic scale]
Let $(X, \omega)$ be a translation surface with saddle connection $\gamma$.  We say that $\gamma$ has \emph{logarithmic scale} $j\in \mathbb{Z}$ if $\gamma$ satisfies
$$2^j \leq |\gamma| < 2^{j+1}.$$
If $\Gamma$ is a collection of saddle connections, we say that \emph{$\Gamma$ is a collection of saddle connections of logarithmic scale $j$} if each of the saddle connections in $\Gamma$ has logarithmic scale $j$.
\end{definition}

This work focuses on the problem of counting saddle connections up to a given length
(or logarithmic scale).  To accomplish this, we will consider collections of saddle connections and the techniques we use to count them will depend on various properties they obey.  We define all of the necessary properties here.

\begin{definition}[$\epsilon$-isolated]
Fix a small $\epsilon > 0$.  Let $(X,\omega)$ be a translation surface of area $A$.  
Two saddle connections $\gamma$ and $\gamma'$ on $(X,\omega)$ are \emph{$\epsilon$-isolated} if the angle between them is at least 
$$\frac{\epsilon^2A}{|\gamma||\gamma'|}.$$

Let $\Gamma$ be a collection of saddle connections containing at least two non-parallel saddle connections.  If every pair of non-parallel saddle connections in $\Gamma$ are $\epsilon$-isolated, then we say that \emph{$\Gamma$ is $\varepsilon$-isolated}.
\end{definition}

\begin{convention}
Let $\Gamma$ be a collection of saddle connections on $(X, \omega)$.  Any time we consider $\Gamma$, we assume it has the property that for all $\gamma \in \Gamma$, any $\gamma'$ that is parallel to $\gamma$ is also contained in $\Gamma$.
\end{convention}

In the contrary case, 
if a pair of saddle connections $\gamma$ and $\gamma'$ are not $\epsilon$-isolated,
then we can apply the Teichm\"uller flow to make them simultaneously short.  
Such a technique is called \emph{Teichm\"uller renormalization}. It is well-known to experts and appears in various forms in the literature, e.g., \cite{MasurGrowthRateTrajs}.  We state a version that will be used below.  Since it will be used hereafter, we include a proof.

\begin{lemma}\label{lem:Teich}
 Let $\gamma$ and $\gamma'$ be a pair of saddle connections on $(X, \omega)$ that is not $\epsilon$-isolated. Then we can use the Teichm\"uller flow to deform $(X, \omega)$ such that on the new surface both the length of $\gamma$ and $\gamma'$ are less than $2 \sqrt{A}\epsilon$, where $A$ denotes the area of $(X, \omega)$.
\end{lemma}
\begin{proof}
Without loss of generality, we may assume that $|\gamma|\geq |\gamma'|$. 
  We can also rotate such that $\gamma$ is vertical. Apply the Teichm\"uller flow $g_t$ to $(X, \omega)$, for $$e^t=\frac{|\gamma|}{\sqrt{A}\epsilon}.$$ 
  Denote the new surface by $(X_t,\omega_t)$. Since the Teichm\"uller flow contracts the vertical direction by $e^{-t}$, the length of $\gamma$ is at most $\sqrt{A}\epsilon$ on $(X_t,\omega_t)$. 
  
  For $\gamma'$, since the angle between $\gamma'$ and $\gamma$ is at most $\frac{\epsilon^2 A}{|\gamma||\gamma'|}$,
 $|h_{\gamma'}|$ 
  is at most $\frac{\epsilon^2 A}{|\gamma|}$. 
 The  Teichm\"uller flow $g_t$ expands $|h_{\gamma'}|$ by $e^t$ and contracts $|v_{\gamma'}|$ by $e^{-t}$. Thus on the new surface $(X_t,\omega_t)$ we have 
 $|\gamma'|$ bounded above by
 $$e^t \frac{\epsilon^2 A}{|\gamma|} + e^{-t}|\gamma'|,$$
  which is less than $2\sqrt{A}\epsilon$ because $e^t=\frac{|\gamma|}{\sqrt{A}\epsilon}$
  and $\frac{|\gamma'|}{|\gamma|} \leq 1$. 
\end{proof}

We recall a result from \cite{KMS} and \cite[Cor.~6.3]{EskinMasurAsymptForms},
which is used to construct a ``complex generated by a collection of short saddle connections''. By definition, a \emph{complex} is a subsurface
whose boundary is a union of interiorly disjoint saddle connections.  Given a collection of short saddle connections $\Gamma$, we say the \emph{complex is generated by $\Gamma$} if it satisfies the conclusion of the following lemma.

\begin{lemma}[Eskin-Masur]
\label{EMEpsCplx:Lemma}
There are constants $c_1$ and $c_2$ (depending on the stratum) such that for $\epsilon$ sufficiently small,  given a collection of saddle connections $\Gamma$  on $(X,\omega)$, each of length at most $\epsilon$ such that each $\gamma\in \Gamma$ interiorly intersects some other $\gamma'\in\Gamma$, there is a closed subsurface $Y$ generated by $\Gamma $ called an $\epsilon$-complex with the following properties:
\begin{enumerate}
  \item[(i)] $Y$ contains each $\gamma\in\Gamma$;
  \item[(ii)] $Y$ is triangulated by saddle connections of length at most $c_2\epsilon$ and $Y$ has area at most $c_1\epsilon^2$;
  \item[(iii)] $Y$ is  minimal in the sense that any closed curve in the interior of $Y$ intersects some $\gamma\in\Gamma$.
\end{enumerate}
\end{lemma} 

\begin{definition}
A collection $\Gamma$ of two or more saddle connections on a translation surface $(X, \omega)$ \emph{fill a subsurface} $Y$ if any closed curve $\eta \subset Y$ that is not isotopic to a boundary component of $Y$ intersects some element of $\Gamma$.
\end{definition}

\begin{remark}
Part of the proof of the above lemma includes the statement that a pair of interiorly intersecting saddle connections $\alpha_1$ and $\alpha_2$ \emph{fill a locally convex surface} with diameter at most $|\alpha_1|+|\alpha_2|$.
\end{remark}
We conclude with a fundamental lemma about triangles on a translation surface.  Throughout, a \emph{triangle} will always mean an \emph{embedded} Euclidean triangle in a surface.  Vertices of the triangle are singularities of the flat metric, and edges are saddle connections.  While the following result is elementary, it is essential to what follows below (see the proof of Theorem~\ref{CountTriangles}).

\begin{lemma}
\label{lem:smallangle}
Let $(X,\omega)$ be a translation surface with a horizontal slit $\beta$.  Let $\Delta$ and $\Delta'$ be two distinct triangles on $(X,\omega)$, each of which contain $\beta$ as a side.  Suppose further that $\Delta$ and $\Delta'$ are both above $\beta$. Let $\gamma$ (resp. $\gamma'$) be the longer of the two sides of $\Delta$ (resp. $\Delta'$) excluding $\beta$.  Finally, assume that $\gamma$ and $\gamma'$ interiorly intersect and that  $|\gamma'|/2\leq |\gamma|\leq 2|\gamma'|$.  Then we have
 $$|\gamma\times \beta|\leq 2|\gamma\times\gamma'|. $$ 
 \end{lemma}   
 \begin{proof}
 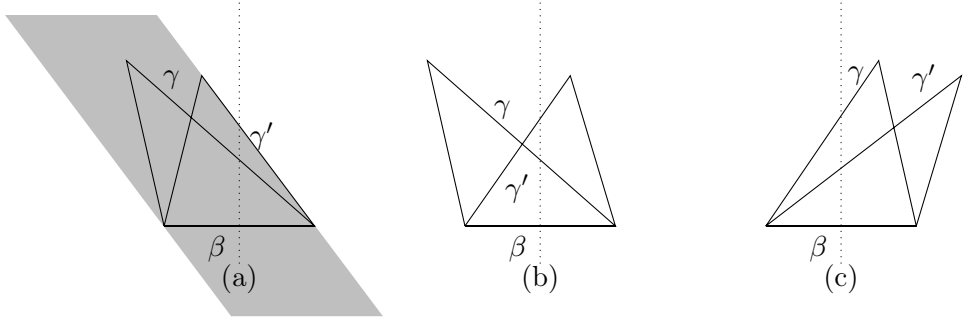
\begin{figure}
 	\begin{tikzpicture}
 	\draw[fill, lightgray](2.9-4,-1.2) -- (2-4,0)-- (-0.1-4,2.8)--(-2.1-4,2.8)-- (0.9-4,-1.2);
 \draw[fill=yellow] (-0.5-4,2.2)--(2-4,0) --(-4,0)--(-0.5-4,2.2); 
 		\draw(0-4,0) -- (2-4,0)-- (0.5-4,2)--(0-4,0) (0-4,0)--(2-4,0)--(-0.5-4,2.2)--(0-4,0); 
 		\draw[dotted](1-4,-0.5)--(1-4,3);
 		\draw(0.7-4,0)node[below]{$\beta$} 
 		(0.1-4,1.7)node[above]{$\gamma$}(1.3-4,1.2)node{$\gamma'$} (1-4,-0.7)node{(a)};
 		
  \draw[fill=yellow] (-0.5,2.2) -- (2,0) -- (0,0) --  (-0.5,2.2); 
 		\draw (0,0) -- (2,0)-- (1.4,2)--(0,0) (0,0)--(2,0)--(-0.5,2.2)--(0,0); 
 		\draw[dotted](1,-0.5)--(1,3);
 		 		\draw(0.7,0)node[below]{$\beta$} 
 		(0.5,1.3)node[above]{$\gamma$}(0.4,0.5)node[right]{$\gamma'$} (1,-0.7)node{(b)};

         \draw[fill=yellow] (0+4,0) -- (2+4,0)-- (-0.5+4+2,2.2)--(0+4,0); 
 		\draw(0+4,0) -- (2+4,0)-- (0.6+4+2,2)--(0+4,0)(0+4,0)--(2+4,0)--(-0.5+4+2,2.2)--(0+4,0); 
 		\draw[dotted](1+4,-0.5)--(1+4,3);
 		\draw(0.7+4,0)node[below]{$\beta$} 
 		(1.2+4,1.75)node[above]{$\gamma$}(2.1+4,1.65)node[above]{$\gamma'$} (1+4,-0.7)node{(c)};
\end{tikzpicture}
\caption{Relative positions of triangles determined by $\gamma$ and $\gamma'$.}
\label{Fig:small:triangle}
\end{figure}
 	
Without loss of generality, we may assume that
the third vertex of $\Delta'$ that is above $\beta$ sits to the right of the third vertex of $\Delta$ (see Figure \ref{Fig:small:triangle}). Consider the position of the third vertex of $\Delta$ relative to the perpendicular bisector of $\beta$. There are three possibilities as indicated in Figure~\ref{Fig:small:triangle} (a)-(c),
where the triangle $\Delta$ is filled with yellow. After reflecting $\gamma$ and $\gamma'$, we see that Case~(c) is equivalent to Case~(a).  We now consider Case~(a) and Case~(b).   
For Case~(b), we claim that $$|\gamma\times\beta|\leq |\gamma\times\gamma'|.$$  Indeed, consider the portion of $\gamma'$ interior to the triangle bounded by $\gamma$ and $\beta$.  Then the cross product of $\gamma$ with the vector corresponding to that portion of $\gamma'$ is equal to $|\gamma \times \beta|$ because they bound triangles of equal area.  However, $\gamma'$ is longer than the portion interior to the triangle because it ends at a singularity and $\Delta$ does not contain a singularity in its interior, so the inequality follows.
 
 For Case~(a), if $|\gamma'\times\beta|\leq |\gamma\times\gamma'|$, since the interior angle of $\Delta$ bounded by  $\beta$ and $\gamma$ is less than the interior angle of $\Delta'$ bounded by $\gamma'$ and $\beta$, and $|\gamma'|/2\leq |\gamma|\leq 2|\gamma'|$, then 
 $$|\gamma\times \beta|\leq 2|\gamma'\times\beta|\leq 2|\gamma'\times\gamma|.$$
 This proves the lemma for this case. 

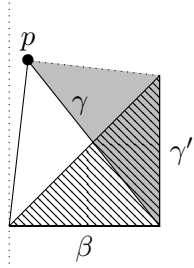
\begin{figure}
\centering
 	\begin{tikzpicture}
  \draw[fill=yellow](0,0) -- (0.25,2.2)--(2,0) -- (0,0); 
	 	\draw[fill, lightgray](2,0)--(0.25,2.2)--(2,2);
 		\draw[fill, pattern={north west lines}](0,0) -- (2,0)-- (2,2); 		
 		\draw(0,0) -- (2,0)-- (2,2)--(0,0) (0,0)--(2,0)--(0.25,2.2)--(0,0); 
 		\draw[dotted](0,-0.5)--(0,3);
 		\draw[dotted](2,2)--(0.25,2.2);
		\draw(0.25,2.2)node{\textbullet}(0.25,2.2)node[above]{$p$};
 		\draw(1,0)node[below]{$\beta$} 
 		(0.7,1.6)node[right]{$\gamma$}(2,1)node[right]{$\gamma'$};
	\end{tikzpicture}
\caption{Case (a) with $\gamma'$ sheared to a vertical saddle connection.}
\label{Fig:triangle:area:arg}
\end{figure}

Next we claim that it is impossible for $|\gamma'\times\beta| > |\gamma\times\gamma'|$ and proceed by contradiction to prove it.  Figure~\ref{Fig:triangle:area:arg} shows that if $|\gamma'\times\beta| > |\gamma\times\gamma'|$, then the point $p$ has to lie to the right of the dashed line in Figure~\ref{Fig:triangle:area:arg}.  Indeed, after shearing $\gamma'$ to vertical, $\beta$ becomes the altitude of the striped triangle and bounds the altitude of the shaded triangle, where both triangles share base $\gamma'$.  Therefore, the cross product condition bounds the area of the triangle with side $\gamma$.  In Figure~\ref{Fig:small:triangle}(a), the triangle with side $\gamma$ corresponds to the triangle $\Delta$ determined by $ \gamma$ and $\beta$ and this proves it is contained in the bi-infinite strip that contains $\beta$ and whose boundary sides are parallel to $\gamma'$.  

Recall that, by assumption,   the saddle connections $\gamma$ and $\gamma'$ interiorly intersect. Therefore, if we focus on the intersection of $\gamma$ and $\gamma'$ and travel along $\gamma'$, which necessarily lies in the interior of $\Delta$, which does not contain any singularities because it is an embedded triangle, we conclude that $\gamma'$ must exit $\Delta$.  However, since the slope of the third side of $\Delta$ is bigger than the slope of $\gamma'$ because it lies in the strip bounded by $\gamma'$, $\gamma'$ has to intersect the interior of $\beta$, which is a contradiction.  Thus, this case does not occur.
\end{proof}

%\subsection{Moduli Space of Translation Surfaces}
%\label{ModSpTransSurfs:Section}

%DELETE

\subsection{Hyperelliptic Translation Surfaces}
\label{HypTransSurf:Section}

In this section we collect all of the necessary results concerning the structure of hyperelliptic translation surfaces that will be used below.

Let $\kappa = (2g-2)$ or $(g-1,g-1)$.  We denote a hyperelliptic component by  $\cH^{hyp}(\kappa)$. Given any $(X,\omega)\in \cH^{hyp}(\kappa)$, we denote the hyperelliptic involution by $\tau$. 
In the case of $\cH^{hyp}(g-1,g-1)$, the map $\tau$ interchanges the zeroes.  
The existence of the hyperelliptic involution imposes very strong restrictions on the structure of the translation surface.  We recall some properties here and prove others that will be essential to our analysis below.  

With respect to the flat structure, $\tau$ can be viewed as rotation by angle $\pi$.  Consequently, the hyperelliptic property is invariant under the action by $\text{GL}(2,\bR)$.  Another key consequence of this rotation, that will be used throughout this work is that given an embedded triangle with a side $\beta$ that is invariant under $\tau$, we will apply the hyperelliptic involution to this triangle to obtain a parallelogram containing $\beta$ as a diagonal.
In particular, when a pair of opposite sides of the parallelogram are the same saddle connection (which is invariant under $\tau$), then we obtain a simple cylinder that contains $\beta$.

Since hyperelliptic strata will be the focus of this work, we observe that the hyperelliptic involution splits saddle connections into two types.  Let $\sigma$ be a saddle connection.  Then either $\sigma$ is an \emph{invariant saddle connection} or otherwise, it is a \emph{non-invariant saddle connection} with respect to $\tau$.

The following lemma \cite[Lem.~6.5]{NguyenPanSu2020} will be fundamental to this paper.
\begin{lemma}[Nguyen-Pan-Su]
\label{TwoBdSC:Lemma}
Let $(X, \omega)$ be a hyperelliptic translation surface.  If $\sigma$ is a non-invariant saddle connection, then $(X, \omega) \setminus \{\sigma, \tau(\sigma)\}$ is disconnected.
Furthermore, if we cut along $\{\sigma, \tau(\sigma)\}$ to obtain two surfaces and glue them along their respective boundaries to obtain two closed surfaces, then they each lie in hyperelliptic strata with dimensions strictly less than the one containing $(X, \omega)$, and $\sigma$ and $\tau(\sigma)$ are identified to form an invariant saddle connection.
\end{lemma}

\begin{remark}
It is immediate from Lemma~\ref{TwoBdSC:Lemma} that the lengths of saddle connections that are interiorly disjoint from $\sigma\cup \tau(\sigma)$ and the angles between two such saddle connections do not change under this operation.
\end{remark}

We now observe a fundamental decomposition for translation surfaces in hyperelliptic strata.  In $\cH^{hyp}(4)$, it appears as \cite[Thm.~1.1]{NguyenParallelogramH4}.

\begin{lemma}
\label{Parallelogramulation:Lemma}
Let $(X, \omega) \in \cH^{hyp}(\kappa)$ be a hyperelliptic translation surface with hyperelliptic involution $\tau$, and let $d = \dim_\bC \cH^{hyp}(\kappa)$.  Then $(X, \omega)$ admits a decomposition into $d-1$ parallelograms such that the vertices of each parallelogram lie at the singularities of $(X, \omega)$ and each parallelogram is invariant under $\tau$.
\end{lemma}

\begin{proof}
In fact, we will prove a stronger statement and proceed by induction to do so.  Let $\sigma$ be a saddle connection on $(X, \omega)$ that is invariant under $\tau$.  We will prove that there is a decomposition of $(X, \omega)$ into parallelograms such that $\sigma$ is internal to one of the parallelograms.  For the base case, if $(X, \omega)$ is a once punctured torus with a distinguished saddle connection $\sigma$ (from the puncture to itself), then it can be realized as a single parallelogram with $\sigma$ as one side.  By cutting along a diagonal of this parallelogram and gluing $\sigma$ to its copy on the other side of the parallelogram, we see that the claim holds.

For the induction assumption, assume that the statement holds for all strata of dimension less than that of $\cH^{hyp}(\kappa)$.  Let $(X, \omega)$ have a distinguished saddle connection $\sigma$.  By \cite[Prop.~6.1]{NguyenPanSu2020}, there is a simple cylinder $C$ on $(X, \omega)$ containing $\sigma$ in its interior.  Then $C$ is invariant under the hyperelliptic involution by \cite[Lem.~2.1]{LindseyInvCompsHyp}.  Therefore, it contains a boundary saddle connection $\sigma'$ that is not invariant under the hyperelliptic involution.  By Lemma~\ref{TwoBdSC:Lemma}, after cutting along $\sigma' \cup \tau(\sigma')$, we get $C$ and a hyperelliptic translation surface $(X', \omega')$, with a distinguished saddle connection $\sigma'$ by abuse of notation.  Then $\sigma'$ is invariant under the hyperelliptic involution because it was glued from the two boundary components that were non-invariant on $(X, \omega)$, and $(X', \omega')$ lies in a stratum with dimension strictly less than that of $\cH^{hyp}(\kappa)$.  Therefore, we apply the inductive assumption to $(X', \omega')$ with distinguished invariant saddle connection $\sigma'$, observe that $C$ is a parallelogram with $\sigma$ interior to it, and by reversing the surgery, obtain a decomposition of $(X, \omega)$ into parallelograms such that $\sigma$ is interior to one of them.

Finally, we note that a parallelogram has two complex dimensions.  If we identify two parallelograms, the second parallelogram adds one complex dimension.  Hence, every surface can be decomposed into at most $d-1$ parallelograms.
\end{proof}

\begin{lemma}
\label{TauInvSubsurfBd:Lemma}
Let $(Y, \eta)$ be a subsurface of $(X, \omega)$ that is invariant under $\tau$.  Then there are saddle connections $\sigma_i \subset (X, \omega)$ such that $\sigma_i \not= \tau(\sigma_i)$ and $\partial Y = \bigcup\limits_i \left( \sigma_i \cup \tau(\sigma_i)\right)$.  In particular, the boundary of $Y$ consists of an even number of saddle connections.   
\end{lemma}

\begin{proof}
Let $\sigma$ be a saddle connection in the boundary of $(Y, \eta)$.  Then $\sigma$ cannot be invariant under $\tau$ because $(Y, \eta)$ is closed, and so $\sigma$ would lie in the interior of $(Y, \eta)$.  Therefore, the invariance of $(Y, \eta)$ under $\tau$ implies that $\tau(\sigma)$ is a boundary saddle connection as well.  Since this holds for all boundary saddle connections, the lemma follows.
\end{proof}

Both the concept of $C$-separated defined in the Introduction and the classical concept of a separating curve occur in this work.  Since a separating curve is a purely topological concept while $C$-separation relies on the geometric flat structure on the surface, we write the admittedly redundant term \emph{topologically separating} for a separating curve to emphasize for the reader which concept is being used at which point.  In other words, let $\gamma,\rho$ and $\sigma$ be saddle connections, where $\sigma\neq \tau(\sigma)$. 
We say that $\sigma \cup\tau(\sigma)$
\emph{topologically separates} $\gamma$ from $\rho$ if any path connecting $\gamma$
and $\rho$ crosses $\sigma \cup\tau(\sigma)$. 
As a consequence, $\gamma$ and $\rho$ lie in distinct connected components
of $(X, \omega) \setminus \{\sigma, \tau(\sigma)\}$.

\begin{lemma}\label{lem:triangulation}
Let $(X,\omega) \in \cH^{hyp}(\kappa)$. Let $N \geq 0$ and $\rho_0,\rho_1, \ldots , \rho_N$ be a finite collection of pairwise interiorly disjoint $\tau$-invariant saddle connections on $(X,\omega)$.  Let $D$ be the diameter of $(X,\omega)$.  Then  there exists a constant $c$ depending on the stratum such that for any saddle connection $\gamma$ that is interiorly disjoint from all $\rho_i$, there is a pair $(\sigma,\tau(\sigma))$ of saddle connections with $\sigma \not= \tau(\sigma)$ such that
\begin{enumerate}[label=\arabic*)]
    \item \label{lem:triangulation:IntDisj} $\sigma$ is interiorly disjoint
    from  all $\rho_i$. 
	\item \label{lem:triangulation:triangleBd} Either the triple  $\{\gamma,\rho_0,\sigma\}$ bounds a triangle, or $\sigma\cup\tau(\sigma)$ topologically separates $\gamma$ from  $\rho_0$ and $\{\gamma,\sigma\}$ are two sides of a triangle.
	\item \label{lem:triangulation:sigmaBd} 
	$|\sigma|\leq c\cdot  \left(D+|\gamma|+\sum\limits_{i=0}^N |\rho_i|\right)$, and
 \item \label{lem:triangulation:crossprod} $|\sigma\times\gamma|\leq 2A$, where $A$ is the area of the subsurface bounded by $\sigma\cup\tau(\sigma)$ that contains  $\gamma$.
\end{enumerate}
\end{lemma}

\begin{proof}
We begin by constructing a $\tau$-invariant triangulation such that every saddle connection in it satisfies Item~\ref{lem:triangulation:sigmaBd}.  To ease notation, define $\brho = \cup_{i=0}^N \rho_i$ and $|\brho| = \sum_{i=0}^N |\rho_i|$.  In fact,  the diameter of each component of $(X,\omega)\setminus \left\{ \gamma\cup\tau(\gamma)\cup\brho\right\}$ is at most $4|\gamma|+4|\brho|+D$, which is true regardless of whether $\alpha$ is invariant or not under the involution.
We claim we can find a saddle connection $\gamma_1$ on $(X,\omega)\backslash \{\gamma\cup\tau(\gamma)\cup\brho\}$ such that 
\begin{equation*}
	|\gamma_1|\leq 2|\gamma|+2|\brho|+D.
\end{equation*}
The claim follows directly for simply connected components of $(X,\omega)\setminus \left\{ \gamma\cup\tau(\gamma)\cup\brho\right\}$. For a non-simply connected component $Y$, consider a shortest non-contractible geodesic arc $\gamma_1$ in that component.
 Of course, this geodesic has length at most the diameter of $Y$. Either it contains an interior saddle connection, or it is parallel to a maximal family of geodesic arcs realizing that minimal distance.
In the later case, the union of this family of geodesic arcs is a rectangle.  Consider the pair of boundary sides of this rectangle that are parallel to $\gamma_1$. Since $Y$ is not simply connected, at least one side of this pair is contained in the interior of $Y$, giving an interior saddle connection of length at most the diameter, while the other side consists of $\gamma$ or a subset of $\brho$ and their involutions.  In either case, we found an interior saddle connection of length at most the diameter of $Y$. 
 
Since $(X,\omega)\setminus \left\{ \gamma\cup\tau(\gamma)\cup\brho\right\}$ is $\tau$ invariant, by Lemma~\ref{TwoBdSC:Lemma}, it also contains $\tau(\gamma)$.  Similarly, there exists a saddle connection $\gamma_2$ on $(X,\omega)\setminus \left\{ \gamma\cup\tau(\gamma)\cup\brho\cup\gamma_1\cup\tau(\gamma_1)\right\}$ such that
	\begin{eqnarray*}
	  |\gamma_2| &\leq& 2|\gamma_1|+2|\gamma|+2|\brho|+D \\
	   &\leq& 6|\gamma|+6|\brho|+3D.
	\end{eqnarray*}
Inductively, this gives a triangulation of $(X,\omega)$.  Since the number of triangles in the triangulation is bounded in terms of the genus, as we iterate this procedure, we observe that the constant $c$ in Item~\ref{lem:triangulation:sigmaBd} is bounded in terms of the genus as well.  

Next, we use this triangulation to show that Item~\ref{lem:triangulation:triangleBd} is satisfied.  Assume first that $\gamma$ is invariant. Now take a triangle with $\gamma$ as an edge. Denote the other two edges by $\sigma$ and $\alpha$.  They cannot both be invariant for then by applying $\tau$ to the triangle, we build a parallelogram with opposite sides identified.  This would give a closed torus embedded in $(X,\omega)$, which is impossible.

Suppose without loss of generality that $\sigma$ is not invariant.  If $\alpha$ is invariant, then $\sigma$ defines a simple cylinder containing $\gamma$ which separates $\gamma$ from $\rho_0$ if $\rho_0\neq \alpha$.  If $\rho_0=\alpha$, 
then the simple cylinder contains both $\gamma$ and $\rho_0$.  In the latter case,  $\{\gamma,\rho_0,\sigma\}$ bounds a triangle. 

If $\alpha$ is not invariant, then the collection $\{\sigma$, $\tau(\sigma)$, $\alpha,\tau(\alpha)\}$  cuts the surface into three components, one of which contains $\rho_0$ and one of which is a parallelogram containing $\gamma$ and bounded
 by the union of the four saddle connections.  Since one of the two components in the complement of the parallelogram contains $\rho_0$, one of the pairs $(\sigma,\tau(\sigma))$ or $(\alpha,\tau(\alpha))$ separates $\gamma$ from $\rho_0$. 

Finally, if $\gamma$ is not invariant, then we excise the pair $\gamma \cup \tau(\gamma)$.  By Lemma~\ref{TwoBdSC:Lemma}, this splits the surface into two components - one of which contains $\rho_0$.  We consider the component 
that contains $\rho_0$ and identify the two boundary components that corresponded to $\gamma$ and $\tau(\gamma)$.  The resulting saddle connection is invariant under $\tau$ by construction, and we apply the argument above.

We observe that by forming a cylinder or parallelogram containing $\alpha$ with $\sigma$, the cross product  claim in Item~\ref{lem:triangulation:crossprod} follows.
\end{proof}

The final lemma of this section is an important result for the counting below.  It says that there are not many saddle connections forming a triangle with the slit that are much shorter than the slit.

Let $\gamma$ be a saddle connection forming a triangle with an invariant saddle connection  $\beta$.  We denote by $P_\gamma$ the parallelogram determined by $\gamma$ and $\beta$, i.e., the union of the triangles formed by $\{\beta,\gamma\}$ and $\{\beta,\tau(\gamma)\}$.

\begin{lemma}\label{lem:shortL}
Let $(X,\omega)\in\cH^{hyp}(\kappa)$ and let $\beta$ be an invariant saddle connection on $(X,\omega)$.  Then there exists at most one  non-invariant saddle connection $\gamma$ of length $|\gamma|\leq |\beta|/2$ such that $P_\gamma$ is a simple cylinder on $(X,\omega)$.  
\end{lemma}

\begin{proof}
If no or one such $\gamma$ exists, we are done.  By contradiction, assume that two such saddle connections $\gamma$ and $\gamma'$ exist and form the parallelograms $P_{\gamma}$ and $P_{\gamma'}$ as described above, and assume that both $P_{\gamma}$ and $P_{\gamma'}$ are simple cylinders.  Then $P_\gamma$ and $P_{\gamma'}$ have to intersect. Hence, there is a triangle $\Delta$ in $P_\gamma\cap P_{\gamma'}$ such that the sides of $\Delta$ are given by $\beta$, a subsegment of $\gamma$, and a subsegment of $\gamma'$. Therefore, $|\beta|\leq |\gamma|+|\gamma'|$, which implies that at least one of $\gamma$ or $\gamma'$ has length at least half of the length of $\beta$.  This contradiction implies that at most one of them is a simple cylinder satisfying the length inequality.
\end{proof}

\subsection{Notation and Conventions}
In the rest of the paper, we will use the following notation and conventions:

\begin{itemize}
\item the letters $\gamma, \rho, \beta$ or $\sigma$ 
denotes saddle connections on translation surfaces. 

\item The notation $\kappa$ means $(g-1, g-1)$ or $(2g-2)$, and the corresponding hyperelliptic component will be denoted by $\cH^{hyp}(\kappa)$. Let $d$ be the complex dimension of $\mathcal{H}^{hyp}(\kappa)$.  For $\kappa=(2g-2)$, we have $d=2g$, and for $\kappa=(g-1,g-1)$, we have $d=2g+1$. The hyperelliptic involution of hyperelliptic translation surfaces will be denoted by $\tau$. 

\item The symbols $c$
(or $ c'$, $c_1$, $c_2$, etc.) will denote positive constants that may depend on the stratum,
while independent of the translation surfaces in the stratum. Sometimes the precise values of $c$ may change in the presentation of results as they absorb other constants. 

\item In the counting of saddle connections, if two saddle connections are parallel or symmetric under the hyperelliptic involution, i.e., $\gamma'=\tau(\gamma)$, then we only consider one of them because there are at most finitely many such parallel saddle connections bounded in terms of genus.

\item A \emph{subsurface} $(Y, \eta)$  of a translation surface $(X, \omega)$ is defined to be a compact subsurface $Y\subset X$, where $\eta$ is the restriction of $\omega$ on $Y$.
We always assume that $\partial Y$ is a union of 
some saddle connections of $(X,\omega)$.

\item We take $\log x := \log_2 x$.  
\end{itemize}

\section{Auxiliary Sets of Saddle Connections}
\label{AuxSC:Section}

To derive the upper bound, we will consider a collection of saddle connections 
$\Gamma$ of logarithmic scale $j$.  We will organize this collection into subsets and construct additional collections of saddle connections.  These additional saddle connections will allow us to bound the number of saddle connections in the original collection.

We will use induction, which will force us to work with subsurfaces of our unit area surface.  Consequently, it will be necessary to allow the area to be any quantity less than one.  Throughout we assume that $(X,\omega)$ is a hyperelliptic translation surface of area $A$, and we write $(X,\omega)\in \cH_A^{hyp}(\kappa)$.

The idea behind Lemma~\ref{SeparatingSCSetForGammaBeta:Lemma} is as follows.  We fix a small constant $\epsilon > 0$.  
Recall that two saddle connections $\gamma$ and 
$\gamma'$ on a translation surface of area $A$ are $\epsilon$-isolated if the angle between them is at least $$\frac{\epsilon^2A}{|\gamma||\gamma'|}.$$ 
The upper bound is easy to prove if the saddle connections are $\epsilon$-isolated.  On the other hand, if they are not $\epsilon$-isolated, an application of the Teichm\"uller flow shows that they must be contained in small area subsurfaces
(applying Lemma \ref{lem:Teich} and Lemma \ref{EMEpsCplx:Lemma}).  However, any such subsurface is bounded by saddle connections.  We choose one such boundary saddle 
connection $\sigma_1$ so that 
the complement of $\sigma_1\cup \tau(\sigma_1)$ has a component
with area at least $\frac{A}{d-2}$. We can consider the set, which we call $\Sigma_1$, of all such boundary saddle connections for all such clusters of non-$\epsilon$-isolated saddle connections in our original collection of saddle connections.  These new boundary saddle connections may be tightly clustered together themselves, so we continue this process and form a collection of saddle connections $\Sigma_2$.  We will prove that the process of forming such increasingly larger small subsurfaces and considering their boundary saddle connections will terminate in a finite number of steps depending on the dimension of the stratum.  In what follows, we will use these specially constructed collections to establish the required upper bounds.

%\subsection{Partitions of saddle connections}
%We begin with a pair of definitions  that explains how we will partition our collections of saddle connections and all of the sets of auxiliary saddle connections associated to it.  In Lemma~\ref{SeparatingSCSetForGammaBeta:Lemma}, we will carefully explain how we construct the sequences of saddle connections $(\sigma_0, \sigma_1, \ldots, \sigma_\ell)$ in the definition below to satisfy a list of necessary properties.

%\subsection{Constructing Sequences of Saddle Connections}

%In this section we prove the following key lemma, which associates each $\gamma$ a sequence of saddle connections. 
%Based on the construction of these sequences, we automatically obtain a canonical partition of $\Gamma$ as described by Definition \ref{def:partition}.

\begin{lemma}
\label{SeparatingSCSetForGammaBeta:Lemma}
Let $(X, \omega)$ be a translation surface in $\cH_A^{hyp}(\kappa)$, and let $\beta$ be an invariant saddle connection on $(X, \omega)$.
Let $d=\dim_{\mathbb{C}} \cH^{hyp}(\kappa)$.
Let $c_0$ be a constant depending only on $d$. Let $\epsilon > 0$ be sufficiently small. 
%Given any $j$, 
Assume that $\Gamma$ is a collection of saddle connections on $(X,\omega)$
such that all $\gamma \in \Gamma$ satisfy
\begin{enumerate}
\item  $|\gamma|\geq |\beta|/2$. 
\item \label{SeparatingSCSetForGammaBeta:Lemma:Assump:gammaL} For all $\gamma, \gamma' \in \Gamma$, $\gamma$ has the same logarithmic scale as $\gamma'$.
%\item \label{SeparatingSCSetForGammaBeta:Lemma:Assump:TauInvGamma} $\gamma \in \Gamma$ if and only if $\tau(\gamma) \in \Gamma$.\footnote{Item~\ref{SeparatingSCSetForGammaBeta:Lemma:Assump:TauInvGamma} is actually subsumed by our convention that collections of saddle connections always contain all parallel saddle connections.}
\item  \label{SeparatingSCSetForGammaBeta:Lemma:Assump:gammabetaTri} Each $\gamma\in \Gamma$ together with $\beta$ form two sides of a triangle of area at most $c_0 \epsilon^2 A$.
\end{enumerate}
Then there exist positive constants $c_1$ and $c_2$ depending only on the stratum, and $m \leq d-2$, such that for each $\gamma \in \Gamma$ there is an associated sequence of saddle connections $(\gamma=\sigma_0, \sigma_1, \ldots, \sigma_p)$, where $p \leq m$,
satisfying the following properties:
\begin{enumerate}[label=\Alph*)]
\item \label{SeparatingSCSetForGammaBeta:Lemma:IntDisjBeta} The set of saddle connections $\{\beta, \sigma_0, \ldots, \sigma_p\}$ are pairwise interiorly disjoint.
%\item \label{SeparatingSCSetForGammaBeta:Lemma:sigiCSepsigip1} 
\item \label{SeparatingSCSetForGammaBeta:Lemma:triangle} $\sigma_i$ and  $\sigma_{i+1}$ bound a triangle.
\item \label{SeparatingSCSetForGammaBeta:Lemma:sigiAngsigip1} The angle between $\sigma_i$ and $\sigma_{i+1}$ on $(X,\omega)$ is at most $$c_2 \frac{\epsilon^2A}{|\sigma_i||\sigma_{i+1}|}.$$
\item \label{SeparatingSCSetForGammaBeta:Lemma:sigiLengthsigip1} $|\sigma_{i+1}| \leq  c_2|\sigma_i|$.
\item \label{SeparatingSCSetForGammaBeta:Lemma:sigmIso2}  
If $\gamma$ and $\gamma'$ define sequences such that the terminal members  $\sigma_{p}$ and  $\sigma'_p$, respectively,  have the same logarithmic scale, 
then 
either $\sigma_{p}$ and  $\sigma'_p$ are $\epsilon$-isolated or interiorly disjoint.  In particular, two saddle connections $\gamma$ and $\gamma'$ with $p = 0$ are
either $\epsilon$-isolated or interiorly disjoint.\footnote{We observe that our convention of including all parallel saddle connections is automatically satisfied because any two parallel saddle connections are obviously interiorly disjoint.}
\item \label{SeparatingSCSetForGammaBeta:Lemma:SigmaPart} Denote $\Sigma_0=\Gamma$. We consider the entire collection of all saddle connections in $\Gamma$ as well as all of the $\sigma_i$ appearing in the sequences associated to each saddle connection $\gamma \in \Gamma$.  We organize\footnote{We observe that for two distinct saddle connections $\gamma$ and $\gamma'$, $\sigma_1$ in the sequence for $\gamma$ could be equal to $\sigma_2$ in the sequence for $\gamma'$.  Consequently, this saddle connection will appear in both collections $\Sigma_1$ and $\Sigma_2$.  Therefore, what we form is strictly speaking not a partition.} this collection into sets 
 $\{\Sigma_i \ | \ 0 \leq i \leq m\}$, where for each $0\leq i \leq m$, each set $\Sigma_i$ contains all saddle connections $\sigma_i$ in the sequences associated to each $\gamma \in \Gamma$.  
 We partition the $\Sigma_i$ first  into subsets denoted $\Sigma_i(j)$ consisting of all elements of $\Sigma_i$ of logarithmic scale $j$.
Then we partition each $\Sigma_i(j)$ into  subsets
by
$$\Sigma_i(j)=  \left( \bigcup_k \Sigma_{i,k}(j) \right) \cup \hat \Sigma_i(j)$$
so that 
\begin{enumerate}
%\item all saddle connections in $\Sigma_{i}(j_k)$ have the same logarithmic scale $j=j_k$ (and we allow $j_k = j_{k'}$ for distinct values of $j_k$ and $j_{k'}$)
\item 
The subset $\hat \Sigma_i(j)$ is the maximal set consisting of those saddle connections in $ \Sigma_i(j)$ that are either  $\epsilon$-isolated or  interiorly disjoint  from each other. In other words, if $\sigma_i$ is in $\hat \Sigma_i(j) $, then any other  $\sigma_i'$ in $\hat \Sigma_i(j)$  must be either $\epsilon$-isolated to $\sigma_i$ or interiorly disjoint from $\sigma_i$. 
\item Each $\Sigma_{i,k}(j)$ is a subset of $\Sigma_{i}(j)$ in the complement of $\hat\Sigma_i(j)$ with the following properties:
    \begin{enumerate}
      \item[(i)] Every pair of saddle connections in $\Sigma_{i,k}(j)$ is not $\epsilon$-isolated.
\item[(ii)] The union of the interiors of the saddle connections in $\Sigma_{i,k}(j)$ is connected.
\item[(iii)] All saddle connections in $\Sigma_{i,k}(j)$ determine the same saddle connection  $\sigma_{i+1}^{(k)} \in \Sigma_{i+1}$ in the sense that for each $\sigma_i \in \Sigma_{i,k}(j)$ appearing in a sequence associated to a $\gamma \in \Gamma$, the next saddle connection in every one of these sequences is $\sigma_{i+1}^{(k)}$.  Obviously, $\sigma_{i+1}^{(k)}$ needs to satisfy
Items B), C) and  D). 
\item[(iv)]  $ \Sigma_{i,k}(j)$  is maximal in the following sense.  No saddle connection in $\Sigma_i(j)\setminus (\hat\Sigma_i(j)\cup \Sigma_{i,k}(j))$ can be added to $\Sigma_{i,k}(j)$ so that the  previous items (i)-(iii) hold.
And if $\sigma,\sigma'\in \Sigma_i(j)$  interiorly intersect and are not $\epsilon$-isolated from each other, then at least one of them is contained in a $\Sigma_{i,k}(j)$. 
    \end{enumerate}
      \end{enumerate}
     
%\item the saddle connections in $\Sigma_{i,k}(j)$ are all $\epsilon$-isolated relative to $\Sigma_i(j)$,
%\item  \label{SeparatingSCSetForGammaBeta:Lemma:SigmaPart:SCContained} 

\item \label{SeparatingSCSetForGammaBeta:Lemma:subsurfs} In Item~\ref{SeparatingSCSetForGammaBeta:Lemma:SigmaPart} (b) (iii) above where
the saddle connections in $\Sigma_{i,k}(j)$ determine the same saddle connection $\sigma_{i+1}^{(k)}$, 
there is an associated subsurface $Y_{i+1}^{(k)}$ invariant under $\tau$ such that  
the saddle connections in $\Sigma_{i,k}(j)$  are contained in the interior of $Y_{i+1}^{(k)}$, and $\sigma_{i+1}^{(k)}$ is a boundary component of $Y_{i+1}^{(k)}$. 
Furthermore,  if $\sigma_{i+1}^{(k)}$ is contained in some $\Sigma_{i+1,k'}(j')$,
then the subsurface $Y_{i+2}^{(k')}$ associated to 
$\Sigma_{i+1,k'}(j')$ contains $Y_{i+1}^{(k)}$ in the interior. 

\item \label{SeparatingSCSetForGammaBeta:Lemma:beta-area}
For each $Y_i^{(k)}$, $1\leq i\leq m$, we have 
\begin{enumerate}
\item $\beta$ is contained in the interior of $Y_i^{(k)}$. 
    \item The area of  $Y_i^{(k)}$ is at most $c_1\epsilon^2A$.
\end{enumerate}
\end{enumerate}
\end{lemma}

\begin{definition}[Associated sequence of saddle connections]
 Given a saddle connection $\gamma$,  we say that a sequence $(\sigma_0, \sigma_1, \ldots, \sigma_p)$ as constructed in Lemma~\ref{SeparatingSCSetForGammaBeta:Lemma} is \emph{a sequence of saddle connections associated to $\gamma$}.
We call $p$  the \emph{size} of the sequence. 
We say that the sequence \emph{starts} at $\gamma$ and \emph{terminates} at $\sigma_p$. 
\end{definition}

\begin{remark}
 	As we shall see in the proof, for a given $\gamma$ the choice of $\sigma_i$ in each step may not be unique, so the sequence $(\sigma_1,\dots,\sigma_p)$ may be also not uniquely determined. 
In what follows, for each $\gamma$, we specify only one of these sequences. 
  In our construction, each $\Sigma_i$ is invariant under $\tau$, i.e., $\sigma_i\in \Sigma_i$ if and only if $\tau(\sigma_i) \in \Sigma_i$. So if $(\sigma_1,\ldots, \sigma_p)$ is a sequence associated to $\gamma$, then we shall let $(\tau(\sigma_1),\ldots, \tau(\sigma_p))$ be the sequence associated to $\tau(\gamma)$.
 \end{remark}

\begin{remark}
A saddle connection $\gamma$ determines a sequence of size $p$ and terminates at $\sigma_p$ if and only if $\sigma_p$ belongs to some $\hat{\Sigma}_{p}(j)$. 
In other words, the sequence terminates at $\sigma_p$ (stops at step $p$) if and only if the saddle connection $\sigma_p$ determined by $\gamma$ belongs to $\hat \Sigma_p(j)$.
In particular, ${\Sigma}_m(j)$ is exactly $\hat{\Sigma}_m(j)$, i.e., any pair of saddle connections in $\Sigma_m(j)$ are either $\epsilon$-isolated  or interiorly disjoint from each other.
\end{remark}

\begin{remark}
The idea in the  proof below is that, if at the $i$-th step, if $\Sigma_i(j)=\hat\Sigma_i(j)$, then  we can stop.  Otherwise we need to deal with the complement.  
We shall initially define 
$\Sigma_{i,k}(j)$ to be a subset of saddle connections of logarithmic scale $j$
satisfying the two assumptions  (i) and (ii) in Item \ref{SeparatingSCSetForGammaBeta:Lemma:SigmaPart}. With these two assumptions, we 
can contract  $\Sigma_{i,k}(j)$ using the Teichm\"uller flow and then apply Lemma \ref{EMEpsCplx:Lemma} to produce an $\epsilon$-complex. To make sure that  Item \ref{SeparatingSCSetForGammaBeta:Lemma:triangle} is satisfied, we may need to 
further subdivide $\Sigma_{i,k}(j)$. 
\end{remark}

\begin{example}
  In the case that $(X,\omega)\in \mathcal{H}(0,0)$, for a given invariant saddle connection $\beta$, any non-invariant $\gamma$  together with $\beta$ bound a triangle. In fact, $\gamma$ and $\tau(\gamma)$ bounds a simple cylinder containing $\beta$. See Figure \ref{simplecylinder}.
  Assume that the area of $(X,\omega)$ is one, and  let $\Gamma$ be the collections of non-invariant $\gamma$ of such that  $\gamma$  together with $\beta$ bound a triangle
  of area at most $\epsilon^2$. We claim that any $\gamma,\gamma'$ in $\Gamma$ are $\epsilon$-isolated. 
  For otherwise, by applying Lemma \ref{lem:Teich}, we can contract the triple $\beta, \gamma,\gamma'$ to be of length at most $2\epsilon$. 
  Then Lemma \ref{EMEpsCplx:Lemma} shows that $\beta, \gamma,\gamma'$ must fill the whole surface, with area at most $c_1\epsilon^2$. This contradicts our assumption that the area of $(X,\omega)$ is one. So in this simple case, the sequence associated to $\gamma\in \Gamma$ is $\gamma$ itself. 
  %We have $\Gamma=\Gamma^{(0)}$ in the canonical partition and t
  The collection of $\gamma$ in $\Gamma$ are $\epsilon$-isolated.
\end{example}

\begin{figure}[htb]
\centering
	\begin{tikzpicture}[scale=0.8]
    \filldraw[fill=lightgray] (0,0) -- (-1,0) -- (-2,-3) -- (-1,-3) -- cycle;
    \node at (-0.5,0.3) {$\beta$};
    \node at (-0.4,-2) {$\gamma$};
      \node at (-2.2,-1.9) {$\tau(\gamma)$};
 \node at (1.2,-2.2) {$\tau(\gamma)$};
    \node at (-1.7,-3.3) {$\beta$};
    \draw (0,0) -- (-1,-3) -- (0,-4) -- (1,-1) -- cycle;
	\end{tikzpicture}
	\caption{The saddle connection $\gamma$, which is not invariant by $\tau$, bounds a cylinder containing $\beta$.} 
	\label{simplecylinder}
\end{figure}

\begin{example}
   In the case that $(X,\omega)\in \mathcal{H}(2)$,  let us consider $\Gamma$ as the collections of non-invariant $\gamma$ of logarithmic scale $j$ such that  $\gamma\cup \tau(\gamma)$ bounds a simple cylinder containing $\beta$, and assume that  
  the areas of the simple cylinders are most $\epsilon^2$. 
  It may happen that $\Gamma$ is not $\epsilon$-isolated, and we may have a pair $\gamma,\gamma'$ in $\Gamma$ that determines a proper subsurface $Y$ with boundary 
  $\sigma \cup \tau(\sigma)$. By cutting along $\sigma \cup \tau(\sigma)$ and regluing,
  $Y$ becomes a slit torus containing $\beta$. The same argument as above shows that $Y$ has small area and thus any other $\sigma'$ that is produced by another pair of non-$\epsilon$-isolated saddle connections in $\Gamma$ must be  $\epsilon$-isolated with $\sigma$. In this case, the sequence associated to $\gamma$ is given by $(\gamma,\sigma)$. 
\end{example}

%\begin{figure}[htb]
%\centering
%	\begin{tikzpicture}[scale=0.8]
 %   \filldraw[fill=yellow!20] (0,0) -- (-1,0) -- (-2,-3) -- (-1,-3) -- cycle;
 %   \node at (-0.5,0.2) {$\beta$};
 %   \node at (-0.4,-2) {$\gamma$};
 %   \draw (0,0) -- (-1,-3) -- (0,-4) -- (1,-1) -- cycle;
  %   \draw (0,0) -- (1,-1) -- (3,2) -- (2,3) -- cycle;
%	\end{tikzpicture}
%	\caption{The saddle connection $\gamma$, which is not invariant by $\tau$, bounds a cylinder containing $\beta$.} 
%	\label{simplecylinder2}
%\end{figure}

We now prove Lemma \ref{SeparatingSCSetForGammaBeta:Lemma}.

\begin{proof}
We outline the argument first.  The procedure is iterative.  We will start with $\Gamma=\Sigma_0$ and form the sequences of saddle connections $(\sigma_0, \ldots, \sigma_p)$ iteratively.
We will first partition $\Gamma$ into subsets,
in terms of whether the $\gamma$  are $\epsilon$-isolated or not.  In other words, $\gamma$ and $\gamma'$ are in the same subset if their angle apart is at most $\epsilon^2 A/2^{2j}$. Then we will use these to construct collections of saddle connections $\Sigma_1$.  We will then partition these collections of saddle connections, examine their properties and explain how to either partition them further if they lack all of the desired properties or discard the elements that lack the desired properties in order to rechoose elements that do possess them.  We then explain why the sequence of collections of saddle connections $\Sigma_i$ terminates at a collection $\Sigma_m$ for some $m \leq d-2$.  This will require an area argument and use the Teichm\"uller flow.  Consequently, we will consider a sequence of subsurfaces that occur from these collections and we will need to record their area to achieve the desired conclusion.

\paragraph{Partitioning $\Gamma$ -- $\epsilon$-isolated:} 
First, we collect those saddle connections in  $\Gamma$ that either are $\epsilon$-isolated or if not, are interiorly disjoint from any saddle connection in $\Gamma$ that they are not isolated from. Denote this set by $\hat \Sigma_0$.
In other words, if $\gamma$ is in $\hat \Sigma_0$, then any other $\gamma'$ 
in $\Gamma$ must be $\epsilon$-isolated to $\gamma$ or interiorly disjoint from $\gamma$.

By Assumption~\ref{SeparatingSCSetForGammaBeta:Lemma:Assump:gammaL}, all saddle connections in $\Gamma$ have logarithmic scale $j$.
For the saddle connections in $\hat \Sigma_0$, we let each $\gamma$ be its own sequence so that $p=0$, and each claim either holds vacuously or follows from the assumptions.

\paragraph{Partitioning $\Gamma$ -- $\epsilon$-non-isolated (The Base Case):} 
Again, by Assumption~\ref{SeparatingSCSetForGammaBeta:Lemma:Assump:gammaL}, all saddle connections in $\Gamma$ have logarithmic scale $j$.  
For the saddle connections in $\Gamma$ that are not in $\hat \Sigma_0$, we form a preliminary partition into collections $\Sigma_{0,k}(j)$ as follows. 
First partition them so that for each $k$, for any pair of saddle connections in $\Sigma_{0,k}(j)$ is not $\epsilon$-isolated from each other and the union of the interior of each saddle connection $\gamma\in\Sigma_{0,k}(j)$ is a connected subset of $(X,\omega)$.  In particular, for any $\gamma\in\Sigma_{0,k}(j)$, there is $\gamma'\in \Sigma_{0,k}(j)$ that interiorly intersects $\gamma$.  Furthermore, for each $\gamma \in \Sigma_{0,k}(j)$, we assume $\tau(\gamma) \in \Sigma_{0,k}(j)$, which can be done because $\tau(\gamma) \in \Gamma$ by the assumption that all saddle connections parallel to $\gamma \in \Gamma$ are in $\Gamma$.

Fix a particular $\Sigma_{0,k}(j)$, we apply the following argument as we did in Lemma~\ref{lem:Teich}.  Let $\gamma$ be one of the longest saddle connections in $\Sigma_{0,k}(j)$. Rotate  $\gamma$ to be vertical and apply the Teichm\"uller flow $g_t$ for $$e^t=\frac{|\gamma|}{\sqrt {A}\epsilon}$$ to contract it to length $\sqrt A\epsilon$. Let $(X_t,\omega_t)=g_t \cdot (X,\omega)$ denote the deformed surface, and we will refer to $(X, \omega)$ as the \emph{base surface}.  On $(X_t, \omega_t)$, the vertical component of any other saddle connection in $\Sigma_{0,k}(j)$ is at most $\sqrt A\epsilon$ and the angle condition says the horizontal component is also at most $ \sqrt A\epsilon$.

Moreover, by our assumption, $\gamma$ and $\beta$ bound a triangle $T$ with area at most $c_0\epsilon^2 A$. We denote the third side of $T$ by $\rho$. By our assumption, $$|\beta|< 2|\gamma|.$$  On $(X_t, \omega_t)$, the vertical component of $\beta$ is at most $2 \sqrt A\epsilon$ and the area condition guarantees that the horizontal component of $\beta$ is  at most $ c_0 \sqrt A\epsilon$. Therefore, 
on $(X_t, \omega_t)$, the length of $\beta$ is at most $(2+c_0)\sqrt A\epsilon$.
Since  $\rho$ is the third side of the triangle, on $(X_t, \omega_t)$, the length of $\rho$ is at most $(3+c_0)\sqrt A\epsilon$.

By Lemma~\ref{EMEpsCplx:Lemma}, there exists $Y_1^{(k)} \subset (X_t,\omega_t)$, which is an $\epsilon$-complex containing all $\gamma'\in\Gamma_k$ and the saddle connections $\beta$, $\gamma$ and $\rho$ (and their images under $\tau$).\footnote{Strictly speaking,  $Y_1^{(k)}$ is a $(3+c_0)\sqrt A\epsilon$-complex because each of the saddle connections that we have has length at most $(3+c_0)\sqrt A\epsilon$ on $(X_t,\omega_t)$. This does not change our conclusion; it only changes the constants
$c_1,c_2$ in Lemma~\ref{EMEpsCplx:Lemma}.} The complex $Y_1^{(k)}$ has the property that for a constant $c_2$ depending only  on the stratum, the boundary $\partial Y_1^{(k)}$ consists of saddle connections of length at most $c_2\sqrt A\epsilon$.  

Because the collection $\Sigma_{0,k}(j)$ is invariant under the hyperelliptic involution $\tau$ by construction, we have that $Y_1^{(k)}$ is invariant under $\tau$. 
 It is clear that  each boundary saddle connection  $\sigma_1$ of $Y_1^{(k)}$ is disjoint from $\beta$ and $\gamma$. This justifies   Item~\ref{SeparatingSCSetForGammaBeta:Lemma:IntDisjBeta} for $\sigma_1$.  The area of $Y_1^{(k)}$ is at most $c_1A\epsilon^2$ as required in Item~\ref{SeparatingSCSetForGammaBeta:Lemma:beta-area}, where $c_1$ depends only  on the stratum and $C$ is as guaranteed by Lemma~\ref{EMEpsCplx:Lemma}.
 
Furthermore, since the union of the interiors of all $\gamma\in\Sigma_{0,k}(j)$  is a connected subset of $(X,\omega)$,  the resulting subsurface $Y_1^{(k)}$ is connected. The proof of \cite[Prop.~6.2]{EskinMasurAsymptForms} then implies that the diameter of $Y_1^{(k)}$ is at most $c_3\sqrt{A}\epsilon$ for some constant $c_3$ depending on $\cH^{hyp}(\kappa)$ and $C$.  Pulling back to the base surface  $(X,\omega)$, we see that  each boundary saddle connection  $\sigma_1$  of $Y_1^{(k)}$ satisfies $ |\sigma_1|\leq c_2|\gamma|$, for each $\gamma$ in the subcollection, as required by Item~\ref{SeparatingSCSetForGammaBeta:Lemma:sigiLengthsigip1}; the diameter of $Y_1^{(k)}$ is at most $2c_3|\gamma|$.

By Lemma~\ref{TauInvSubsurfBd:Lemma}, every boundary saddle connection of $Y_1^{(k)}$ occurs with its image under the hyperelliptic involution.  Moreover, by Lemma~\ref{TwoBdSC:Lemma}, excising any boundary saddle connection $\sigma$ of $Y_1^{(k)}$ along with $\tau(\sigma)$ will disconnect the surface.  Hence, excising every boundary saddle connection along with their involutions under $\tau$ will result in a collection of connected components.  Since the area of the complement of $Y_1^{(k)}$ has area at least $A-c_1A\epsilon^2$, there must be some connected component of the complement with area at least $$\frac {A-c_1A\epsilon^2}{d-2},$$ where $d-2$ is the upper bound of the number of complementary components.  To see the upper bound of $d-2$, observe that $Y_1^{(k)}$ has at least two complex dimensions and each complementary component corresponds to at least one complex dimension of the dimension of the stratum.  We choose  $\sigma_{1}^{(k)} \cup \tau(\sigma_{1}^{(k)})$ in the boundary of $Y_1^{(k)}$ such that the complementary component bounded by $\sigma_{1}^{(k)} \cup \tau(\sigma_{1}^{(k)})$, say $W_1^{(k)}$,  has largest area
\begin{equation}\label{eq:area:W1}
	\mathrm{Area}\left(W_1^{(k)}\right)\geq \frac {A-c_1A\epsilon^2}{d-2}\geq \frac{A}{d-1},
\end{equation}
where the last inequality holds for $\epsilon$ sufficiently small.

For each $k$, we take the corresponding boundary saddle connection of $Y_1^{(k)}$ as explained above and denote it by $\sigma_{1}^{(k)}$.  
So each $\Sigma_{0,k}(j)$ generates a subsurface $Y_1^{(k)}$ and determines a $\sigma_{1}^{(k)}$. 
With this, we take $\Sigma_1$ to be the union of all such boundary saddle connections $\sigma_{1}^{(k)}$.
 The fact that $\sigma_1^{(k)}$ has length at most $c_2\sqrt A \epsilon$ after applying the Teichm\"uller flow justifies Item~\ref{SeparatingSCSetForGammaBeta:Lemma:sigiAngsigip1}, and Item~\ref{SeparatingSCSetForGammaBeta:Lemma:sigiLengthsigip1} was also justified above.

However,  Item~\ref{SeparatingSCSetForGammaBeta:Lemma:triangle} may not be satisfied. 
 We may need to further partition $\Sigma_{0,k}(j)$ and choose a new saddle connection $\sigma_1^{k}$.  
Suppose that Item~\ref{SeparatingSCSetForGammaBeta:Lemma:triangle} is not satisfied for $\sigma_1^{(k)}$, that is, $\gamma$ and $\sigma_1^{(k)}$ do not bound a triangle. 

\begin{lemma}\label{lem:newsigma}
    With the above assumptions, we can further partition $\Sigma_{0,k}(j)$ and replace $\sigma_1^{(k)}$ with a saddle connection in $Y_1^{(k)}$ which satisfies Items~\ref{SeparatingSCSetForGammaBeta:Lemma:IntDisjBeta}, \ref{SeparatingSCSetForGammaBeta:Lemma:triangle} \ref{SeparatingSCSetForGammaBeta:Lemma:sigiAngsigip1}, \ref{SeparatingSCSetForGammaBeta:Lemma:sigiLengthsigip1}, and by substituting $Y_1^{(k)}$ with a subsurface we have  Items~\ref{SeparatingSCSetForGammaBeta:Lemma:SigmaPart}, \ref{SeparatingSCSetForGammaBeta:Lemma:subsurfs},
    \ref{SeparatingSCSetForGammaBeta:Lemma:beta-area}.
\end{lemma}

\begin{proof}
We follow the notation of Lemma~\ref{lem:triangulation}. 
Denote the boundary components of $Y_1^{(k)}$ by $\rho_0, \rho_1, \cdots,\rho_{N-2}$, where $\rho_0=\sigma_1^{(k)}$. 
Here $N\geq 2$. Consider the above triangle $T$ bounded by $\beta,\gamma$ and $\rho$.
We take $\rho_{N-1}=\beta$. 
Recall that, by our assumption, $|\beta|\leq 2|\gamma|$ and then $|\rho|< |\gamma|+|\beta|< 3|\gamma|.$ 

There are two cases. 
If $\rho$ is invariant, then we take $\rho_N=\rho$.
By identifying each $\rho_i$ with $\tau(\rho_i)$ (where $i=0,\ldots, N-2$) we obtain a closed surface, still denoted by ${Y}_1^{(k)}$. 
Applying Lemma~\ref{lem:triangulation} to ${Y}_1^{(k)}$, by considering $\rho_i, i=0,\ldots, N$ as invariant saddle connections, we see that there exists a saddle connection $\sigma$ with $\tau(\sigma)\neq \sigma$ satisfying the following:
\begin{itemize}
    \item $\sigma$ is interiorly disjoint from all $\rho_i$ and $\beta$; 
	\item $\sigma\cup\tau(\sigma)$ separates $\gamma$ and $\sigma_1^{(k)}$, and $\{\gamma,\sigma\}$ bounds a triangle on $Y_1^{(k)}$;
	\item $|\sigma|\leq c\left(\mathrm{Diam}\left(Y_1^{(k)}\right)+|\beta|+|\gamma|+|\rho|+\sum\limits_{i=0}^{N-2}|\rho_i|\right)$;
	\item $|\gamma\times \sigma|\leq 2\mathrm{Area}\left(Y_1^{(k)}\right)$.
\end{itemize}
Note that, by our previous estimate, the diameter of $Y_1^{(k)}$ is less than $2c_3|\gamma|$ and the boundary components of $Y_1^{(k)}$ have length at most $c_2|\gamma|$. So by taking $c_2$ sufficiently large we have $|\sigma|\leq c_2|\gamma|$. 
Furthermore, since $\sigma$ is interiorly disjoint from the triangle bounded by $\beta,\gamma,\rho$,  $\sigma\cup \tau(\sigma)$
does not separate $\beta$, $\gamma$ and $\rho$.

If $\rho$ is not invariant, then we cut $Y_1^{(k)}$ along $\rho\cup\tau(\rho)$ and replace $Y_1^{(k)}$ 
by the subsurface containing $\beta$ and $\gamma$.
(Note that $\beta,\gamma$ and $\rho$ bound a triangle, so $\rho\cup\tau(\rho)$ does not separate $\beta$ and $\gamma$).
We abuse notation and continue to call this new surface by $Y_1^{(k)}$.  We glue the boundary saddle connections of (the new) $Y_1^{(k)}$ together in pairs to make a closed surface and then apply Lemma~\ref{lem:triangulation} again. The same argument as before show that there exists a saddle connection $\sigma$ with $\tau(\sigma)\neq \sigma$ satisfying the same properties as in the previous case. 

In either case, for each $\sigma$, we have a $\sigma$ that bounds a triangle with $\gamma$. Then 
we replace $\sigma_1^{(k)}$ with $\sigma$.
This new $\sigma_1^{(k)}$ together with $\tau(\sigma_1^{(k)})$ bound a subsurface of $Y_1^{(k)}$ containing $\beta$ and $\gamma$, and we replace $Y_1^{(k)}$ with this subsurface (still called $Y_1^{(k)}$ for simplicity). 
Furthermore, we partition the previous collection $\Sigma_{0,k}(j)$
into subcollections such that, in each  subcollection, all $\gamma$ are contained in 
the same $Y_1^{(k)}$ and each of them bounds a triangle with $\sigma_1^{(k)}$.
By abuse of notation, we denote each such subcollection by $\Sigma_{0,k}(j)$.
From the discussion above, we see that the new $\sigma_1^{(k)}$ satisfies 
Items~\ref{SeparatingSCSetForGammaBeta:Lemma:IntDisjBeta}, \ref{SeparatingSCSetForGammaBeta:Lemma:triangle} \ref{SeparatingSCSetForGammaBeta:Lemma:sigiAngsigip1}, \ref{SeparatingSCSetForGammaBeta:Lemma:sigiLengthsigip1}, and $\sigma_1^{(k)}$ is a boundary saddle connection of $Y_1^{(k)}$, which satisfies Items~\ref{SeparatingSCSetForGammaBeta:Lemma:SigmaPart}, \ref{SeparatingSCSetForGammaBeta:Lemma:subsurfs} and \ref{SeparatingSCSetForGammaBeta:Lemma:beta-area}.
\end{proof}

We have shown that to every $\gamma\in\Gamma$, either $\gamma$ is $\epsilon$-isolated  or there exists an associated saddle connection $\sigma_1$, and the non isolated saddle connections in $\Gamma$  can be  partitioned into families $\Sigma_{0,k}(j)$ so that for each $k$, there is a $\sigma_1^{(k)}$ associated to every $\gamma\in \Sigma_{0,k}(j)$. By our construction,  $W_1^{(k)}$ is a complementary component of $Y_1^{(k)}$ with largest area. Note that, in Lemma \ref{lem:newsigma},  the area of 
$Y_1^{(k)}$ decreases when we replace $Y_1^{(k)}$ with a subsurface, while the area of $W_1^{(k)}$ increases when we replace it with a larger one. So Inequality \eqref{eq:area:W1} still holds. 

We let $\Sigma_1$ be the set of all $\sigma_1^{(k)}$ found in this construction.
We have also found  subsurfaces $Y_1^{(k)}$.

\paragraph{The inductive step:} We will begin by defining $\Sigma_i$.  The arguments in the base case above as well as the construction itself will allow us to easily satisfy nearly all of the desired conclusions.  On the other hand, the area claim in Item~\ref{SeparatingSCSetForGammaBeta:Lemma:beta-area} requires additional argument that we provide here.  At the end of the proof, we will explain how each of the remaining items are satisfied.

We inductively define $\Sigma_i$ as follows.  Suppose by the induction assumption that we have defined $\Sigma_1, \ldots, \Sigma_{i-1}$.   Also inductively assume that associated to any sequence $(\gamma=\sigma_0,\sigma_1,\ldots,\sigma_{i-1})$ we have the sequence of subsurfaces
$$ Y_1 \subset Y_2 \subset \cdots \subset Y_{i-1}.$$
By the inductive construction, $\sigma_1$ is a boundary saddle connection of $Y_1$,
and, if $\gamma=\sigma_0 \in \Sigma_{0,k}$, then $\Sigma_{0,k}$ is contained in $Y_1=Y_1^{(k)}$. We have similar conclusions for $\sigma_2,\ldots,\sigma_{i-1}$. 
Inductively, we assume that $\sigma_{i-1}\cup \tau(\sigma_{i-1})$,
as a pair of boundary saddle connections of $Y_{i-1}$, 
has a complementary subsurface $W_{i-1}$ satisfying $$\text{Area}\left(W_{i-1}\right)\geq \frac{A}{(d-1)^{i-1}} .$$

We  now form a collection $\Sigma_i$ similar to the way we formed $\Sigma_1$ from the collection $\Gamma$ above.  The saddle connections in $\Gamma$ were all assumed to have the same logarithmic scale.
 However, this is not necessarily the case for the saddle connections in $\Sigma_{i-1}$.  Therefore, we initially partition $\Sigma_{i-1}$ into $\Sigma_{i-1}(j)$ by the logarithmic scales of the saddle connections in it.  Then we further partition $\Sigma_{i-1}(j)$ into subsets 
  $\Sigma_{i-1,k}(j)$ that satisfy the same properties as held for $\Sigma_{0,k}(j)$, which we repeat here: 
\begin{itemize}
    \item any pair of saddle connections $\sigma_{i-1}$ and $\sigma'_{i-1}$ in $\Sigma_{i-1,k}(j)$ is not $\epsilon$-isolated, and
 the union of the interiors of all saddle connection in $\Sigma_{i-1,k}(j)$  is a connected subset of $(X,\omega)$. In particular, for every $\sigma_{i-1}$ there is  $\sigma'_{i-1}$ which interiorly intersects it.
 Moreover the union of such $\Sigma_{i-1,k}(j)$ over all $k$ is a maximal subset of $\Sigma_{i-1}(j)$.
\item  The remaining set  $\hat \Sigma_{i-1}(j)$  is a set of saddle connections such that any pair are either  $\epsilon$-isolated or disjoint.  
\end{itemize}

For $\sigma_{i-1}\in \hat\Sigma_{i-1}(j)$, 
we take $p=i-1$ and let the sequences of saddle connections determining $\sigma_{i-1}$ terminate at $\sigma_{i-1} = \sigma_{p}$.
%We denote the subset of saddle connections 
%$\gamma\in \Gamma$ that determine those $\epsilon$-isolated $\sigma_{i-1}$ by $\Gamma^{(i-1)}$.

For each set  $\Sigma_{i-1,k}(j)$, we now follow the same procedure using the Teichm\"uller flow to obtain a subsurface 
by first forming  the $\epsilon$-complex denoted by $Z_i^{(k)}$. 
  Note that $Z_i^{(k)}$ intersects the interior of $W_{i-1}^{(k)}$ 
because $W_{i-1}^{(k)}$ and  $Y_{i-1}^{(k)}$ have a common boundary - a pair of saddle connections in the interior of $Z_i^{(k)}$. 
As above, Lemma~\ref{EMEpsCplx:Lemma} implies the area of $Z_i^{(k)}$  is at most $c_1 A\epsilon^2$, and the boundary after flowing will have saddle connections of length at most $c_2\sqrt A \epsilon$. 
Define $$Y_i^{(k)} = Y_{i-1}^{(k)} \cup \left(Z_i^{(k)}\cap W^{(k)}_{i-1}\right).$$  Since $Z_i^{(k)}$  has area at most $c_1 A\epsilon^2$, for $\epsilon$ sufficiently small, one of the complementary components of $Y_{i}^{(k)}$ in $W_{i-1}^{(k)}$, which we denote by $W_i^{(k)}$, has area
$$\mathrm{Area}\left(W_{i}^{(k)}\right) \geq \frac{\mathrm{Area}\left(W_{i-1}^{(k)}\right)-c_1A\epsilon^2}{d-2} \geq \frac{\mathrm{Area}\left(W_{i-1}^{(k)}\right)}{d-1}.$$
(That is, we have some small area subsurface $Y_{i}^{(k)}\cap W^{(k)}_{i-1}=Z_{i}^{(k)}\cap W_{i-1}^{(k)}$ that we are removing from $W_{i-1}^{(k)}$, so some component left has area bounded from below.)
The inductive assumption on the area of $W_{i-1}^{(k)}$ yields
\begin{equation}\label{eq:area:W}
	\mathrm{Area}\left(W_{i}^{(k)}\right)\geq  \frac{A}{(d-1)^{i-1}} \cdot \frac{1}{d-1} = \frac{A}{(d-1)^i}.
\end{equation}

Let $\sigma_i^{(k)}$ be a boundary component of $W_i^{(k)}$.  We apply the same technique as in the case of $\sigma_1^{(k)}$  to choose a new $\sigma_i^{(k)}$ if necessary to guarantee that Item~\ref{SeparatingSCSetForGammaBeta:Lemma:triangle} is satisfied.  By further partitioning $\Sigma_{i-1,k}(j)$, we may assume that all $\sigma_{i-1}$ in $\Sigma_{i-1,k}(j)$ determine the same $\sigma_i^{(k')}$ and $Y_i^{(k')}$, for some $k'$.

We now verify each conclusion.  Item~\ref{SeparatingSCSetForGammaBeta:Lemma:IntDisjBeta} follows because each saddle connection represents the boundary of a subsurface containing all previous saddle connections.  Item~\ref{SeparatingSCSetForGammaBeta:Lemma:triangle} was just verified above.  Items~\ref{SeparatingSCSetForGammaBeta:Lemma:sigiAngsigip1} and \ref{SeparatingSCSetForGammaBeta:Lemma:sigiLengthsigip1} follow from the same calculations involving the Teichm\"uller flow from the base case.  Each part of  Item~\ref{SeparatingSCSetForGammaBeta:Lemma:SigmaPart} and Item~\ref{SeparatingSCSetForGammaBeta:Lemma:subsurfs} was justified via the construction above.  Finally, in Item~\ref{SeparatingSCSetForGammaBeta:Lemma:beta-area}, the containment of $\beta$ follows from the fact that each successive subsurface contains all previous subsurfaces and $\beta \subset Y_1^{(k)}$ for all $k$ by the base case, and the area condition is satisfied by the argument above.

\paragraph{Termination of the Construction:}
Therefore, this process can be continued each time we have a subcollection of saddle connections that are not $\epsilon$-isolated.  We claim that this process must stop in $p\leq m$ steps,  where $m$ is bounded by $d-2$.  Indeed, the complement of each $Y_i^{(k)}$ necessarily lies in a stratum of strictly decreasing dimension as $i$ increases because for each $i$ we added a triangle and its image under $\tau$ to $Y_i^{(k)}$ so we lose one dimension.
Furthermore, since the complement of $Y_m^{(k)}$ must have dimension at least two, we conclude that $m\leq d-2$.  The area condition and Inequality~\eqref{eq:area:W} imply that when it does terminate,  for $\epsilon$ sufficiently small,  the complement of $Y_m^{(k)}$ in $(X, \omega)$ has area at least 
$$\mathrm{Area}\left(W_{m}^{(k)}\right)\geq \frac{A}{(d-1)^m}.$$
Because the process has terminated, any two distinct saddle connections in $\Sigma_m$ with the same logarithmic scale are necessarily $\epsilon$-isolated or interiorly disjoint.
% which justifies Item~\ref{SeparatingSCSetForGammaBeta:Lemma:sigmIso2}.
\end{proof}

\section{Upper Bound Theorems}
\label{UpperBound:Section}

In this section, we will prove Theorem~\ref{CountAlphaSC:Theorem} and use it to prove the upper bound in Theorem~\ref{MainSummaryThm}. We set notation.  Let $\beta$ be a marked saddle connection on $(X,\omega) \in \cH^{hyp}(\kappa)$ satisfying $\tau(\beta)=\beta$.  In order to have control of the non-isolated saddle connections, the concept of $C$-separated  from the introduction will be essential.
Let us recall the definition.

\begin{definition}
Let $C>1$.  Given  $\beta$ and $\gamma$, we say that 
\emph{$\beta$ is $C$-separated from $\gamma$} by a pair of saddle connections
$\rho\cup\tau(\rho) $ if any path from $\alpha$ to $\beta$ crosses $\rho\cup\tau(\rho)$ and  $|\rho|<|\beta|/C$.  We say that \emph{$\alpha$ is not $C$-separated from $\beta$} if there is no such $\rho\cup\tau(\rho)$.
\end{definition}

For this reason, we will fix a constant $C$ and restrict ourselves to counting saddle connections that are \emph{not} $C$-separated from our slit $\beta$.

In the introduction, we defined the set $A((X, \omega) \setminus \beta,L,C)$ with \emph{length} as the argument, and here we define it with \emph{logarithmic scale}.  Indeed, logarithmic scale is more naturally adapted to the proof than length, so we only use logarithmic scale throughout this section.  For $C>1$, let $A((X, \omega) \setminus \beta,j,C)$ be the set of saddle connections on $(X, \omega)$ of logarithmic scale $j$ that are interiorly disjoint from $\beta$ and not $C$-separated from $\beta$.

Theorem~\ref{CountAlphaSC:Theorem} can be restated as follows.

\begin{theorem}
\label{CountAlphaSC:Theorem3}
Given $C>1$, there exists $c$ depending only on $\cH^{hyp}(\kappa)$ and $C$ such that if $(X, \omega) \in \cH^{hyp}(\kappa)$, $\beta$ is an invariant saddle connection on $(X, \omega)$ of logarithmic scale $\ell$, and $j - \ell \geq 1$, then
$$\# A((X, \omega) \setminus \beta,j,C) \leq c({j-\ell})^{d-2}2^{j-\ell},$$
where $d = dim_\bC\,\cH^{hyp}(\kappa)$.
\end{theorem}

We recall from the introduction that $A^{cyl}((X, \omega) \setminus \beta,j) $ and $A^{non-inv}((X, \omega) \setminus \beta, j)$ are the subsets of $A((X, \omega) \setminus \beta,j)$ consisting of cylinders and non-invariant saddle connections that are interiorly disjoint from $\beta$, respectively.  With a small modification, the proof of Theorem~\ref{CountAlphaSC:Theorem3} can also be used to prove the following theorem.  

\begin{theorem}
\label{CountAlphaSC:Theorem3cyl}
Given $C>1$, there exists $c$ depending only on $\cH^{hyp}(\kappa)$ and $C$ such that if $(X, \omega) \in \cH^{hyp}(\kappa)$, $\beta$ is an invariant saddle connection on $(X, \omega)$ of logarithmic scale $\ell$, and $j - \ell \geq 1$, then 
$$\# A^{cyl}((X, \omega) \setminus \beta,j,C) \leq c({j-\ell})^{d-3}2^{j-\ell}$$
and
$$\# A^{non-inv}((X, \omega) \setminus \beta,j,C) \leq c({j-\ell})^{d-3}2^{j-\ell},$$
where $d = dim_\bC\,\cH^{hyp}(\kappa)$.
\end{theorem}

Theorem~\ref{CountAlphaSC:Theorem3cyl} will be proven in the course of proving Theorem~\ref{CountAlphaSC:Theorem3} because the proofs are identical up to the final step.  The key to the proof is counting triangles.  More precisely, let $T((X, \omega) \setminus \beta, j)$ be the set of non-invariant saddle connections $\gamma$ of logarithmic scale $j$ such that $\beta$ and $\gamma$ are two sides of a triangle.

\begin{theorem}
\label{CountTriangles}
There exists $c$ depending only on $\cH^{hyp}(\kappa)$ such that if $(X, \omega) \in \cH^{hyp}(\kappa)$ and $\beta$ is an invariant saddle connection on $(X, \omega)$ of logarithmic scale $\ell$, then $$\# T((X, \omega) \setminus \beta, j) \leq c \left( 2^{j-\ell}+ 1\right).$$
\end{theorem}

 The proof of Theorems~\ref{CountAlphaSC:Theorem3} and \ref{CountTriangles}  is by induction on the dimension of strata and will work as 
follows. Given a stratum $\cH^{hyp}(\kappa)$, we will assume by induction that  Theorem \ref{CountTriangles} is true for strata $\cH^{hyp}(\kappa')$ of smaller dimension than that of $\cH^{hyp}(\kappa)$.  We will then use the inductive assumption to prove Theorem \ref{CountTriangles} for $\cH^{hyp}(\kappa)$. 
Finally, we prove Theorem~\ref{CountAlphaSC:Theorem3} using  
Theorem \ref{CountTriangles}. 
Furthermore, since the proof is by induction on the exponent of $j- \ell$, only at the very end of the induction does the exponent $d-2$ or $d-3$ appear.  This is how we will be able to prove Theorem~\ref{CountAlphaSC:Theorem3} and Theorem \ref{CountAlphaSC:Theorem3cyl} simultaneously.

It should be noted that to prove the above theorems, it suffices to assume that 
the area of $(X,\omega)$ is one. In fact, scaling $(X,\omega)$ by the factor $\frac{1}{\sqrt{A}}$ does not change the ratio $2^{j-\ell}$.

We conclude with a corollary of Theorem~\ref{CountTriangles}, which may be of independent interest.  If a simple cylinder contains the saddle connection $\beta$, then the cylinder can be decomposed into two triangles because cylinders are invariant under $\tau$ by \cite[Lem.~2.2]{LindseyInvCompsHyp}.  Therefore, Theorem~\ref{CountTriangles} provides an upper bound for the number of simple cylinders that contain a slit $\beta$.  In fact, this corollary gives the growth rate for the number of simple cylinders containing $\beta$ on almost every surface because the lower bound will be proven in Theorem~\ref{thm:uniformtostart}.

\begin{corollary}\label{coro:cylinder-upper}
Let $(X, \omega) \in \cH^{hyp}(\kappa)$  and let $\beta$ be an invariant saddle connection on  $(X, \omega)$ of logarithmic scale $\ell$.
There exists $c$ depending only on $\cH^{hyp}(\kappa)$ such that 
for any $j \geq \ell$
the number of simple cylinders of logarithmic scale at most $j$ containing $\ell$ is at most 
$c\cdot  2^{j - \ell}.$
\end{corollary}

\subsection{The Slit Torus $\cH(0,0)$}
\label{SlitTorusUB:subsect}

In order to prove Theorems~\ref{CountAlphaSC:Theorem3} and \ref{CountTriangles} by induction, we will need to establish the base case for each of them -- the slit torus.  Let $\cH(0,0)$ denote the stratum of tori with two marked points.  We observe that every torus in $\cH(0,0)$ is hyperelliptic in the sense that it  can be decomposed into two parallelograms,  each parallelogram admits an involution, and performing both involutions simultaneously yields a well-defined involution on the torus such that quotienting by it results in a sphere.  For this reason, $\cH(0,0)$ will serve as the base case for our induction and $\tau$ will denote the involution as usual.  In this case, the slit $\beta$ will be a marked line segment between the two marked points.  Indeed, this implies $\tau(\beta) = \beta$ because $\tau$ interchanges the two marked points and we will refer to $\beta$ as a saddle connection.

The following is Theorem~\ref{CountTriangles} for slit tori. It suffices to assume that $(X,\omega)$ has unit area. 
 
\begin{proposition}
\label{CrossProductBdCountTorus:Theorem}
Let $(X, \omega) \in \cH_1(0,0)$, and let $\beta$ be a marked saddle connection on $(X, \omega)$ of logarithmic scale $\ell$ connecting the distinct marked points.
Let $\Gamma$  be a collection of saddle connections of logarithmic scale $j$ that are closed loops  joining a marked point to itself that are interiorly disjoint from $\beta$. Then
$$\#\Gamma = \# T((X, \omega) \setminus \beta, j) \leq \max \left\{ 2^{j-\ell}, 1\right\}.$$
\end{proposition}

\begin{proof}
We first observe that every such saddle connection $\gamma$ determines a simple cylinder on the slit torus containing $\beta$.  Therefore, $\gamma$ and $\beta$ form two sides of a triangle and we have the first equality.

By Lemma~\ref{lem:shortL}, there is at most one saddle connection in $\Gamma$ of logarithmic scale less than $\ell-1$.  Therefore, it suffices to assume that $j \geq \ell$. 

The angle between $\gamma$ and $\beta$ is bounded above by $\frac{1}{2^{j + \ell}}$. 
 On a torus of unit area, we have  $|\gamma_1\times\gamma_2|\geq 1$ for any pair of loops. Thus, the saddle connections in $\Gamma$ are angle at least $\frac{1}{2^{2j}}$ apart.  Then in a sector of angle $1/2^{j+\ell}$, there are at most $\left(1/2^{j+\ell}\right)/\left(1/2^{2j}\right) = \frac{2^j}{2^{\ell}}$ saddle connections in $\Gamma$.
\end{proof}

The following is Theorem~\ref{CountAlphaSC:Theorem3} for slit tori.  

\begin{proposition}
\label{CountAlphaSCTorus:Theorem}  Given $C > 1$, let $(X, \omega) \in \cH_1(0,0)$, and let $\beta$ be an invariant saddle connection of logarithmic scale $\ell$ on $(X, \omega)$.  If $j \geq \ell+1$, then there exists a positive constant $c$ depending only on $C$ such that
$$\# A((X, \omega) \setminus \beta, j, C) \leq c({j-\ell})2^{j-\ell}.$$
\end{proposition}

\begin{proof}
We first consider the set of $\alpha \in A((X, \omega) \setminus \beta, j, C)$ such that $|\alpha\times\beta|\leq 2$.  Inside this collection, the number of saddle connections that are $\epsilon$-isolated is at most $$\frac{2^{2j}}{\epsilon^2}\cdot \frac{1}{2^{j + \ell}} = \frac{2^{j - \ell}}{\epsilon^2}.$$  

For those that are not $\epsilon$-isolated, we build subsurfaces (which are necessarily simple cylinders in this setting) with boundaries formed by the collection (cf. Lemma~\ref{SeparatingSCSetForGammaBeta:Lemma}).  Each boundary is of the form $\gamma\cup \tau(\gamma)$, where $\gamma$ has logarithmic scale at most $j$ (up to an additive constant).  Let $\Gamma$ be the collection of all such $\gamma$.  By the assumption that $\alpha$ is not $C$-separated from $\beta$, we have $|\gamma|\geq |\beta|/C$.

By Lemma~\ref{lem:shortL}, there is at most one  $\gamma_0$ such that $|\gamma_0| < \frac{|\beta|}{2}$. The number of $\alpha$ determining $\gamma_0$ is at most 
$C \cdot 2^{j-\ell}$ because the saddle connections $\alpha$ of logarithmic scale $j$ wrap around a cylinder with core curve $\gamma_0$ and therefore pick up multiples of its length.

It remains to consider the case that $|\gamma| \geq |\beta|/2$. 
Partition $[2^{\ell - 1},2^{j+1}]$ into intervals $I_k=[2^{k}, 2^{k+1}]$, $\ell-1\leq k\leq j$.  
It follows from  Proposition~\ref{CrossProductBdCountTorus:Theorem} that 
 for each $k$, there are at most $2^{k-\ell+1}$ possible $\gamma \in \Gamma$ with logarithmic scale $k$. And the number of $\alpha$ determining each $\gamma$ of logarithmic scale $k$ is at most $2^j/2^{k}$ because the saddle connections $\alpha$ of logarithmic scale $j$ wrap around a cylinder with core curve $\gamma$ and therefore pick up multiples of its length. Multiplying $2^j/2^{k}$ by $2^{k-\ell+1}$, we get a bound of  $2^{j-\ell+1}$ such  $\alpha$ for each $k$.  Summing over all values of $k$, we get the total bound to be 
 $$\sum_{k=\ell-1}^{j} \frac{2^{j+1}}{2^{\ell}} = \frac{2^{j+1}}{2^{\ell}} (j - \ell+1)\leq 3 (j - \ell) 2^{j-\ell}.$$

The proof is almost the same for those $\alpha$ such that $|\alpha\times\beta|\geq 2$. 
Since the number of pairwise interiorly disjoint saddle connections on a slit torus is at most five (excluding $\beta$), we may assume that each saddle connection in $A((X, \omega) \setminus \beta, j, C)$ intersects at least one other saddle connection.  Let $\alpha\in A((X, \omega) \setminus \beta, j, C)$ be a saddle connection that intersects $\alpha'\in A((X, \omega) \setminus \beta, j, C)$.  Moreover, $\alpha$ is necessarily invariant under $\tau$ because otherwise, $\alpha \cup \tau(\alpha)$ would bound a cylinder containing $\beta$, which would contradict the unit area assumption.  As stated in the remark following Lemma~\ref{EMEpsCplx:Lemma} there is a locally convex subsurface $Y$ generated by $\alpha$ and $\alpha'$ with diameter at most $2^{j+2}$.  In this case, since $\alpha$ and $\alpha'$ are intersecting invariant saddle connections, $Y$ is necessarily a simple cylinder containing them.  Let $\gamma$ be a boundary saddle connection of this cylinder, which necessarily satisfies $|\gamma| \leq 2^{j+1}$.
In particular, $\gamma\cup\tau(\gamma)$ cuts $(X,\omega)$ into two simple cylinders.  The fact that  $\alpha$ is not $C$-separated from $\beta$ implies that $|\gamma|\geq |\beta|/C$.
The rest of the proof follows as before.
\end{proof}

\subsection{Proofs of Theorems~\ref{CountAlphaSC:Theorem3} and \ref{CountTriangles}}

As stated above the proofs of Theorems~\ref{CountAlphaSC:Theorem3} and \ref{CountTriangles} are by induction.  Propositions~\ref{CrossProductBdCountTorus:Theorem} and \ref{CountAlphaSCTorus:Theorem} have established the base cases of these theorems in $\cH(0,0)$.

\begin{proof}[Proof of Theorem~\ref{CountTriangles}]
The outline is as follows.  Fix $\epsilon > 0$ sufficiently small.  
We will first prove the linear upper bound for the set of $\gamma$ that are $\epsilon$-isolated.
For the more complicated case of non-isolated $\gamma$, we will apply Lemma~\ref{SeparatingSCSetForGammaBeta:Lemma}. The proof then will be divided into two cases depending on the logarithmic scale $j_m$ of the isolated  $\sigma_m$ in $\Sigma_m$ provided by the lemma.  In the case that $j_m$ is sufficiently large, we will use induction to give a linear upper bound on the number of $\gamma$. In the case that $j_m$ is small compared to $j$, we will show that all $\gamma$ in this case are contained in a proper subsurface and then apply the induction for strata of lower dimension.  

Up to scaling by a factor, we may assume that $(X,\omega)$ has unit area.

\paragraph{Isolated Saddle Connections:}
Since $\gamma$ and $\beta$ form a triangle, the unit area assumption implies that $|\gamma \times \beta| \leq 2$, so the angle between $\gamma$ and $\beta$ is bounded above by $\frac{4}{2^{j+ \ell}}$.  
Consider the $\epsilon$-isolated saddle connections   in $T((X, \omega) \setminus \beta, j)$. Their angles  are at least $\epsilon^2 2^{-2j}$ apart.  So in a sector of angle $4/2^{j+ \ell}$, there are at most 
$$\max\left\{ \frac{2^{j+2}}{2^{\ell}\epsilon^2},1\right\}$$ saddle connections in $T((X, \omega) \setminus \beta, j)$. This gives the desired linear bound for the isolated saddle connection $\gamma$, by taking  $c = 4/\epsilon^2$.

\paragraph{Large Area Triangles:} Now we fix $c_1 > 0$ and consider the subset $\Gamma_{c_1} \subset T((X, \omega) \setminus \beta, j)$ consisting of $\gamma$ that form triangles with $\beta$ of area at least $c_1 \epsilon^2$, i.e., $|\gamma \times \beta| \geq c_1 \epsilon^2$.  

We may assume that $\gamma$ is the longer side of the triangle excluding $\beta$. In fact, given a triangle with sides $\beta, \gamma$ and $\rho$,
there is a  one-to-one correspondence between $\gamma$ and $\rho$. Assume that $|\gamma|\leq |\rho|$. 
Then $|\rho|<|\gamma|+|\beta|$.
So proving an upper bound for a collection of $\gamma$
of logarithmic scale at most $j$ can be reduced to proving an upper bound for a collection of $\rho$ of logarithmic scale at most $\log{(2^j+2^\ell)}$. This would change the ratio $\frac{2^j}{2^\ell}$ to  $\frac{2^j}{2^\ell}+1$, which does not affect the conclusion.

We can partition $\Gamma_{c_1}$ into maximal sub-collections 
say $\Gamma_{c_1,i}$ such that any two saddle connections in the same $\Gamma_{c_1,i}$ bound an angle less than 
\begin{equation}\label{equ:largearea}
  \frac{c_1\epsilon^2}{2^{2j+2}}.
\end{equation}
 By Lemma~\ref{lem:smallangle}, if any $\gamma, \gamma' \in \Gamma_{c_1,i}$ interiorly intersect to each other, then we have $|\gamma \times \beta| \leq 2 |\gamma \times \gamma'|$.  This would imply that the angle between $\gamma$ and $\gamma'$ is 
 larger than 
 $$\frac{c_1\epsilon^2}{2|\gamma||\gamma'|},$$
 which contradicts \eqref{equ:largearea}. As a result, the saddle connections in the same sub-collection $\Gamma_{c_1,i}$ are pairwise interiorly disjoint. 
 
 Since the number of pairwise interiorly disjoint saddle connections is bounded from above by some constant depending only on $\cH^{hyp}(\kappa)$, we may assume that each $\Gamma_{c_1,i}$ consists of a single saddle connection. 
 Thus, the saddle connections in $\Gamma_{c_1}$ are $\sqrt{c_1}\epsilon$-isolated,
  so this case is subsumed by the isolated saddle connection case above.

\begin{remark}
%Recall that a translation surface $(X,\omega)$ is a \emph{Veech surface} if $\operatorname{SL}(X,\omega)$, the 
%group of derivatives of affine automorphisms of $(X,\omega)$, 
%is a lattice in $\operatorname{SL}(2,\mathbb{R})$. A Veech surface has dynamical properties similar to the square torus.  
It was shown in \cite[Thm.~1.1]{SmillieWeissCharLattice} that $(X,\omega)$
is a Veech surface if and only if it has ``no
 small triangles'', i.e., the areas of triangles in $(X,\omega)$
 is bounded from below by a positive number.
 So if $(X,\omega)\in \cH^{hyp}_1(\kappa)$ is a Veech surface,
 then it suffices to consider large area triangles and at this stage the proof of Theorem~\ref{CountTriangles} is complete for Veech surfaces.
\end{remark}

\paragraph{Non-Isolated Saddle Connections (Small Area Triangles):} Now assume there is a collection of saddle connections in $T((X, \omega) \setminus \beta, j)$ that are not $\epsilon$-isolated.  For a fixed $c_0$, we define a new set $$\Gamma = T((X, \omega) \setminus \beta, j, c_0\epsilon^2)$$ to be the subset of $T((X, \omega) \setminus \beta, j)$ such that the triangle formed by $\gamma$ and $\beta$ has area at most $c_0 \epsilon^2$.  Again, we assume that $\gamma$ is the longer side of the triangle excluding $\beta$, and then we have $|\gamma|\geq |\beta|/2$. In this way, we observe that the assumptions of Lemma~\ref{SeparatingSCSetForGammaBeta:Lemma} are satisfied for $\Gamma$.

We consider the collections of saddle connections $\Sigma_i$ and the partitions of the subcollections $\Sigma_{i}(j)$ of logarithmic scale $j$ for some $j$ given in  Lemma~\ref{SeparatingSCSetForGammaBeta:Lemma}.  By  Lemma~\ref{SeparatingSCSetForGammaBeta:Lemma}, $m\leq d-2$, and by Lemma~\ref{SeparatingSCSetForGammaBeta:Lemma} Item~\ref{SeparatingSCSetForGammaBeta:Lemma:SigmaPart}, every pair of saddle connections in $\hat\Sigma_{i}(j)$ is either $\epsilon$-isolated or interiorly disjoint.  

As it is stated in Lemma~\ref{SeparatingSCSetForGammaBeta:Lemma}, 
a saddle connection $\gamma$ determines a sequence of size $p$ if and only if the sequence terminates at  some $\sigma_p\in \hat{\Sigma}_{p}(j)$. 
We are able to decompose $\Gamma$ into at most $m$ subsets such that each $\gamma$ in the same subset determines a sequence of some fixed size. 
In the following proof, we just consider those $\gamma\in \Gamma$ such that the sequence associated to $\gamma$ has size $m$. We denote the 
associated sequence by $(\gamma = \sigma_0, \sigma_1,\ldots, \sigma_m)$.   
Those $\gamma\in \Gamma$ that determine sequences of size $p<m$ can be counted  
in the same way. The bounds on the count work identically for every size.  

We note that two collections $\Sigma_{i, k}(j)$ and $\Sigma_{i, k'}(j')$ may determine the same saddle connections in $\Sigma_{i+1}$.  
Also, by Lemma~\ref{SeparatingSCSetForGammaBeta:Lemma} Item~\ref{SeparatingSCSetForGammaBeta:Lemma:SigmaPart}, we form the filtration of subsurfaces from the collections of sequences of saddle connections
$$Y^{(k_1)}_{1}\subset Y^{(k_2)}_2 \subset \cdots \subset Y^{(k_m)}_m.$$
(Observe that the $k_i$ could be distinct because two different collections $\Sigma_{i, k}(j)$ and $\Sigma_{i, k'}(j)$ may determine the same $Y^{(k)}_i$.)

The key to the inductive argument is as follows.  We count saddle connections in $\Sigma_{i,k}(j)$ by considering $\sigma_{i+1}$, which is determined by $\Sigma_{i,k}(j)$.
Assume that $\sigma_{i+1}$ has logarithmic scale $j_{i+1}$.  By Lemma~\ref{TwoBdSC:Lemma}, cutting along $\sigma_{i+1} \cup \tau(\sigma_{i+1})$ results in two hyperelliptic translation surfaces, one of which contains the saddle connections in $\Sigma_{i,k}(j)$.  Since this hyperelliptic translation surface lies in a smaller dimensional stratum, the induction hypothesis can be applied to it.

We fix some vector $\bj = (j_0,j_1,j_2,\ldots,j_m)$, with $j = j_0$, and count the subset of $\gamma\in \Gamma$ such that the sequences of saddle connections associated to $\gamma$ have logarithmic scales specified by $\bj$.  At the end, we will sum over all possible vectors $\bj$ to obtain the total count of all elements in $\Gamma$.

We claim that the sequence of saddle connections associated to $\gamma$ bounds the angle between $\sigma_m$ and $\beta$ by
\begin{equation}
\label{eq:smallangle}
  \frac{4m c_2^{2m}}{2^{2j_m}}+\frac{1}{2^{j + \ell}} \leq 2\max \left\{ \frac{4m c_2^{2m}}{2^{2j_m}},\frac{1}{2^{j + \ell}} \right\}. 
\end{equation}
Indeed, by Lemma~\ref{SeparatingSCSetForGammaBeta:Lemma} Item~\ref{SeparatingSCSetForGammaBeta:Lemma:sigiAngsigip1}, the angle between $\sigma_i$ and $\sigma_{i+1}$ is at most $$\frac{4}{2^{j_i + j_{i+1}}} < \frac{{4c_2^{2m}}}{2^{2j_m}}$$ because $2^{j_{i+1}}/c_2 < 2^{j_i}$, for all $i$ by Lemma~\ref{SeparatingSCSetForGammaBeta:Lemma} Item~\ref{SeparatingSCSetForGammaBeta:Lemma:sigiLengthsigip1}.  Therefore, adding up over $i\leq m$  and adding in the angle between $\gamma$ and $\beta$, we get the desired bound.

We consider the set of saddle connections $\sigma_m\in \Sigma_m(j_m)$, which are either $\epsilon$-isolated or a set of parallel saddle connections, and this exists by Lemma~\ref{SeparatingSCSetForGammaBeta:Lemma} Item~\ref{SeparatingSCSetForGammaBeta:Lemma:sigmIso2}.  There are two cases to consider: either
\begin{itemize}
\item \textbf{Case~I} : $2j_m\geq j+\ell$, or,
\item \textbf{Case~II} : $2j_m < j+\ell$.
\end{itemize}

\paragraph{Proof of Theorem~\ref{CountTriangles} in Case I: $2j_m\geq j+\ell$}
\begin{lemma}
\label{Lemma:Triangles:BigSigmam}
Given the setup of Theorem~\ref{CountTriangles} in the setting of Case I, there exists a constant $c$ depending only the stratum $\cH^{hyp}(\kappa)$ such that 
	\begin{equation*}
		\# \{\gamma\in \Gamma : 2j_m\geq j+\ell\}\leq c\cdot  2^{j-\ell}.
	\end{equation*}
\end{lemma}

\begin{proof}
In order to count the number of saddle connections in this subset of $\Gamma$, we follow Lemma~\ref{SeparatingSCSetForGammaBeta:Lemma}.  We form the associated sequence of saddle connections to each $\gamma \in \Gamma$ as guaranteed by Lemma~\ref{SeparatingSCSetForGammaBeta:Lemma}.  Then we partition those collections of saddle connections as in Lemma~\ref{SeparatingSCSetForGammaBeta:Lemma} Item~\ref{SeparatingSCSetForGammaBeta:Lemma:SigmaPart} and count the number of saddle connections in $\Sigma_{i,k}$ for each $i$ and $k$.

By Lemma~\ref{lem:shortL}, there is at most one saddle connection $\gamma \in \Gamma$ such that $|\gamma|\leq |\beta|/2$.  Therefore, up to multiplying $c$ by $4$, we can always assume that $j\geq \ell$.

We begin with the easiest case - counting the number of saddle connections in $\Sigma_m(j_m)$.  In this case, the Case I assumption implies $$\frac{1}{2^{2j_m}} \leq \frac{1}{2^{j + \ell}},$$ so the $\epsilon$-isolated property on saddle connections  $\sigma_m$ together with Inequality~\eqref{eq:smallangle} implies that there are at most 
\begin{equation}\label{eq:sigma:m}
	\frac{2^{2j_m}}{\epsilon^2 2^{j + \ell}}+\frac{4mc_2^{2m}}{\epsilon^2}\leq \frac{8m c_2^{2m}}{\epsilon^2}\cdot\frac{2^{2j_m}}{2^{j + \ell}}
\end{equation}
saddle connections in $\Sigma_m$ of logarithmic scale $j_m$. 
In the case that all saddle connections in $\Sigma_m(j_m)$ are parallel, the number of such saddle connections is uniformly bounded in terms of genus.  Thus, we can ignore this case because it is bounded by the $\epsilon$-isolated case.

Next we explain how to count the number of saddle connections in $\Sigma_{i,k}$ for any $k$.  This will require the inductive assumption.  Given $\Sigma_{i,k}$, let $\sigma_{i+1}$ be the boundary saddle connection of $Y^{(k)}_{i+1}$ (the subsurface containing $\Sigma_{i,k}$).  Since $\sigma_{i+1} \cup \tau(\sigma_{i+1})$ bounds a subsurface disjoint from $Y^{(k)}_{i+1}$ by Lemma~\ref{TwoBdSC:Lemma}, we can excise the subsurface and identify $\sigma_{i+1}$ with $\tau(\sigma_{i+1})$.
The resulting subsurface  lies in a lower dimensional stratum $\cH^{hyp}(\kappa')$ and contains  $\sigma_{i+1}$ as an invariant slit.

We can apply Theorem~\ref{CountTriangles} (which holds for arbitrary area surfaces)
using the induction hypothesis that $\dim \cH^{hyp}(\kappa') < \dim \cH^{hyp}(\kappa)$ to conclude that there is a  constant $c$ depending on the stratum such that the number of saddle connections in all possible $\Sigma_{i,k}$ that determine $\sigma_{i+1}$ is at most
\begin{equation}
\label{sigmaipl1persigmai:eqn}
c\cdot \frac{2^{j_i}}{2^{j_{i+1}}}.
\end{equation}
If the subsurface has genus one, we apply Proposition~\ref{CountAlphaSCTorus:Theorem} instead.

With these preliminary counts, we combine them to count all saddle connections in $\Gamma$ as follows.  We first continue to fix $\bj$ and count the total number of saddle connections in $\Gamma$ with the property that the logarithmic scales of the saddle connections in their sequences associated to $\gamma$ are specified by $\bj$.  Then we sum over all possible tuples $\bj$ to get the final total.

Recall the convention $j_0 = j$.  Then by \eqref{sigmaipl1persigmai:eqn}, the number of saddle connections in $\Gamma$ with boundary saddle connection $\sigma_1$ is at most 
$$c \cdot 2^{j_0-j_1}.$$  To count the number of saddle connections in $\Gamma$ with boundary saddle connections in $\Sigma_{1,k}$ of logarithmic scale $j_1$ with boundary saddle connection $\sigma_{2}$, we multiply the counts of each.  If we do this iteratively, and in the last step, multiply by the number of elements in $\Sigma_m$ of logarithmic scale $j_m$ computed above in \eqref{eq:sigma:m}, this yields the upper bound,  for some new constant $c$, 
\begin{equation}\label{eq:combine-linear}
{c} 2^{j_0-j_1}\cdots 2^{j_{m-1}-j_m}\cdot\frac{2^{2j_m}} {2^{j_0+\ell}}=c 2^{j_m-\ell}.
\end{equation}

Next we would like to sum this over all possible partitions $\bj$.  Recall that, up to an additive constant, we may assume that $j_i \geq j_{i+1}$.  
From this we conclude that, by summing over all admissible  $\bj$, there is some constant $c$ (which incorporates all of the previous constants) such that
\begin{eqnarray}\label{equ:Gamma:CaseII}
\label{ineq:BoundSum}
\#\Gamma   &\leq& \sum_{\substack{0\leq i\leq m-1\\j_{i+1}\leq j_i}}c 2^{j_m-\ell}\leq c' 2^{j_0-\ell}
\end{eqnarray} 
for a bigger constant $c'$.
Since this bound is  true for each $m\leq d-2$, if we sum over $m\leq d-2$, we get the desired result. 
This finishes the proof of the lemma.
\end{proof}

\paragraph{Proof of Theorem~\ref{CountTriangles} in Case II : $2j_m\leq {j+\ell}$}

\begin{lemma}
\label{aCaseIIExists:lemma}
Let $\Gamma$ be a collection of saddle connections in the setting of Case~II.  Let $\gamma \in \Gamma$ and let $(\gamma=\sigma_0, \sigma_1, \ldots, \sigma_m )$ be the sequence of saddle connections associated to $\gamma$ from Lemma~\ref{SeparatingSCSetForGammaBeta:Lemma}.  For each $i$, let $j_i$ be the logarithmic scale $\sigma_i$.  Then there exists $a \in \{1, \ldots, m-1\}$ such that 
\begin{equation}\label{eq:short}
2j_a\geq j+\ell-1 \quad \text{ and } \quad 2j_{a+1}\leq j+\ell.
\end{equation}
\end{lemma}

\begin{proof}
By Lemma~\ref{lem:shortL}, we may assume $j \geq \ell-1$ because there is at most one saddle connection such that $j < \ell-1$.  We have $2j \geq j + \ell - 1$ and  $2j_m \leq j + \ell$.  There must be a smallest index $a$ where Inequality~\eqref{eq:short} is satisfied.
\end{proof}

We now introduce some notation.
 For $\gamma \in \Gamma$ satisfying Case~II, let $a=a(\gamma)$ be the smallest index guaranteed by Lemma~\ref{aCaseIIExists:lemma}, and let $\sigma_{a+1}=\sigma_{a+1}(\gamma)$ be the saddle connection in the sequence of saddle connections associated to $\gamma$.
As we have shown in Lemma~\ref{SeparatingSCSetForGammaBeta:Lemma}, $\sigma_{a+1}$ is a boundary component of a subsurface
$Y_{a+1}$ (defined in the filtration of subsurfaces associated to $\gamma$). 
Cut the surface $(X,\omega)$   along $\sigma_{a+1}\cup\tau(\sigma_{a+1})$ into two components, denoted by $Y(\sigma_{a+1})$ and $W(\sigma_{a+1})$.  We assume that 
$Y(\sigma_{a+1})$ is the component that
contains  $Y_{a+1}$ and $W(\sigma_{a+1})$ is the component intersecting $Y_{a+1}$ on a set of measure zero. 
By Inequality~\eqref{eq:area:W} in the proof of  Lemma~\ref{SeparatingSCSetForGammaBeta:Lemma},
 the area of $W(\sigma_{a+1})$ is at least $\frac{1}{(d-1)^a}$.
So the area of $Y(\sigma_{a+1})$ is at most 
$$1-\frac{1}{(d-1)^a}.$$

In the following, we shall fix $a$ and some integer partition
$\kappa'$ of $2g'-2$. We consider those $\gamma\in \Gamma$  
satisfying Case~II  such that $a$ is the smallest index satisfying ~\eqref{eq:short} and $Y(\sigma_{a+1})$ defines a 
surface in $\mathcal{H}^{hyp}(\kappa')$, i.e., the topological type of $\sigma_{a+1}(\gamma)$ is fixed. Denote the subset of such saddle connections by $\Gamma_a(\kappa')$. Define
$$\Sigma_{a+1}(\kappa') = \{\sigma_{a+1}(\gamma) \ | \ \gamma \in \Gamma_a(\kappa')\}.$$

The crucial result proven in the lemma below is that unlike elements in  the sets $\Sigma_{i,k}$ in Case~I that are partitioned by logarithmic scale, all of the saddle connections in $\Sigma_{a+1}(\kappa')$ can be made simultaneously short via the Teichm\"uller flow.  The result will be a subsurface $Z$ that contains all of them and which will allow us to build a possibly bigger surface $X_\rho$ in a lower dimensional stratum to which we can apply induction.

\begin{lemma}
\label{lem:onesubsurface}
For any $a$ and $\kappa'$, consider the set $\Gamma_a(\kappa')$ of saddle connections  defined above.  Then there exists a pair of non-invariant saddle connections $\rho$ and $\tau(\rho)$ such that one subsurface $X_\rho$ of $(X,\omega)\backslash\{\rho\cup\tau(\rho)\}$ contains $\beta$ and all $\gamma \in \Gamma_a(\kappa')$.
\end{lemma}

\begin{proof}
We will use ideas similar to those in the proof of Lemma~\ref{SeparatingSCSetForGammaBeta:Lemma}.  First we consider a saddle connection $\sigma_{a+1}$ associated to any of the $\gamma$ in $\Gamma_a(\kappa')$ and bound its angle with $\beta$.  We flow in the stratum by the Teichm\"uller flow to contract $\beta$, and obtain a subsurface with small area as in the proof of Lemma~\ref{SeparatingSCSetForGammaBeta:Lemma}.  From this we conclude the uniqueness of a large area surface as we did in the proof of Lemma~\ref{SeparatingSCSetForGammaBeta:Lemma}.

We now proceed with the detailed proof.
Let $\gamma \in \Gamma_a(\kappa')$ and let $\sigma_{a+1} = \sigma_{a+1}(\gamma)$ be the saddle connection defined above from the sequence of saddle connections associated to $\gamma$.  We first observe that, by applying the Teichm\"uller flow renormalization, we can make both $\beta$ and $\sigma_{a+1}$
 simultaneously short.  
In fact,  by Lemma~\ref{SeparatingSCSetForGammaBeta:Lemma} Item~\ref{SeparatingSCSetForGammaBeta:Lemma:sigiAngsigip1}, the angle between $\beta$ and $\sigma_{a+1}$ can be bounded above in terms of logarithmic scales by summing the angles of the successive saddle connections.  This yields
	\begin{eqnarray*}
		&&\frac{4\epsilon^2}{2^{j + \ell}}+\frac{4\epsilon^2}{2^{j + j_1}}+\frac{4\epsilon^2}{2^{j_1 + j_2}}+\cdots + \frac{4\epsilon^2}{2^{j_a + j_{a+1}}}\\
		&\leq & \frac{4\epsilon^2}{2^{j+\ell}}+\frac{4\epsilon^2c_2^{2(a+1)}}{2^{j_a + j_{a+1}}}+\frac{4\epsilon^2 c_2^{2a}}{2^{j_a + j_{a+1}}}+\cdots + \frac{4\epsilon^2}{2^{j_a + j_{a+1}}}\\
		&\leq &\frac{4\epsilon^2}{2^{j+\ell}}+\frac{4(a+1)c_2^{2(a+1)}\epsilon^2 }{2^{j_a + j_{a+1}}},
	\end{eqnarray*}
where the first inequality follows from Lemma~\ref{SeparatingSCSetForGammaBeta:Lemma} Item~\ref{SeparatingSCSetForGammaBeta:Lemma:sigiLengthsigip1}.
Rotate the surface so that $\beta$ is vertical.  Apply the Teichm\"uller flow $g_t$ for
$$e^t=\frac{2^{(j+\ell)/2}}{\epsilon}.$$ Let $|\cdot|_{t}$, $|\cdot|_{h,t}$ and $|\cdot|_{v,t}$ denote the lengths of a saddle connection (resp. its horizontal component and vertical component) after applying $g_t$.  
Then after applying $g_t$, the lengths of the horizontal and vertical components of $\sigma_{a+1}$ are respectively:
\begin{eqnarray*}
|\sigma_{a+1}|_{h,t}&\leq& 2^{j_{a+1}}\left(\frac{4\epsilon^2}{2^{j+\ell}}+\frac{4(a+1)c_2^{2(a+1)}\epsilon^2 }{2^{j_a + j_{a+1}}}\right) \frac{2^{(j+\ell)/2}}{\epsilon},\\
|\sigma_{a+1}|_{v,t}&\leq&  2^{j_{a+1}}\frac{\epsilon}{2^{(j+\ell)/2}}.
\end{eqnarray*}
Applying the inequalities in \eqref{eq:short} yields
\begin{eqnarray*}
  |\sigma_{a+1}|_t &\leq& |\sigma_{a+1}|_{h,t} + |\sigma_{a+1}|_{v,t}  \\
  &\leq& 8(a+1) c_2^{2(a+1)}\epsilon.
\end{eqnarray*} 
Denote $\epsilon'=8(a+1)c_2^{2(a+1)}\epsilon.$

We observe that the above calculation bounding $|\sigma_{a+1}(\gamma)|$ holds \emph{for any} $\gamma \in \Gamma_a(\kappa')$. 
  Let $Z$ be the subsurface of $(X, \omega)$ generated by all $\{\sigma_{a+1}(\gamma), \tau(\sigma_{a+1}(\gamma))\} \subset \Sigma_{a+1}(\kappa')$, i.e., 
$$\bigcup_{\gamma \in \Gamma_a(\kappa')} \left\{ \sigma_{a+1}(\gamma), \tau(\sigma_{a+1}(\gamma)) \right\} \subset Z$$ (cf. Lemma~\ref{EMEpsCplx:Lemma}).  
Note that any two distinct $\sigma_{a+1}(\gamma)$ and $\sigma_{a+1}(\gamma')$ must interiorly intersect. For otherwise, it is impossible that $\sigma_{a+1}(\gamma)$ and $\sigma_{a+1}(\gamma')$ determine subsurfaces in the same stratum 
$\cH^{hyp}(\kappa')$. As a result, $Z$ is a connected subsurface. 
 Since $Z$ is generated by saddle connections of length at most $\epsilon'$, the area of $Z$ is at most $c_1 (\epsilon')^2$ by Lemma~\ref{EMEpsCplx:Lemma}.
 
 Now for an arbitrarily chosen $\gamma\in\Gamma_a(\kappa')$  we have the subsurface $Y(\sigma_{a+1})$, where $\sigma_{a+1}=\sigma_{a+1}(\gamma)$. The pair of saddle connections $\sigma_{a+1}, \tau(\sigma_{a+1})$ is shared by 
 $Y(\sigma_{a+1})$ and $Z$. 
 The union $Y(\sigma_{a+1}) \cup Z$ has area at most $$1-\frac{1}{(d-1)^a}+c_1(\epsilon')^2,$$
 which is strictly less than $1$ when $\epsilon$ is sufficiently small. 
There are finitely many complementary components of $Y(\sigma_{a+1})\cup Z$ (the number depends on the stratum). So by removing one of the complementary components of $Y(\sigma_{a+1})\cup Z$ with largest area, we obtain a subsurface $X_\rho$ (with boundary $\rho\cup\tau(\rho)$) that contains $Y(\sigma_{a+1})\cup Z$.
The area of $X_\rho$ is strictly less than $1$.

We claim that
for all $\gamma' \in \Gamma_a(\kappa')$, $\gamma' \subset X_\rho$.
To see this, denote by $\sigma_{a+1}'=\sigma_{a+1}(\gamma')$ the saddle connection in $\Sigma_{a+1}(\kappa')$ determined by $\gamma'$.
Then both $\beta$ and $\sigma_{a+1}'$ are  contained in the subsurface $X_\rho$ (because the other component is separated from $\beta$ and $\sigma_{a+1}'$ by $\rho\cup\tau(\rho)$). Again, 
$\gamma'$ and $\rho$ are separated by $\sigma_{a+1}'\cup\tau(\sigma_{a+1}')$,
and only one of the components contains $\beta$ and $\gamma'$. 
This situation is illustrated in Figure \ref{fig:subsurface}.
This shows that $\gamma' \subset X_\rho$.
\end{proof}

\begin{figure}[htb]
\begin{center}
 \begin{tikzpicture}[scale=0.9]
		\draw[fill=lightgray] (0,0) -- (-1.5,0) -- (-2.5,-3) -- (-1,-3) -- cycle;
	\node at (-1,-0.3) {$\beta$};
\draw[->, >= triangle 90] (0,0) -- (-0.8,0);
\draw[->, >= triangle 90] (-1,-3) -- (-1.8,-3);
\draw[->>] (-2.5,-3) -- (-2,-1.5);
\draw[->>] (-0,-4) -- (0.5,-2.5);
		\draw[fill=lightgray] (0,0) -- (-1,-3) -- (0,-4) -- (1,-1) -- cycle;
\draw[-latex] (-1,-3) -- (-0.5,-3.5);
\draw[-latex] (2,1) -- (2.5,0.5);
     \draw[fill=lightgray] (0,0) -- (1,-1) -- (3,0) -- (2,1) -- cycle;
     \draw[ pattern={north west lines}]  (0,0) -- (1,-1) -- (3,0) -- (2,1) -- cycle;
      \draw[-latex] (2,1) -- (2.5,0.5);
            \draw[->|] (1,-1) -- (2,-0.5);
            \draw[->|] (-0.5,2) -- (0.5,2.5);
     \draw[fill=lightgray] (0,0) -- (2,1) -- (1.5,3) -- (-0.5,2) -- cycle;
     \draw[ pattern={north west lines}] (0,0) -- (2,1) -- (1.5,3) -- (-0.5,2) -- cycle;
       \draw[->] (2,1) -- (1.75,2);
       \draw[->] (-2,2) -- (-2.25,3);
      \draw (0,0) -- (-0.5,2) -- (-2.5, 4) -- (-2,2) -- cycle;
        \node at (0,0.7) {$\rho$};
           \node at (2.3,1.7) {$\tau(\rho)$};
       \draw[-{> >[sep] >}] (0,0) -- (-1,1);
        \draw[-{> >[sep] >}] (-0.5, 2) -- (-1.5,3);
        \node at (-1.7,-3.4) {$\beta$};
		\node at (-0.3,-1.6) {$\gamma'$};
        \node at (1.2,-2.6) {$\tau(\gamma')$};
          \node at (0.8,-0.3) {$\sigma_{a+1}'$};
           \node at (-1,-3.8) {$\tau(\sigma_{a+1}')$};
           \node at (-2.5,-1.2) {$\tau(\gamma')$};
            \node at (3, 0.8) {$\tau(\sigma_{a+1}')$};
	\end{tikzpicture}
\end{center}
	\caption{The shaded domain denotes the subsurface $Z$. The subsurface $X_\rho$ is 
 filled by color.}
\label{fig:subsurface}
\end{figure}

We continue the proof of Theorem~\ref{CountTriangles} for Case~II.  For each $\Gamma_a(\kappa')$, consider the subsurface $X_\rho$ from Lemma~\ref{lem:onesubsurface}. Note that identifying $\rho$ and $\tau(\rho)$ for $X_\rho$ yields a slit surface in a stratum of strictly lower dimension. Here we view $\beta$ as the slit. The theorem now follows by the induction assumption that Theorem~\ref{CountTriangles} is true for strata of strictly lower dimension.  Finally, we multiply by the number of values $\kappa'$ and $a$ can take, which are both bounded in terms of genus, and this gives us the desired bound.

We now summarize the above argument. With respect to the collection of saddle connections $\Gamma$, we have bounded the isolated saddle connections as well as those that form large area triangles.  For those that are not $\epsilon$-isolated and form small area triangles, we associated sequences of saddle connections to each $\gamma$ and formed the partition of $\Gamma$ and its auxiliary saddle connections.  We split the collection into two collections depending on the logarithmic scale of the $\epsilon$-isolated saddle connections at the ends of the sequences of saddle connections associated to $\gamma$.  For the isolated saddle connections with large logarithmic scale in Case I, we used the induction hypothesis that  Theorem~\ref{CountTriangles} holds for strata of strictly lower dimension with Proposition~\ref{CountAlphaSCTorus:Theorem} forming the base case to achieve the desired upper bound.  For the isolated saddle connections with small logarithmic scale in Case~II, we applied the induction hypothesis that Theorem~\ref{CountTriangles} holds for strata of strictly lower dimension with Proposition~\ref{CrossProductBdCountTorus:Theorem} forming the base case to achieve the desired upper bound.  Thus, the proof of Theorem~\ref{CountTriangles} is complete. 
\end{proof}

\paragraph{Proof of Theorem~\ref{CountAlphaSC:Theorem3} by induction:}
Having shown that if Theorems~\ref{CountTriangles}  holds for all strata of dimension strictly less than $d$, then Theorem~\ref{CountTriangles} holds for strata of dimension $d$, we now proceed to show that Theorem \ref{CountAlphaSC:Theorem3} also holds for strata of dimension $d$.  

Before we begin, we highlight a change in perspective that is crucial to understanding the proof.  To count the non-isolated saddle connections in $\Gamma$ in Theorems~\ref{CountTriangles}, we used Lemma~\ref{SeparatingSCSetForGammaBeta:Lemma} to build sequences of saddle connections related to subsurfaces lying in strata of decreasing dimension until we obtained a collection of isolated saddle connections on which to base our count.  In what follows, we will be given a saddle connection $\alpha$ interiorly disjoint from $\beta$ with no information as to where it lies on the surface.  In order to relate it to $\beta$, we will repeatedly use Lemma~\ref{lem:triangulation} to find desirable saddle connections that ultimately form a ``triangulation between $\alpha$ and $\beta$.''  We will still call the intermediate saddle connections $\sigma_i$.  However, their properties will differ from the $\sigma_i$ appearing in Lemma~\ref{SeparatingSCSetForGammaBeta:Lemma}.  Indeed, we will \emph{not} use Lemma~\ref{SeparatingSCSetForGammaBeta:Lemma} in the proof of Theorem~\ref{CountAlphaSC:Theorem3}, and we emphasize this here to avoid confusion below with the similar setup and notation.

\begin{proof}[Proof of Theorems~\ref{CountAlphaSC:Theorem3} and \ref{CountAlphaSC:Theorem3cyl}]
Since the number of pairwise interiorly disjoint saddle connections is bounded from above by a constant depending only on the stratum $\cH(\kappa)$, we may assume that each saddle connection in $A((X, \omega) \setminus \beta,j,C)$ intersects at least one other saddle connection in $A((X, \omega) \setminus \beta,j,C)$.  Let $\alpha\in A((X, \omega) \setminus \beta,j,C)$ be a saddle connection that intersects $\alpha'\in A((X, \omega) \setminus \beta,j,C)$.  In the remark following Lemma~\ref{EMEpsCplx:Lemma}, the smallest locally convex subsurface $Y$ filled by $\alpha$ and $\alpha'$ has diameter at most $2^{j+1}$.  If $\beta$ is on the boundary of $Y$, we apply Lemma~\ref{lem:triangulation} to $\alpha$ and $\beta$ to find a non-invariant saddle connection $\sigma_1$ that either bounds a triangle with $\alpha$ and $\beta$, or together with $\tau(\sigma_1)$ separates $\alpha$ from $\beta$ and bounds a triangle with $\alpha$,   and has length at most $c(2^{j+1}+|\alpha|+|\beta|)\leq 4c2^{j+1}$.
Note that to apply the lemma, we take $\rho$ to be $\beta$ and $\alpha$ to be $\gamma$.

If $\beta$ is not on the boundary of $Y$, then there is some $\beta'$ which together with $\tau(\beta')$ is on the boundary of $Y$ and separates $\alpha$ from $\beta$.  We identify $\beta'$ and $\tau(\beta')$ so that $\beta'$ becomes an invariant slit on $Y$.  Then either $\alpha$ forms a triangle with $\beta'$ or if not, we again apply Lemma~\ref{lem:triangulation} to find $\sigma_1$ that separates $\alpha$ from $\beta'$; hence, also separates $\alpha$ from $\beta$ and bounds a triangle with $\alpha$. Its length is bounded as above.   In summary:
    \begin{itemize}
  	\item Either $\{\alpha,\sigma_1,\beta\}$ bounds a triangle, or $\sigma_1\cup\tau(\sigma_1)$ topologically separates $\alpha$ from $\beta$ and $\{\alpha,\sigma_1\}$ are two sides of a triangle;
  	\item $|\sigma_1|\leq c 2^{j+1}$, where $c$ is redefined from above.
  \end{itemize} 

We partition $A((X, \omega) \setminus \beta,j,C)$ into two subsets: $A_0^{tri}(\beta,j,C)$ consisting of $\alpha$ such that $\{\alpha,\sigma_1,\beta\}$ bounds a triangle and the complement $$A_0^{sep}(\beta,j,C):=A((X, \omega) \setminus \beta,j,C) - A_0^{tri}(\beta,j,C).$$    
  We will apply Theorem \ref{CountTriangles} to count $A_0^{tri}(\beta,j,C)$  so we only need to consider $A_0^{sep}(\beta,j,C)$.

Consequently, let $A_1(\beta,j_1,C)$ be the set of $\sigma_1$'s of logarithmic scale $j_1\leq j+1 + \log c$ defined for $\alpha\in A_0^{sep}(\beta,j,C)$. Note that $\alpha$ is not $C$-separated from $\beta$,  so for each $\sigma_1$, we have $|\sigma_1|\geq |\beta|/C$.

We apply this operation inductively for $i \geq 1$.  Consider the two subsurfaces of $X$ cut out by $\sigma_i\cup\tau(\sigma_i)$.  Identifying $\sigma_i$ and $\tau(\sigma_i)$ for each of them  gives two slit surfaces $X_i\in\cH(\kappa_i)$ that contains $\alpha$ and $X_i'\in\cH(\kappa_i')$ that contains $\beta$ with  $$\dim \cH(\kappa_i) < \dim \cH(\kappa) \quad \text{ and} \quad \dim\cH(\kappa_i')<\dim \cH(\kappa).$$
 
If $X_i'$ is not a simple cylinder, then for each $\sigma_i\in A_i(\beta,j,C)$,  repeating the above operation for $X_i'$ and $\sigma_i$ yields a non-invariant saddle connection $\sigma_{i+1}$ such that 
      \begin{itemize}
  	\item Either $\{\sigma_i,\sigma_{i+1},\beta\}$ bounds a triangle, or $\sigma_{i+1}\cup\tau(\sigma_{i+1})$ topologically separates $\sigma_i$ from $\beta$ and $\{\sigma_i,\sigma_{i+1}\}$ are two sides of a triangle;
  	\item $|\sigma_{i+1}|\leq c2^{j_i+1}$.
  \end{itemize}
  As before  divide $A_i(\beta,j_i,C)$  into $A_{i}^{tri}(\beta,j_i,C)$ and $A_{i}^{sep}(\beta,j_i,C)$, where we include the possibility that $X_i'$ is a cylinder in the set  $A_{i}^{tri}(\beta,j_i,C)$.   As before we only continue with $\sigma_{i+1}$ in the case $\sigma_i\in A_{i}^{sep}(\beta,j_i,C)$.  Then let $X_{i+1}$ (resp. $X_{i+1}'$)  be the slit surface obtained by identifying $\sigma_{i+1}$ and $\tau(\sigma_{i+1})$ on the subsurface of $X$ that is cut out by $\sigma_{i+1}\cup\tau(\sigma_{i+1})$ and contains $\alpha$ (resp. $\beta)$. So $X_{i+1}\in\cH(\kappa_{i+1})$ and $X_{i+1}'\in  \cH(\kappa_{i+1}') $ with $$\dim \cH(\kappa_{i+1})>\dim\cH(\kappa_i)\quad \text{ and} \quad \dim \cH(\kappa_{i+1}')<\dim\cH(\kappa_i').$$
  
Since the dimension of the resulting stratum $\dim \cH(\kappa_i')$ decreases by at least one at each step, the algorithm will terminate after $m\leq \dim\cH(\kappa)-2=d-2$ steps where either $X_m'$ is a simple cylinder or $A_{m}^{sep}(\beta,j_m,C)=\emptyset$.  In the case where we consider cylinders or non-invariant saddle connections, the algorithm will terminate in at most $d-3$ steps because the subsurface $X_1$ cut out by $\sigma_1\cup\tau(\sigma_1)$ above cannot be a simple cylinder. This means that the dimension drop $\dim\cH(\kappa)-\dim\cH(\kappa_1')$ at the first step is at least two.  In both cases, the algorithm proceeds identically by losing at least one dimension at each iteration.  

In this way, we have a sequence of decompositions of saddle connections:
  \begin{eqnarray*}
  	A((X, \omega) \setminus \rho,j,C) &=& A_{0}^{tri}(\beta,j,C) \sqcup A_{0}^{sep}(\beta,j,C),\\
  	 A_1(\beta,j_1 ,C) &=& A_{1}^{tri}(\beta,j_1,C) \sqcup A_{1}^{sep}(\beta,j_1,C),\\
  	 &\cdots& \\
  A_m(\beta,j_m ,C) &=& A_{m}^{tri}(\beta,j_m,C)\sqcup A_{m}^{sep}(\beta,j_m,C) \\  &=& A_{m}^{tri}(\beta,j_m,C),
  \end{eqnarray*}
  and a family of $m$-tuples $ (\sigma_0=\alpha,\sigma_1,\cdots,\sigma_m)$ of logarithmic scales $(j_0=j,j_1,\cdots, j_m)$ satisfying
        \begin{itemize}
  	\item $\sigma_m\in A_m(\beta,j_m,C)$, and either $X_m'$ is a simple cylinder or $\{\sigma_{m-1},\sigma_{m},\beta\}$ bounds a triangle,
  	\item  For $i\leq m-1$, $\sigma_i\in A_{i}^{sep}(\beta,j_i,C)$:  $\sigma_{i}\cup\tau(\sigma_{i})$ topologically separates $\sigma_{i-1}$ from $\beta$ and $\{\sigma_{i+1},\sigma_i\}$ are two sides of a triangle;
  	\item $j_{i+1}\leq j_i+1 + \log c$ and $j_i\geq \ell-\log C$ for $i\leq m-1$, where $\ell$ is the logarithmic scale of $\beta$.
  \end{itemize}
  
Note that counting saddle connections in $A((X, \omega) \setminus \beta,j,C)$ can be reduced to counting the total number of $m$-tuples $(\sigma_0=\alpha,\sigma_1,\cdots,\sigma_m)$.
For each $\sigma_{i+1}$ of logarithmic scale $j_{i+1}$,  consider the slit surface $X_{i+1}$ with $\sigma_{i+1}$ being  the slit. The pair $\{\sigma_i,\sigma_{i+1}\}$ are two sides of a triangle on $X_{i+1}$.  Hence, by Theorem \ref{CountTriangles}, the number of $\sigma_i\in A_{i}(\beta,j_i,C)$ is at most
$$c\cdot \frac{2^{j_i}}{2^{j_{i+1}}}.$$
Let $\bj = (j_0, \ldots, j_m)$ be a tuple of logarithmic scales.  We note that for each $i$, the inequality $$\ell-\log C\leq j_{i+1}\leq j_i+1 + \log c$$ is satisfied.  So, summing over all partitions, we see that the number of $(m+1)$-tuples is at most
\begin{eqnarray*}
 && \sum_{\bj} c^m\cdot \frac{2^{j_0}}{2^{j_{1}}}\cdot \frac{2^{j_1}}{2^{j_{2}}}\cdots \frac{2^{j_{r-1}}}{2^{j_{m}}} \\
 &=&\sum_{\bj} c^m\cdot \frac{2^{j_0}}{2^{j_{m}}}\\
  &	\leq& c^m (j-\ell+\log (cC))^m 2^{j-\ell+\log (cC)} \\&=&c^{m+1} C(j-\ell+\log (cC))^m 2^{j-\ell}.
\end{eqnarray*}
Since $m \leq d-2$,
\begin{eqnarray*}
&& \# A((X, \omega) \setminus \rho,j,C) \\
 &\leq & c^{m+1} C(j-\ell+\log (cC))^m 2^{j-\ell} \\
  	&\leq &c^{d-1} C(j-\ell+\log (cC))^{d-2} 2^{j-\ell}\\
  	&\leq & c'(j-\ell)^{d-2} 2^{j-\ell}.
  	\end{eqnarray*}
Thus, the proof of Theorem~\ref{CountAlphaSC:Theorem3} is completed.

In the case of simple cylinders and non-invariant saddle connections, the inequality $m \leq d-3$ above implies that the last two lines read 
$$c^{d-2} C(j-\ell+\log (cC))^{d-3} 2^{j-\ell} \leq c'(j-\ell)^{d-3} 2^{j-\ell},$$
where the constants still suffice.  This finishes the proof of Theorem~\ref{CountAlphaSC:Theorem3cyl}.
\end{proof}

\subsection{The Upper Bound in Theorem~\ref{MainSummaryThm}}

We now show Theorem~\ref{CountAlphaSC:Theorem} gives the upper bound in Theorem~\ref{MainSummaryThm}.

\begin{proof}[Proof of the Upper Bound in Theorem~\ref{MainSummaryThm}]
Fix some sufficiently large $C$.  First, if there are no saddle connections in $A((X, \omega) \setminus \beta,L)$ that are $C$-separated from $\beta$, then we are done.
Thus, assume that there are saddle connections that are $C$-separated from $\beta$.

Let $\sigma_1$ be a non-invariant systole of $(X, \omega)\setminus \beta$, i.e.,
 a shortest (not necessarily unique) saddle connection on  $(X, \omega)\setminus \beta$ that is not invariant under $\tau$.   Then $\sigma_1 \cup \tau(\sigma_1)$ separates $(X, \omega)\setminus \beta$ into two components $X_1$ and $ X_1'$ by Lemma~\ref{TwoBdSC:Lemma}, where $X_1$ is the component that contains $\beta$.  Since $\sigma_1$ is a non-invariant systole, there is no saddle connection $\alpha\in A((X, \omega) \setminus \beta,L)$ that is contained in $X_1'$ and $C$-separated from $\sigma_1$. After gluing $X_1'$ along $\sigma_1$ and $\tau(\sigma_1)$ to form a closed surface, which we abuse notation and denote by $X'_1$ as well and belongs to some hyperelliptic
component of lower dimension $d'<d$ by Lemma~\ref{TwoBdSC:Lemma}, $\sigma_1$ becomes a marked saddle connection on
$X'_1$ that is invariant under the hyperelliptic involution.
Because $\sigma_1$ is shorter than any non-invariant saddle connection on $X'_1$, nothing on $X'_1$ can be $C$-separated from it, so we apply Theorem~\ref{CountAlphaSC:Theorem} to $X'_1$ to conclude that the number of saddle connections on $X'_1$ is bounded by $$c\frac{L}{|\sigma_1|}\left(\log \frac{L}{|\sigma_1|}\right)^{d'-2}.$$   Since $d'-2<d-2$, there exists $L_0$ depending on $(X,\omega)$ such that for $L\geq L_0$, we have
$$c\frac{L}{|\sigma_1|}\left(\log \frac{L}{|\sigma_1|}\right)^{d'-2}\leq c\frac{L}{|\beta|}\left(\log \frac{L}{|\beta|}\right)^{d-2}.$$

If there are remaining saddle connections not contained in $X'_1$ that are $C$-separated from $\beta$, we will recursively construct a sequence of (non-invariant) saddle connections bounding subsurfaces such that each remaining saddle connection $\alpha$ is contained in one such subsurface and not $C$-separated from the boundary of that subsurface so that Theorem~\ref{CountAlphaSC:Theorem} can be applied to count it.  

We handle this by a recursive argument. For $i \geq 2$, we find a non-invariant systole $\sigma_i \subset X_{i-1}$ and use Lemma~\ref{TwoBdSC:Lemma} to decompose $(X, \omega)$ into $X_i$ and $X'_i$ as above with boundaries $\sigma_i \cup \tau(\sigma_i)$, containing $\beta$ and disjoint from $\beta$, respectively.  Note that for each $i$, $X_i \subset X_{i-1}$ and $X'_i \supset X'_{i-1}$.  We will prove that if $i$ is the smallest value such that $\alpha \subset X'_i$, then $\alpha$ is not $C$-separated from the boundary $\sigma_i$ by taking $C$ sufficiently large.

Note that any $\alpha\in A((X, \omega) \setminus \beta,L)$ in $X_i'\setminus X_{i-1}'$ is not $C$-separated from $\sigma_i$ because $X_i'\setminus X_{i-1}' \subset X_{i-1}$ and $\sigma_i$ is a systole of $X_{i-1}$.

We make the following claim.

\begin{claim} There are no saddle connections contained in $X_i'$ that are $C$-separated from $\sigma_i$ and cut across $\sigma_{i-1} \cup \tau(\sigma_{i-1})$.  
\end{claim}

\begin{proof}[Proof of the claim] We proceed by contradiction and assume that there is such a saddle connection $\alpha$ with a saddle connection $\rho_0$ that $C$-separates $\alpha$ from $\sigma_i$.  In this case, $\rho_0$ must be transverse to $\sigma_{i-1} \cup \tau(\sigma_{i-1})$ because otherwise, $\rho_0$ would be the systole on $X_{i-1}$.  Recall that for each $j < i$, $\sigma_j \cup \tau(\sigma_j)$ separates the surface.  Therefore, there is a minimal such $j < i$ such that $\rho_0$ intersects $\sigma_j \cup \tau(\sigma_j)$ and is contained in $X'_i \setminus X'_{j-1} \subset X_{j-1}$, where we define $X'_0 = \emptyset$.  We observe that
$$\frac{|\sigma_i|}{C} \geq |\rho_0| \geq |\sigma_j|,$$
where the first inequality follows from the $C$-separated assumption and the second inequality follows from the assumption that $\sigma_j$ is a non-invariant systole on $X_{j-1}$.

For this minimal $j$, we combine $\rho_0$ with $\sigma_j$ using \cite[Prop.~6.2]{EskinMasurAsymptForms} to obtain a saddle connection $\rho_1 \subset X'_i \setminus X'_j \subset X_j$.  We claim that $\rho_1$ is non-invariant because $\rho_1$ is a boundary saddle connection of the subsurface generated by $\rho_0$ and $\sigma_j$ and we choose $\rho_1 \cup \tau(\rho_1)$ so that it topologically separates $\rho_j$ from $\rho_{j+1}$.  Since the subsurface can be taken to be invariant under the hyperelliptic involution, the boundary saddle connections are necessarily non-invariant. Therefore, the saddle connection $\rho_1$ satisfies
$$|\rho_1| \leq |\rho_0| + |\sigma_j| \leq \frac{2|\sigma_i|}{C}.$$
We continue this process so that for each $k \leq i-j$, we combine $\rho_k \subset X'_i \setminus X'_{j+k-1} \subset X_{j+k-1}$ with $\sigma_{j+k}$ to get a non-invariant saddle connection $\rho_{k+1}$ satisfying
$$|\rho_{k+1}| \leq |\rho_k| + |\sigma_{j+k}| \leq \frac{k|\sigma_i|}{C} + \frac{|\sigma_i|}{C} = \frac{k+1}{C}|\sigma_i|.$$
Continuing this process until $k = i-j$, we obtain a non-invariant saddle connection $\rho_{i-j+1} \subset X_{i-1}$ such that when $C > i \geq i-j$, $\rho_{i-j+1}$ is strictly shorter than a non-invariant systole $\sigma_i$ of $X_{i-1}$, which is a contradiction.  It remains to be seen that $i$ is bounded in terms of genus to conclude that our choice of $C$ depends only on genus.
\end{proof}

Therefore, for every $\alpha\in A((X, \omega) \setminus \beta,L)$, there exists an $i$ such that $\alpha$ is not $C$-separated from $\sigma_i\cup\tau(\sigma_i)$.  Applying Theorem~\ref{CountAlphaSC:Theorem} again gives a bound for the number of saddle connections $\alpha$ contained in $X_i'$ that are not contained in $X_j'$ in terms of $|\sigma_i|$, and as above, taking $L_0$ sufficiently large, we have the desired bound for counting $\alpha\subset X_i'$.

We now show that this process terminates.  Indeed, for each $i$, $X_{i+1}$ lies in a stratum of dimension strictly less than the dimension of the stratum containing $X_i$.  Therefore, the process must terminate after at most $i \leq d-2$ steps.  In particular, our choice of $C$ in the contradiction proof above can be chosen in terms of $d$.  This finishes the proof of the upper bound in  Theorem~\ref{MainSummaryThm}.
\end{proof}

We have verified the upper bound  in  Theorem~\ref{MainSummaryThm}.
The lower bound is deferred to the next section. Assuming Theorem~\ref{MainSummaryThm}, we sketch a proof of Theorem \ref{MainSummaryThm-noninvariant} below.

\begin{proof}[Proof of Theorem \ref{MainSummaryThm-noninvariant}] 
This follows immediately from Theorem \ref{MainSummaryThm} and Lemma~\ref{TwoBdSC:Lemma}.  By Lemma~\ref{TwoBdSC:Lemma}, $\beta\cup\tau(\beta)$ decompose the surface into two components.  If we excise $\beta\cup\tau(\beta)$, take closures of the resulting open surfaces and identify the resulting boundaries to form a pair of disjoint translation surfaces, then each one will necessarily lie in a (hyperelliptic) stratum of lower dimension, and the identified boundary curves will serve as slits on each surface.  We can apply the first two items from Theorem~\ref{MainSummaryThm} to conclude the growth rate 
$$ \frac{L}{|\beta|} \left(\log \frac{L}{|\beta|}\right)^{d'-2},$$
where $d'$ is the maximum of the dimensions of the above strata of lower dimension. 
\end{proof}

A similar result holds if we assume that $\beta$ is a slit defined by a union of pairwise interiorly disjoint saddle connections, invariant under $\tau$.

\section{Lower Bound}
\label{LowerBound:Section}

We now prove the second part of Theorem \ref{MainSummaryThm}, which we restate as follows.

\begin{theorem}
\label{thm:lower-general}
Let $\cH^{hyp}(\kappa)$ be a stratum with $d = dim_\bC\,\cH^{hyp}(\kappa)$.  Then there exists $c > 0$ such that for $\mu$-almost every surface $(X,\omega)\in \mathcal{H}_1^{hyp}(\kappa)$ and for any invariant saddle connection $\beta$ of logarithmic scale $j_0$ on $(X,\omega)$, there is a constant $L_0$ depending on $(X,\omega)$ and $j_0$ such that for  any $L > L_0$, the  number of saddle connections of length at most $L$ interiorly disjoint from $\beta$ is at least
$$c\frac{L}{|\beta|}\left(\log \frac{L}{|\beta|}\right)^{d-2}.$$
\end{theorem}

The first step of the proof of Theorem~\ref{thm:lower-general} is to count cylinders on $(X,\omega)$ that contain $\beta$.  This is accomplished in Theorem~\ref{thm:uniformtostart:genstatement}.  We state a special case of Theorem~\ref{thm:uniformtostart:genstatement}, which serves as the base case of the induction.

\begin{theorem}
\label{thm:uniformtostart}
Let $\cH^{hyp}(\kappa)$ be a stratum with $d = dim_\bC\,\cH^{hyp}(\kappa)$.  Then there exists $c > 0$ such that for $\mu$-almost every surface $(X,\omega)\in \mathcal{H}_1^{hyp}(\kappa)$ and for an invariant slit $\beta$ of logarithmic scale $j_0$ on $(X,\omega)$, there exists $L_0$ depending on $(X,\omega)$ and $j_0$ such that for every $L\geq L_0$, the number of simple cylinders $\gamma$ of length less than $L$ containing $\beta$ is at least $c \frac{L}{|\beta|}$.  As in Theorem~\ref{thm:lower-general}, the coefficient $c>0$ only depends on the stratum.
\end{theorem}

\begin{remark}\label{rem:cylinder:genusone}
We do not know if the result holds for every surface, cf. the discussion following Theorem~\ref{CBRank:hyp}.  Nevertheless, as noted in the introduction, the result does hold for \emph{every} Veech surface as explained in the remark at the beginning of Section~\ref{sec:bad-measure}.

The result is new even for $\mathcal{H}(0,0)$. Although in this special case, the measure estimates made in Section \ref{sec:bad-measure} can be greatly simplified, 
we do not know if the result holds for every torus in $\mathcal{H}(0,0)$.
\end{remark}

In fact, we will need a more general statement than Theorem~\ref{thm:uniformtostart} in order to perform induction.  For this reason we will prove Theorem~\ref{thm:uniformtostart:genstatement} instead of Theorem~\ref{thm:uniformtostart}.  Theorem~\ref{thm:uniformtostart:genstatement} is proven by quantifying the algorithm of \cite{NguyenPanSu2020} recalled below.  The key obstacle is that as we apply the horocycle flow, we may have to wait too long to see the next cylinder.  Bounding this ``wait time'' is the key obstacle to overcome in the proof of Theorem~\ref{thm:uniformtostart:genstatement}.  Indeed, the source of the measure zero assumption results from the use of the Borel-Cantelli lemma to argue that if we do have to wait too long, then this can only happen repeatedly on a set of measure zero.  After proving Theorem~\ref{thm:uniformtostart:genstatement}, we prove Theorem~\ref{thm:lower-general} by induction in Section~\ref{Sect:LowerBdProof}.

Finally, we obtain the growth rates of non-invariant saddle connections and simple cylinders.  Recall that $A^{cyl}((X, \omega) \setminus \beta,L) $ and $A^{non-inv}((X, \omega) \setminus \beta, L)$ are the subsets of $A((X, \omega) \setminus \beta,L)$ consisting of cylinders and non-invariant saddle connections that are interiorly disjoint from $\beta$, respectively.

\begin{theorem}
\label{thm:lower-general-cyls-noninv}
For  $\mu$ almost every  surface $(X,\omega)\in \mathcal{H}_1^{hyp}(\kappa)$ and for any invariant saddle connection $\beta$ of logarithmic scale $j_0$ on $(X,\omega)$, there is a constant  $L_0$ depending on $(X,\omega)$ and $j_0$ such that for  any $L > L_0$, the  number of saddle connections of length at most $L$ interiorly disjoint from $\beta$ is at least
$$c\frac{L}{|\beta|}\left(\log \frac{L}{|\beta|}\right)^{d-3} \leq \# A^{cyl}((X, \omega) \setminus \beta,L)$$
and
$$c\frac{L}{|\beta|}\left(\log \frac{L}{|\beta|}\right)^{d-3} \leq \# A^{non-inv}((X, \omega) \setminus \beta,L),$$
where $d = dim_\bC\,\cH^{hyp}(\kappa)$ and the coefficient $c>0$ is from Theorem~\ref{thm:lower-general}.
\end{theorem}

We postpone the proofs of these theorems until after we establish several preliminary results.  
Section~\ref{LowerBoundSetting:Section} recalls the algorithm of Nguyen-Pan-Su for finding simple cylinders containing $\beta$, which forms the basis of the proof of Theorem~\ref{thm:lower-general}.  Section~\ref{LowerBoundSetting:Section} proceeds to establish the setup and notation for the results that follow.  Fundamental elementary computations are performed on which the proofs below will rely.  Section~\ref{sec:bad-measure} focuses on constraining the key obstacle to our lower bound, which we term ``bad saddle connections.''  In addition to the length parameter, a key to the argument below is the introduction of a time parameter.  In Section~\ref{sect:CylLowerBd}, we control the times that are bad and obstruct our desired growth rate.  Finally, with both the length and time obstacles constrained, Section~\ref{Sect:LowerBdProof} proves Theorems~\ref{thm:lower-general} and \ref{thm:uniformtostart}, and as in the case of the upper bound, we simultaneously prove Theorem~\ref{thm:lower-general-cyls-noninv}.

\subsection{Setting}
\label{LowerBoundSetting:Section}

Here we provide the necessary figures and setup for the arguments below.  As above, let $\beta$ be an invariant saddle connection on $(X, \omega)$.  Rotate $(X, \omega)$ so that $\beta$ is horizontal, and orient $\beta$ from left to right.

We begin by recalling the argument of \cite[Prop.~6.1]{NguyenPanSu2020}, which produces infinitely many cylinders containing $\beta$.  The argument is fundamental to everything that follows so we summarize it here for the reader's convenience.

Recall from Section~\ref{sec:background} that the horocycle flow $u_s$ is the upper triangular unipotent flow.  Define $(X_s, \omega_s) = u_s \cdot (X, \omega)$.  Clearly, $u_s$ preserves $\beta$.

\subsubsection*{The Nguyen-Pan-Su Algorithm for Producing Cylinders}

\begin{itemize}
\item Consider the vertical straight-line flow transverse to $\beta$ and flow all of the points in $\beta$ along the vertical direction.  Either an interior point of $\beta$ hits a singularity or the vertical flow is periodic.  In the latter case, we have produced a cylinder containing $\beta$. If the resulting cylinder is simple then we are done; otherwise we \emph{slightly} apply the horocycle flow $u_s$ on $(X,\omega)$ so that the vertical direction on the new surface is not periodic and we are back in the former case. (Since the set of periodic directions is of measure zero, we may always assume that we are in the former case.)\item In the former case, consider the nearest zero $p_0$ of $(X,\omega)$ above $\beta$. Consider the saddle connections from the endpoints of $\beta$ to $p_0$ and denote them by $\rho$ and $\sigma$.  
    Then $\beta$ together with $\rho$ and $\sigma$ bounds an embedded triangle. 
    Let $P$ be the parallelogram bounded by $\{\rho, \sigma, \tau(\rho), \tau(\sigma)\}$ with diagonal $\beta$.  See Figure~\ref{fig:parallelogramFindpnaught}. 
    
\begin{figure}[htbp]
\centering
 \begin{tikzpicture}[scale=1.7]
 \coordinate (O) at (0,0) node[below] {$\beta$};
 \coordinate (A) at (-2,0);
 \coordinate (B) at (2,0);
 \coordinate (P) at (1,1);
  \coordinate (Q) at (-1,-1);
\fill[color=gray!50] (A) -- (B) -- (P);
  \draw[thick] (-2,0) -- (2,0);
  \draw[red, thick] (A) -- (P);
  \draw[blue, thick] (B) -- (P);
  \draw[red] (B) -- (Q);
  \draw[blue] (A) -- (Q);
  \filldraw (A) circle (2pt) node[left] {A};
 \filldraw  (B) circle (2pt) node[right] {B};
  \filldraw  (P) circle (2pt) node[above] {$p_0$};
   \filldraw  (Q) circle (2pt) node[below] {$\tau(p_0)$};
\node at (-1,0.5) {$\rho$};
\node at (1.7,0.5) {$\sigma$};
 \end{tikzpicture}
 \caption{The saddle connection $\beta$, connecting $A$ to $B$, is invariant under $\tau$.
 There is a zero $p_0$ lying above $\beta$ and nearest to $\beta$.}
\label{fig:parallelogramFindpnaught}
\end{figure}
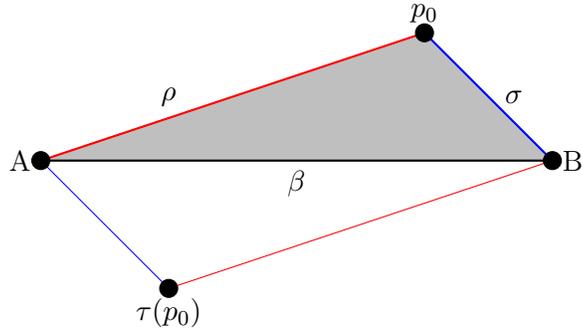

\item If either $\rho$ or $\sigma$ are invariant under $\tau$, then $P$ is a cylinder containing $\beta$.  Otherwise, assume that $\sigma$ is on the right and we apply the horocycle flow $u_s$ to act on $(X,\omega)$ until $\sigma$ becomes vertical.  Excise the subsurface bounded by $\sigma \cup \tau(\sigma)$ that does not contain $\beta$ (see Lemma~\ref{TwoBdSC:Lemma}).  Repeat the argument above for this proper subsurface to produce a new parallelogram $P'$.  See Figure~\ref{fig:parallelogramNewpnaught}.
    
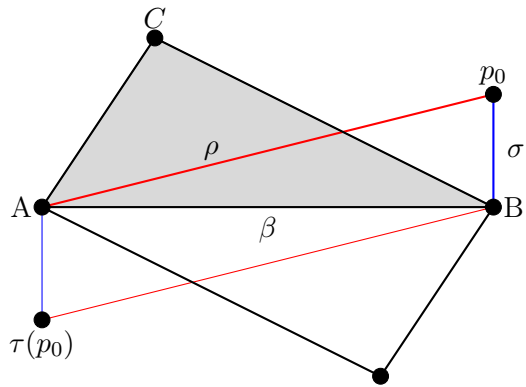
\begin{figure}[htbp]
\centering
 \begin{tikzpicture}[scale=1.5]
 \coordinate (O) at (0,0) node[below] {$\beta$};
 \coordinate (A) at (-2,0);
 \coordinate (B) at (2,0);
 \coordinate (P) at (2,1);
  \coordinate (Q) at (-2,-1);
    \coordinate (R) at (-1,1.5);
        \coordinate (S) at (1,-1.5);
    \draw[fill=gray!30!white]  (A) -- (R) -- (B);
  \draw[red, thick] (A) -- (P);
  \draw[blue, thick] (B) -- (P);
  \draw[red] (B) -- (Q);
  \draw[blue] (A) -- (Q);
          \draw[thick] (A) -- (R) -- (B);
    \draw[thick] (A) -- (S) -- (B);
  \draw[thick] (-2,0) -- (2,0);
  \filldraw (A) circle (2pt) node[left] {A};
 \filldraw  (B) circle (2pt) node[right] {B};
  \filldraw  (P) circle (2pt) node[above] {$p_0$};
    \filldraw (R) circle (2pt) node[above] {$C$};
      \filldraw (S) circle (2pt);
   \filldraw  (Q) circle (2pt) node[below] {$\tau(p_0)$};
\node at (2.2,0.5) {$\sigma$};
\node at (-0.5,0.5) {$\rho$};
 \end{tikzpicture}
 \caption{Under the horocyle flow, $\sigma$ becomes vertical.  A new parallelogram $P'$ is produced with a vertex $C$ on the top of $\beta$.}
\label{fig:parallelogramNewpnaught}
\end{figure}

\item If $P'$ is a simple cylinder, we are done.  Otherwise, we repeat the above process. 
Note that the resulting subsurface at each step lies in a stratum of dimension strictly less than that of the subsurface at the previous step.  Thus, the algorithm terminates after at most $(d-2)$ steps, where $d$ is the dimension
of $\mathcal{H}^{hyp}(\kappa)$.
\end{itemize}

The above algorithm shows that, when applying  the horocycle flow $u_s$ to $(X,\omega)$
for some $s$, we can find a cylinder on $(X_s, \omega_s)$ that contains $\beta$.
By pulling back 
 to $(X,\omega)$, we find a cylinder containing $\beta$.
 
The key challenge for us is to quantify this algorithm and show that there are
linearly many simple cylinders of length at most $2^j$ if we apply the horocycle flow $u_s$ to $(X,\omega)$ in an appropriate interval of time. 
The complication is that there may be parallelograms $P$ with very small areas.  For these, the time under the horocycle flow that the parallelogram $P$ remains will be large.  The goal will be to show that for almost every surface these small area parallelograms are rare enough that they do not interfere with linear growth.  

\paragraph{Definitions and Fixing Constants}  As always $\epsilon>0$ will be a sufficiently small fixed constant.  Let $\cH^{hyp}(\kappa)$ be a stratum with $d = dim_\bC\,\cH^{hyp}(\kappa)$.  Define
$$m_0= \frac{\epsilon^2}{2^{7}d} \min \left( \frac{1}{2^{5}}, \frac{1}{c c'} \right),$$
where $c$ is the constant in the statement of Lemma \ref{Lemma:Triangles:BigSigmam} and $c'$ is a constant depending on $c_2$ from the statement of Lemma~\ref{SeparatingSCSetForGammaBeta:Lemma} - both of which depend only on the stratum.  This inequality will be used  in the proof of Proposition~\ref{lem:badtime} below.

For later use of Theorem \ref{thm:uniformtostart} to 
prove the lower bound theorem by induction, 
we will also need to consider subsurfaces of $(X,\omega)$
or surfaces of area $A<1$ in lower dimensional strata and the horocycle flow for these. Thus, we will consider the Nguyen-Pan-Su algorithm (and write NPS-algorithm for brevity)  for surfaces in $\cH_A^{hyp}(\kappa)$.  However, we will not need to consider measures on lower dimensional strata. Instead our  computations of measure will always take place in the stratum containing $(X,\omega)$.

We also note that there will be two measures on two different spaces going forward.  There will be the Masur-Smillie-Veech measure $\mu$ on strata $\cH_1(\kappa)$, but we will also discuss Lebesgue measure on the real line in the context of estimating times.  To emphasize in which space we are working, we denote Lebesgue measure on the real line by $\lambda$.

\bigskip

Let $(X,\omega) \in \cH_A^{hyp}(\kappa)$.
Given a saddle connection $\rho$, we shall denote its horizontal and vertical components on $(X,\omega)$ by $h_\rho$ and $v_\rho$, respectively.
Without loss of generality, we may assume that $\rho$ lies above $\beta$ so that $v_\rho>0$.    Note that $v_\rho$ is invariant under the horocycle flow.  On $(X_s, \omega_s)$ the horizontal component  of $\rho$ becomes $h_\rho+sv_\rho$.

\begin{definition}
A saddle connection $\rho$ is \emph{available for $\beta$ at time $s$} if on $(X_s, \omega_s)$ it is a side of a parallelogram $P$ whose interior is embedded such that 
\begin{itemize}
	\item $\beta$ is a diagonal of $P$; 
	\item the angles that all the sides make with $\beta$ are  at most $\pi/2$.
\end{itemize}
If $I \subset \bR$ is a maximal interval such that for all $s \in I$, $\rho$ is available at time $s$, then we say that \emph{$\rho$ spends time $|I|$ being available}.
\end{definition}
In particular, any available $\rho$ together with $\beta$ determine an embedded triangle.
If the surface has area $A$, then $v_\rho \leq \frac{A}{|\beta|}$.  
Note that the horocycle flow $u_s$ moves points above $\beta$ from left to right.  Therefore, if a large $s$ is required to make $\rho$ available, then on the surface $(X, \omega)$ \emph{before} applying $u_s$, $\rho$ and $\beta$ form an angle close to $\pi$, and we have $h_\rho < 0$.

In the following, we shall apply the horocyle flow with $s\in \left[\frac{T|\beta|}{2A}, \frac{T|\beta|}{A} \right]$.

\begin{lemma}
\label{rhoLengthTime:Lemma}
Let $\rho$ be available at a time $s\in \left[\frac{T|\beta|}{2A}, \frac{T|\beta|}{A}\right]$.  Then
$$|\rho| \leq T + \frac{A}{|\beta|}.$$
\end{lemma}

\begin{proof}
If a saddle connection $\rho$ is available at some $s\in \left[\frac{T|\beta|}{2A}, \frac{T|\beta|}{A}\right]$, then
we have  $$0 \leq h_\rho+sv_\rho \leq |\beta|.$$ 
From the left-hand side, since $h_\rho<0$, we have $|h_\rho| \leq s v_\rho$, which implies
\begin{eqnarray*}
  |\rho| &\leq & |h_\rho| + |v_\rho| \\
   &\leq& (s+1) v_\rho \\
&\leq& T + \frac{A}{|\beta|},
\end{eqnarray*}
where in the last inequality we use the fact $v_\rho |\beta|\leq A$.
\end{proof}

This implies that all available $\rho$ have length less than $T$  (up to an additive constant depending on $|\beta|$ and the area).

\begin{definition}
Let $(X, \omega)$ have area $A$.  We say an available $\rho$ is  \emph{$m_0$-bad for $\beta$ on $(X, \omega)$} (or simply \emph{bad} when $m_0$, $\beta$, and $(X, \omega)$ are understood) if 
$$0< v_{\rho} < \frac{m_0 A}{|\beta|}.$$
We say that $\rho$ is \emph{good} if it is available and not bad.
\end{definition}

\begin{lemma}
\label{TimeAvailable:Lemma}
	If $\rho$ is available, then the time that $\rho$ spends being available under $u_s$ is $\frac{|\beta|}{v_\rho}$. Likewise,  if 
 $\rho$ is good, then it is good for a total time  at most $\frac{|\beta|^2}{m_0 A}$.
\end{lemma}
\begin{proof}
	This follows from the fact that the zero goes from lying vertically above one endpoint of $\beta$ to vertically above the other changing at rate $v_\rho$.
\end{proof}

Thus, if $\rho$ is bad,  Lemma~\ref{TimeAvailable:Lemma} implies that it will remain bad under the horocycle flow in an interval of length $\frac{|\beta|}{v_\rho} \geq \frac{|\beta|^2}{m_0 A}$.  Furthermore, when applying the NPS-Algorithm to find simple cylinders,
we need to spend about $|\beta|/v_\rho$ time to move the vertex of $\rho$ (the one on the top of $\beta$) from left to right. And we may need $(d-2)$ steps to find a simple cylinder. So if in the process some $\rho$ has vertical length that is much less than $\frac{m_0 A}{|\beta|}$, then the time needed to move  $\rho$ away from the top of $\beta$ will be much larger than $\frac{|\beta|^2}{m_0A}$ and this might prevent the desired linear growth.  The majority of the proof of the linear growth is dedicated to controlling this possibility.

\begin{definition}
A time $s$ is \emph{bad} if there is at least one bad saddle connection $\rho$ at  time $s$.
Let $\mathbf{Bad}\subset \left[\frac{T|\beta|}{2A},\frac{T|\beta|}{A}\right]$  be the set of bad times $s$ and  $\mathbf{Good}\subset \left[\frac{T|\beta|}{2A},\frac{T|\beta|}{A}\right]\setminus \mathbf{Bad}$.
\end{definition}

Given the interval  $\left[\frac{T|\beta|}{2A},\frac{T|\beta|}{A}\right]$, one of the primary objectives below is to estimate the $\lambda$-measure of $\mathbf{Bad}$.

We introduce a concept closely related to logarithmic scale.

\begin{definition}
A saddle connection $\rho$ has \emph{horizontal logarithmic scale $\ell$} if $2^\ell\leq  |h_\rho| \leq 2^{\ell+1}$.
\end{definition}

\begin{remark}
For large $\ell$,  the horizontal logarithmic scale of an available $\rho$ is approximately the logarithmic scale of $\rho$. 
\end{remark}

\begin{convention}
Let $\rho$ be an available saddle connection that forms a parallelogram containing $\beta$ with sides $\{\rho, \tau(\rho), \rho', \tau(\rho')\}$.  
Note that $\rho$ is available (bad) if and only if $\rho'$ is available (bad). 
Then either $|\rho|\geq |\beta|/2$ or $|\rho'| \geq |\beta|/2$.  
By interchanging the roles of $\rho$ and $\rho'$, we shall always assume that 
$|\rho| \geq |\beta|/2$. 
\end{convention}

\begin{convention}
We will often want to discuss the angle between a saddle connection $\rho$ and a horizontal slit $\beta$.  Naturally, this angle will change dramatically under the horocycle flow.  We adopt the following convention that the \emph{angle of $\rho$ relative to $\beta$} is the \emph{supplementary} angle to the angle formed between $\rho$ and $\beta$ on the base surface $(X, \omega)$ (prior to the application of the horocycle flow).  Consequentially, the angle of $\rho$ relative to $\beta$ is closely approximated by its tangent, which is given by $\left|v_{\rho}/h_{\rho}\right|$ because the angle between $\rho$ and $\beta$ is approximately $\pi$.
\end{convention}

We now prove a lemma that contains many of the elementary estimates from which all of the more sophisticated estimates below follow.

We highlight Item~\ref{lem:badtimes:vrhoLowerBound} because it resolves an important issue.  We defined bad saddle connections to have small vertical height.  Clearly, there is no positive minimum vertical height for an arbitrary saddle connection.  This presents a significant obstacle because a priori, there could be a bad saddle connection for the full length of an interval that we consider.  However, Lemma~\ref{lem:badtimes} proves that the vertical height can indeed be bounded from below once we know both the horizontal logarithmic scale \emph{and} the time when it is available.

\begin{lemma}
\label{lem:badtimes}
Let $(X, \omega) \in \cH_A^{hyp}(\kappa)$ with horizontal slit $\beta$.  
Let $\rho$ have horizontal logarithmic scale $\ell$ on $(X,\omega)$ which is available for some
$s\in \left[\frac{T|\beta|}{2A},\frac{T|\beta|}{A}\right]$.  Then 
\begin{enumerate}[label=(b-\arabic*)]
\item \label{lem:badtimes:vrhoLowerBound} $$\frac{2^\ell A}{T|\beta|}\leq v_\rho.$$
\item \label{lem:badtimes:AngleUpperBound} 
If $|\rho| \geq |\beta|/2$, then
$$\left|\frac{v_\rho}{h_\rho} \right| \leq \frac{10 A}{T|\beta|}.$$
\item \label{lem:badtimes:RhoAvailable} The time that $\rho$ is available is at most
$$\frac{|\beta|^2T}{2^\ell A}.$$
\item \label{lem:badtimes:nobad} If $2^{\ell} \geq  m_0 T$, then $\rho$ is not bad.  Furthermore, if $\rho$ is bad, then $|h_\rho| < 2m_0 T$. 
\end{enumerate}
\end{lemma}

\begin{proof}[Proof of Item~\ref{lem:badtimes:vrhoLowerBound}] Since $\rho$ is available for  $s \in \left[\frac{T|\beta|}{2A},\frac{T|\beta|}{A}\right]$, then there exists $s\leq \frac{T|\beta|}{A}$ such that
$$h_\rho + sv_\rho =0.$$
Therefore, 
\begin{eqnarray*}
 v_\rho = \frac{|h_\rho|}{s} \geq \frac{2^{\ell}A}{T|\beta|}.
\end{eqnarray*}
\end{proof}

\begin{proof}[Proof of Item~\ref{lem:badtimes:AngleUpperBound}]
For some $s\in \left[\frac{T|\beta|}{2A},\frac{T|\beta|}{A}\right]$, $0\leq h_\rho + s v_\rho \leq |\beta|$.  By our conventions for the signs of $h_\rho$ and $v_\rho$ (and without assuming $|\rho| \geq |\beta|/2$), we have
$$v_\rho \leq \frac{|\beta| + |h_\rho|}{s} \leq \frac{2A}{T|\beta|} \left(|\beta| + |h_\rho|\right) =  \frac{2A}{T}\left( 1+ \frac{|h_\rho|}{|\beta|} \right).$$
The horizontal logarithmic scale assumption gives $|h_\rho| \geq 2^{\ell}$.

If in addition, $|\rho| \geq |\beta|/2$ , then $2^{\ell} \geq  |\beta|/4$, which implies that 
$$\left|\frac{v_\rho}{h_\rho} \right| \leq \frac{10 A }{T |\beta|}.$$
\end{proof}

\begin{proof}[Proof of Item~\ref{lem:badtimes:RhoAvailable}]
For Item~\ref{lem:badtimes:RhoAvailable}, an upper bound for the time that $\rho$ is available is quickly deduced from Item~\ref{lem:badtimes:vrhoLowerBound} and
$$\frac{|\beta|}{v_\rho} \leq \frac{|\beta|^2T}{2^\ell A}.$$
\end{proof}

\begin{proof}[Proof of Item~\ref{lem:badtimes:nobad}]
By Item~\ref{lem:badtimes:vrhoLowerBound} and the assumption, 
$$v_{\rho} \geq \frac{2^{\ell}A}{T|\beta|} \geq \frac{m_0 A}{|\beta|}.$$
The last claim follows from the definition of $h_\rho$ has horizontal logarithmic scale $\ell$.
\end{proof}

We end this subsection with the following lemma.
Recall that a saddle connection $\gamma$ corresponding to the boundary of a simple cylinder is good if $ v_\gamma \geq \frac{m_0 A}{|\beta|} $ and thus the cylinder
has area at least $m_0 A$.

\begin{lemma}
\label{GoodCylinderAngle:Lemma}
The cross product between any two good cylinder saddle connections is at least $m_0 A$.
\end{lemma}

\begin{proof}
Suppose a pair of cylinders $\gamma_1$ and $\gamma_2$ satisfy $|\gamma_1\times \gamma_2| < m_0 A$.  Assuming $|\gamma_1| \geq |\gamma_2|$, we can rotate so that $\gamma_1$ is vertical and apply the Teichm\"uller flow with $e^t=2|\gamma_1|/ \sqrt{m_0 A}$.   After flowing, $\gamma_1$ has length $ \frac{\sqrt{m_0 A}}{2} $ and $\gamma_2$ has horizontal length less than $ 2 \sqrt{m_0 A} $, and at least one of these inequalities is strict by the cross product assumption.  Since $|\gamma_1|$ is the circumference of the cylinder it bounds and the horizontal length of $\gamma_2$ is greater than that cylinder's height, the area is of the cylinder is 
less than  $ \frac{\sqrt{m_0 A}}{2} \cdot 2 \sqrt{m_0 A} = m_0A$.  This contradicts the assumption that $\gamma_1$ defines a cylinder of area greater than or equal to $m_0 A$.  
\end{proof}

\subsection{Bounding Bad Saddle Connections}\label{sec:bad-measure}

\paragraph{No Bad Saddle Connections:} The first conclusion one draws from the elementary estimates above is that if there were no bad saddle connections, i.e., $v_\rho \geq \frac{m_0 A}{|\beta|}$ for all available $\rho$, we would be essentially done.  

In this case, the speed of the horocycle flow is at least $\frac{m_0 A}{|\beta|}$.
In the NPS-Algorithm, to pass from one available $\rho$ to the next, it suffices to take an amount of time
at most $\frac{|\beta|^2}{m_0 A}$ by Lemma~\ref{TimeAvailable:Lemma}. As a result, from the discussion following Lemma~\ref{TimeAvailable:Lemma}, in every sub-interval of length $(d-2) \frac{|\beta|^2}{m_0 A}$, we can find at least one cylinder. Since the interval $\left[\frac{T|\beta|}{2A}, \frac{T|\beta|}{A}\right]$ contains approximately 
$$\frac{\frac{T|\beta|}{2A} }{\frac{(d-2)|\beta|^2}{m_0 A}}= \frac{m_0 T}{2(d-2)|\beta|}$$
many sub-intervals,  we obtain at least $c \frac{T}{|\beta|}$ many cylinders of length less than $T$,
 by taking $c= \frac{m_0}{2(d-2)}$.

\begin{remark}\label{rem:veech}
For Veech surfaces, the no small triangle property from \cite[Thm.~1.1]{SmillieWeissCharLattice} implies that there is a value of $m_0$ for which there cannot be any triangles with height less than $m_0$.  In other words, for this value of $m_0$, there are no bad saddle connections when $(X,\omega)\in \cH^{hyp}_1(\kappa)$ is a Veech surface. In this case, we do not need the measure estimates below, and we have Theorem \ref{thm:uniformtostart} and Theorem \ref{thm:lower-general} for every Veech surface in $\cH^{hyp}_1(\kappa)$. Note that there are non-Veech surfaces in $\mathcal{H}(0,0)$.
\end{remark}

\paragraph{Bad Saddle Connections:} With the above motivation, we must bound the number of bad saddle connections and bound the measure of bad times.  We present this bound here in Proposition~\ref{BadSaddleConnBd:Prop} and give a complete proof in terms of the results that follow below.

The strategy is as follows.  First we will consider bad saddle connections that are $\epsilon$-isolated and have logarithmic scale that is not too large.  In this case, Proposition~\ref{lem:bad-isolated} will establish that they do not take up too large a measure of the 
interval of times.   Next, we consider bad saddle connections that have small logarithmic scale.  Regardless of whether or not these bad saddle connections are $\epsilon$-isolated, we will show in Proposition~\ref{prop:smalllogarithmic scale:MeasureZero}
 that the subset of translation surfaces where there is a  bad saddle connection with small logarithmic scale can be avoided for sufficiently large $L$ on a full measure subset of the stratum via the Borel-Cantelli Lemma (Corollary~\ref{coro:bad-smallmeasure}).  Finally, in Proposition~\ref{lem:badtime}, we will address the final case of surfaces with bad saddle connections with large logarithmic scale that are not $\epsilon$-isolated, and show that the set of such surfaces can also be avoided on a full measure subset of the stratum.  To use the Borel-Cantelli Lemma, we will need measure estimates that will be established in Section~\ref{MeasureEstimates:Subsection}.

\subsubsection{Measure Estimate: Bounding Bad Saddle Connections of Small logarithmic scale}
\label{MeasureEstimates:Subsection}

Recall that the measure $\mu$
is a finite $\operatorname{SL}_2(\bR)$-invariant measure on
$\cH^{hyp}_1(\kappa)$. The measure is induced by the volume element on
$\cH^{hyp}(\kappa)$ by restricting to the subset of surfaces with area at most one.
Let us recall the following fact from \cite[Claim~7 Pg.~518]{MasurSmillieHausdDim}.

\begin{proposition}[Masur-Smillie]\label{thm:Masur-Smillie}
Let $\epsilon_1$ and $\epsilon_2$ be sufficiently small constants. 
Let $\Omega(\epsilon_1, \epsilon_2)$ be the subset of translation surfaces in 
$\cH^{hyp}_1(\kappa)$ such that there are two non-homologous saddle connections
$\gamma_1$ and $\gamma_2$ with lengths $|\gamma_1|< \epsilon_1, |\gamma_2|< \epsilon_2$.
Then 
$$\mu\left(\Omega(\epsilon_1, \epsilon_2) \right)\leq C \epsilon_1^2 \epsilon_2^2 ,$$
where the constant $C$ depends on the stratum. 
\end{proposition}

Let $\delta>0$.   Recall the $\delta$-thick part of a stratum is the subset with saddle connection systole bounded below by $\delta$.  Denote the $\delta$-thick part of $\cH^{hyp}_1(\kappa)$ by $\cH_{1}^{hyp}(\kappa;\delta)$.

\paragraph{The Set $\Omega_1$:} Fix $\delta>0$.  Define $\Omega_1(j, \delta)$ to be the set of  $(X, \omega) \in \cH_{1}^{hyp}(\kappa;\delta)$ such that:
\begin{enumerate}[label=(1-\alph*)]
\item[(1-a)] \label{Omega1:BetaExists}
There exists a saddle connection $\beta$ of logarithmic scale at most $j/4$.

\item[(1-b)]  There exists a saddle connection $\rho$ of logarithmic scale  
at most $j/3$ and  makes an angle at most  $\frac{20}{2^{j}|\beta|}$
 with $\beta$.   
\end{enumerate}

\begin{lemma}
\label{lem:smalllogarithmic scale}
Given $\delta > 0$ and $\cH^{hyp}(\kappa)$, there is a constant $C_1$ depending on $\delta$ and $\kappa$ such that for $j$ sufficiently large 
$$\mu\left(\Omega_1(j, \delta)\right) \leq C_1 \frac{j}{2^{j/12}}.$$
\end{lemma}

\begin{proof}
Given $\ell\leq j/4$, define $\Omega_1(j,\ell)$ to be the set of $(X,\omega)$ in $\Omega_1(j, \delta)$
such that the logarithmic scale of $\beta$ is $\ell$. 

Take a partition of the circle $S^1$ by intervals $I$ of length $\frac{1}{2^{j+\ell}}$, and for each $I$, let $\Omega_1(j,\ell,I)\subset \Omega_1(j,\ell)$ be the set of $(X,\omega)$ such that the direction of $\beta$ is in $I$.
Rotate so the center  of $I$ is vertical  and  apply the Teichm\"uller flow $g_t$, where $$e^t=\sqrt{2^{j+\ell}}.$$  After flowing, $\beta$ has vertical length and horizontal length bounded by
$2 \frac{2^{\ell/2}}{2^{j/2}}$. So the square of its length  is at most
 $$8 \frac{2^{\ell}}{2^j}.$$
 
After flowing, the  vertical component of $\rho$ is  bounded by  $$2 \frac{2^{j/3}}{2^{j/2+\ell/2}} = \frac{2}{2^{j/6+\ell/2}},$$ and since $\rho$ makes angle at most $\frac{21}{2^{j+\ell}}$ with the vertical direction, the horizontal length of $\rho$ is at most $\frac{42}{2^{j/6+\ell/2}}$.
 Thus the square of its length is at most 
 $$\frac{2^{11}}{2^{j/3+\ell}}.$$
 By Proposition~\ref{thm:Masur-Smillie}, for $j$ sufficiently large, the measure $$\mu \left( \Omega_1(j,\ell,I)\right) = \mu \left(g_t \cdot  \Omega_1(j,\ell,I)\right)\leq 2^{15}\frac{C }{2^{4j/3}}.$$
Summing up over the $2^{j+\ell}$  number of intervals and using $\ell\leq j/4$, we find 
$$\mu\left(\Omega_1(j,\ell)\right)  \leq 2^{15} \frac{C}{2^{j/12}}.$$  We then sum over $\ell$, where $\log \delta \leq \ell\leq j/4$ to get
$$\mu\left(\Omega_1(j, \delta)\right)  \leq 2^{15} C(j-\log \delta) \frac{1}{2^{j/12}},$$ 
which is less than $C_1 \frac{j}{2^{j/12}}$ when $j \geq |\log \delta|$,
by taking $C_1= 2^{16}C$.
\end{proof}

\subsubsection{Measure Estimate: Bounding Bad Saddle Connections of Large logarithmic scale}

Let $M>2$ and $c_3 >0$ be constants.
Define 
$$\Omega_2(\ell, c_3, M, \delta) = \bigcup_{n,p} \Omega_2(\ell, n,p),$$
where $\Omega_2(\ell, n,p)$ is the set of all $(X, \omega) \in \cH_{1}^{hyp}(\kappa;\delta)$ such that:
\begin{enumerate}[label=(2-\alph*)]
\item \label{Omega2:Gammalogarithmic scale} There is a saddle connection $\rho$ on $(X,\omega)$ with logarithmic scale  $\ell$.
\item \label{Omega2:Sigmalogarithmic scale} There is a saddle connection $\sigma$ on $(X,\omega)$ with logarithmic scale 
   $n< \ell/4$ and $\sigma$ makes an angle with $\rho$ at most
   $c_3/2^{m+n}$, where $m=n+ \frac{\ell}{2M}$. 
\item \label{Omega2:Rhologarithmic scale} There is a saddle connection $\beta$ on $(X,\omega)$ with logarithmic scale $p\leq \ell/8M$ and makes an angle with $\rho$ at most $1/2^{\ell+p}$. 
\end{enumerate}

\begin{lemma}\label{lemma:measure-bad-case2}
Given $\delta > 0$, $c_3 > 0$, $M > 2$, and $\cH^{hyp}(\kappa)$, there is a constant $C_2$ such that for $\ell$ sufficiently large
$$\mu\left(\Omega_2(\ell, c_3, M, \delta)\right) \leq C_2 \frac{\ell^2}{2^{\ell/8M}}.$$
\end{lemma}

\begin{proof}
Fix $n$ and $p$.  Let $\Omega_2(\ell,n,p) \subset \Omega_2(\ell, c_3, M, \delta)$ be as above.  By our assumptions, 
\begin{equation*}
  m+n = 2n + \frac{\ell}{2M} \leq \frac{3\ell}{4}.
\end{equation*}

We partition the circle of directions into intervals of length $1/ 2^{m+n}$ and denote by $I$ the interval containing the direction of $\rho$. Rotate the circle of directions so that $I$ contains the vertical direction and apply the Teichm\"uller flow $g_t$ with $$e^t=\frac{2^m}{2^k},$$
where $k$ will be specified below. 

Assume that $(X,\omega)\in\Omega_2(\ell,n,p)$.
Since $(X,\omega)$ lies in the $\delta$-thick part, $n$ and $p$ are bounded below by $\log \delta$. 
  We shall assume that $\ell$ is sufficiently large, in particular, 
   $\ell\geq 4M\left( |\log \delta| +1\right)$. 

Since the angle between $\rho$ and $\beta$ is at most
$\frac{1}{2^{p+\ell}}, $
the angle between $\beta$  and the vertical direction is at most
$$ \frac{1}{2^{p+\ell}}+ \frac{1}{2^{m+n}}\leq \frac{2}{2^{m+n}}.$$
So after flowing by $g_t$, we have 
$$ |v_\beta| \leq\frac{2^{k+p+1}}{2^{m}} , \qquad  |h_\beta| \leq \frac{2^{p+2}}{2^{n+k}}.$$

On the other hand, the angle between $\rho$ and $\sigma$ is  less than $\frac{c_3}{2^{m+n}}$, so the angle between $\sigma$  and the vertical direction is less than  $\frac{1+c_3}{2^{m+n}} $.  After flowing by $g_t$, we have
$$ |v_{\sigma}| \leq \frac{2^{k+n+1}}{2^{m}}, \qquad  |h_{\sigma}| \leq  \frac{c_4}{2^{k}}.$$
Here we take $c_4=2(c_3+1)$. 

We would like to choose $k$ such that $|\rho|$ and $|\sigma|$ are simultaneously small. 
So we take $k$ such that 
$2k+n=m$.  After flowing,  we have
$$|\sigma| \leq  \frac{2+c_4}{2^{k}} .$$
Note that
$$k = \frac{m-n}{2} = \frac{\ell}{4M}.$$ 
So $|\sigma|$ is small.  Similarly, we have
$$|\beta| \leq 5 \frac{2^{p}}{2^{n+k}}.$$
Since $p \leq \frac{\ell}{8M}=k/2$, $|\beta|$ is also small. 
Thus, both $\beta$ and $\sigma$ are short after the flow. 
Consequently, we have 
$$g_t \cdot \Omega_2(\ell,n,p) \subset \Omega\left( 5\frac{2^p}{2^{n+k}}, \frac{2+c_4}{2^k}\right),$$
where $\Omega(\cdot, \cdot)$ is the set from Proposition~\ref{thm:Masur-Smillie}.  Thus, by Proposition~\ref{thm:Masur-Smillie},
$$\mu\left( g_t \cdot \Omega_2(\ell,n,p) \right)  \leq c_5  \frac{2^{2p}}{2^{2m}} ,$$
where $c_5= 25(c_4+2)^2C$. By the invariance of the measure under the Teichm\"uller flow, we obtain
$$\mu\left( \Omega_2(\ell,n,p) \right)  \leq c_5  \frac{2^{2p}}{2^{2m}}.$$
Since there are $2^{m+n}$ many intervals, we have 
$$\mu\left(\Omega_2(\ell,n,p)\right) \leq c_5 \frac{2^{n+2p}}{2^m}  \leq  c_5 \frac{1}{2^{\ell/8M}}.$$
Taking $C_2 =4c_5$ and summing over all possible $p$ and $n$, we obtain (for $\ell\geq |\log \delta|$)
$$\mu\left(\Omega_2(\ell, c_3, M, \delta)\right) \leq C_2 \frac{\ell^2}{2^{\ell/8M}}.$$
\end{proof}

\subsubsection{Applying the Borel-Cantelli Lemma}
\begin{corollary}\label{coro:bad-smallmeasure}
For $\mu$-almost every $(X,\omega)\in \cH^{hyp}_1(\kappa)$, 
there is a constant $L_0$ depending on $(X,\omega)$ such that if $2^j \geq L_0$,
then $(X,\omega)$ does not belong to $\Omega_1(j, \delta)$ (without restricting to on the thick part). The same conclusion holds for $\Omega_2(\ell, c_3, M, \delta)$.
\end{corollary}
\begin{proof}
Restricting to the $\delta$-thick part, it follows from Lemma~\ref{lem:smalllogarithmic scale} that 
$$\sum_{j} \mu\left( \Omega_1(j, \delta)\right) < \infty.$$
Then the Borel-Cantelli Lemma implies that almost every $(X,\omega)\in \cH_{1}(\kappa;\delta)$ does not lie in $\Omega_1(j, \delta)$ when $j$ is greater than some $j_0=j_0(X,\omega)$.  Since the space $\cH^{hyp}_1(\kappa)$ can be exhausted by countably many thick parts, we are done. 

The same proof carries over to $\Omega_2(\ell, c_3, M, \delta)$ using Lemma~\ref{lemma:measure-bad-case2} in place of Lemma~\ref{lem:smalllogarithmic scale}.
\end{proof}

\subsection{Control of the Bad Times}
\label{sect:CylLowerBd}

We first show that isolated saddle connections that are not too short  take up a small proportion of the time interval. 
The next lemma does not require  an almost everywhere statement on points in the stratum.  It also holds for surfaces with area $A$ so that we will  be able to apply this lemma to handle the induction step.

\begin{proposition}
\label{lem:bad-isolated}
Let $(X, \omega) \in \cH_A^{hyp}(\kappa)$ have a horizontal slit $\beta$.  Consider times $T$ and the interval $\left[\frac {T|\beta|}{2A}, \frac{T|\beta|}{A}\right]$.
	Let $\cR$ be a collection of available saddle connections of horizontal logarithmic scale $\ell$ that are $\epsilon$-isolated, which satisfy 
	$$T^{1/8}\leq 2^\ell\leq 2 m_0 T.$$
	Then the $\lambda$-measure of the union of the intervals of the set of times for which a saddle connection in $\cR$ is available and over  all admissible logarithmic scales $\ell$ is at most 
	$$\frac{1}{2^5 d} \frac{T|\beta|}{A}$$ for $T$ sufficiently large depending on $|\beta|$.
\end{proposition}

\begin{remark}
Recall the motivation for the assumption in the lemma.  By Item~\ref{lem:badtimes:nobad} of  Lemma~\ref{lem:badtimes}, if $\rho$ is bad, then $2^\ell \leq 2m_0 T$.
\end{remark}

\begin{proof}
We begin  by fixing a logarithmic scale $\ell$.  For two intersecting saddle connections $\rho$ and $\rho'$ in $\cR$, the $\epsilon$-isolated assumption says that the angle between them is at least $\frac{A\epsilon^2}{2^{2\ell}}$. 
The angle that each available $\rho \in \cR$ makes with $\beta$ is at most $\frac{10A}{T|\beta|}$ by Lemma~\ref{lem:badtimes} Item~\ref{lem:badtimes:AngleUpperBound}.
Therefore, the number of available saddle connections in $\cS$ is bounded above by taking the quotient to get
$$\max\left\{ 10\frac{2^{2\ell}}{\epsilon^2 T |\beta|}, 1\right\}.$$

The maximum naturally splits the proof into two cases.  The first case assumes the number of available saddle connections in $\cS$ is bounded above by
$$10 \frac{2^{2\ell}}{\epsilon^2 T |\beta|}\geq 1.$$
Then the $\lambda$-measure of the union of  the corresponding collection of available intervals is bounded above by multiplying the maximum time that $\rho$ is available as given by Lemma~\ref{lem:badtimes} Item~\ref{lem:badtimes:RhoAvailable} by the number of available saddle connections in $\cR$ to get
$$\left(\frac{|\beta|^2 T}{2^\ell A} \right) \cdot \left(10 \frac{2^{2\ell}}{\epsilon^2 T|\beta|} \right) = 10 \frac{2^{\ell}|\beta|}{\epsilon^2 A}.$$
Summing over the values of $\ell$, using the assumption $2^\ell\leq 2 m_0 T$, 
we obtain the upper bound 
$$\sum_{0\leq \ell \leq \log_2 (2 m_0 T)} 10 \frac{2^{\ell}|\beta|}{\epsilon^2 A} \leq 2^6 \frac{m_0 T|\beta|}{\epsilon^2 A} \leq \frac{1}{2^6 d}\frac{T |\beta|}{A}$$
for the total $\lambda$-measure of these intervals.

The second possibility is that $$10 \frac{2^{2\ell}}{\epsilon^2 T|\beta|}< 1,$$
which we note implies $2^{\ell} < \epsilon \frac{\sqrt{T|\beta|}}{\sqrt{10}}$.
For each such $\ell$, there is at most a single $\rho$ of logarithmic scale $\ell$.
Again by Lemma~\ref{lem:badtimes} Item~\ref{lem:badtimes:RhoAvailable}, each $\rho$ determines an interval of logarithmic scale $\frac{T|\beta|^2}{2^\ell A}$. Summing up over $\ell$ such that 
$T^{1/8}\leq 2^\ell < \epsilon \frac{\sqrt{T|\beta|}}{\sqrt{10}}$,
we get that the sum of the  lengths of these intervals is at most $$\sum_{(\log_2 T)/8 \leq \ell \leq \log_2 (T|\beta|)}\frac{T|\beta|^2}{2^{\ell} A}\leq \frac{2T|\beta|^2}{T^{1/8} A},$$ 
which is less than $\frac{1}{2^6 d}\frac{T |\beta|}{A}$ provided  
$T^{1/8} \geq 2^7 d|\beta|.$ 
\end{proof}

The remaining part of this subsection shows that 
for almost every $(X,\omega)\in \cH^{hyp}_1(\kappa)$ and
for any given invariant saddle connection $\beta$, 
if $T$ is sufficiently large, then
the bad times in $\left[ \frac{T|\beta|}{2}, T|\beta|\right]$, contributed by those 
saddle connections that are either short or non-isolated, has small measure. 
This together with Proposition~\ref{lem:bad-isolated} are sufficient to prove 
Theorem~\ref{thm:uniformtostart}.

In order to obtain Theorem~\ref{thm:lower-general} in the following propositions, we shall consider $(X,\omega)\in \cH^{hyp}_1(\kappa)$
and a saddle connection $\gamma$, which is either $\beta$, or non-invariant and interiorly disjoint from $\beta$.
In the latter case, 
by cutting and gluing along $\gamma\cup \tau(\gamma)$, we obtain a surface 
$Y$ not containing  $\beta$ and lying in a hyperelliptic stratum $\cH^{hyp}(\kappa')$ of lower dimension by Lemma~\ref{TwoBdSC:Lemma}.
Then $\gamma$ becomes an invariant saddle connection on $Y$. 

We will apply the horocycle flow with $\gamma$ horizontal, consider bad times, and so forth. 
The key fact that we use which is  given in the remark following Lemma~\ref{TwoBdSC:Lemma} is that lengths of saddle connections on $Y$ that are disjoint from $\gamma$ and the angle between such saddle connections are the same as on $(X,\omega)$  before the cutting and gluing. This will allow 
us to apply Lemma~\ref{lem:badtimes} Item~\ref{lem:badtimes:AngleUpperBound} for example and calculate measures in $\cH^{hyp}_{1}(\kappa)$.

\begin{remark}
We note  that there are  two natural finite Lebesgue class measures one could consider on strata.  The first is the measure on the stratum of unit area surfaces,  $\cH_1^{hyp}(\kappa)$, and for various  $A$ the measure on lower dimensional strata of area $A$ translation surfaces. Though we will need to consider such subsurfaces with area $A \leq 1$, the \emph{only} measure on strata that will occur in this work is the standard one on $\cH_1^{hyp}(\kappa)$ as that is the one that occurs in the statement of the lower bound. It is also the case that it is nontrivial to go from measures on lower dimensional strata to measures on higher dimensional strata. 
\end{remark}

\paragraph{Setting for Propositions~\ref{prop:smalllogarithmic scale:MeasureZero}, \ref{lem:badtime} and \ref{BadSaddleConnBd:Prop}:}

\

\noindent Given $(X,\omega)\in \cH^{hyp}_1(\kappa)$ with an invariant slit $\beta \subset (X, \omega)$, let $\gamma$ be a saddle connection on $(X, \omega)$.  If $\gamma = \beta$, define $Y = (X, \omega)$.  If $\gamma$ is not invariant and interiorly disjoint from $\beta$, cut and glue along $\gamma \cup \tau(\gamma)$ and let $Y \in \cH^{hyp}(\kappa')$ be the resulting subsurface \emph{not} containing $\beta$ with area $A \leq 1$ (and $A > 0$).  

\begin{proposition}
\label{prop:smalllogarithmic scale:MeasureZero}
Given the setup above, for $\mu$-almost every $(X,\omega)\in \cH^{hyp}_1(\kappa)$, there exists $T_0$ depending on $(X,\omega)$ and $|\gamma|$ such that for any $j$ satisfying $2^{j/4} \geq |\gamma|$ and $T\geq T_0$ satisfying $2^{j-1}\leq T<2^j$, each saddle connection that is bad for $\gamma$ on $Y$ has logarithmic scale at least $\frac{j}{3}$.
\end{proposition}

\begin{proof}
Note $Y \in \cH^{hyp}(\kappa')$ and $\cH^{hyp}(\kappa')$ has lower dimension by Lemma~\ref{TwoBdSC:Lemma}.  Moreover, $\gamma$ is invariant on $Y$ by construction.
We claim that if there exists a saddle connection $\rho$ that is bad for $\gamma$ on $Y$,  whose logarithmic scale is less than $j/3$, then $(X,\omega)\in \Omega_1(j, \delta)$ for some $\delta>0$.  
Indeed, $\gamma$ has logarithmic scale at most $j/4$ by assumption, so if $\gamma$ is the $\beta$ from the definition of $\Omega_1(j, \delta)$, then Item (1-a) in the definition is satisfied.  By Lemma~\ref{lem:badtimes} Item~\ref{lem:badtimes:AngleUpperBound}, the angle between $\gamma$ and $\rho$ is bounded above by 
$$\frac{10 A}{T|\gamma|} \leq \frac{10}{2^{j-1}|\gamma|} \leq \frac{20}{2^j |\gamma|}.$$
So, Item (2-a) in the definition of $\Omega_1(j, \delta)$ is satisfied as well.

Therefore, we can apply the Borel-Cantelli Lemma (Corollary \ref{coro:bad-smallmeasure}) to conclude the proposition.
\end{proof}

Proposition~\ref{prop:smalllogarithmic scale:MeasureZero} implies that we can restrict to the case of bad $\rho$ with logarithmic scale at least $\frac{j}{3}$.

\begin{proposition}\label{lem:badtime}  
Given the setup above, then $\mu$-almost every $(X,\omega) \in \cH^{hyp}_{1}(\kappa)$ has the following property.  For $j$ big enough, again setting $2^{j-1} \leq T < 2^j$, if a collection of saddle connections $\cR$ that are bad for $\gamma$ on $Y$ have logarithmic scale 
$$\ell\geq j/3$$
and are not $\epsilon$-isolated, then the times at which the saddle connections in $\cR$ are available in the interval $\left[\frac{T|\gamma|}{2A},\frac{T|\gamma|}{A}\right]$ have $\lambda$-measure at most $\frac{T|\gamma|}{2^4 dA}$ after summing all possible values of $\ell$.
\end{proposition}

\begin{proof}
Fix $\ell \geq j/3$.  We apply Lemma~\ref{SeparatingSCSetForGammaBeta:Lemma} to $Y$ with the invariant saddle connection $\gamma$ and the (bad) saddle connections in $\cR$ as the collection of saddle connections $\Gamma$.  Indeed, the first two assumptions in Lemma~\ref{SeparatingSCSetForGammaBeta:Lemma} are easily satisfied, and the assumption that the saddle connections in $\cR$ are bad implies that the triangles have small area in Assumption~\ref{SeparatingSCSetForGammaBeta:Lemma:Assump:gammabetaTri}.  Therefore, for each $\rho \in \cR$, we can form the sequence of saddle connections associated to $\rho$
$$(\rho=\sigma_0,\sigma_1, \cdots, \sigma_m )$$
and the partition of $\cR$ and its auxiliary saddle connections.   Let $\ell_i$ be the logarithmic scale of $\sigma_i$.

We split the proof into two cases (exactly as we did in the proof of Theorem~\ref{CountTriangles}): $2^{2\ell_m}\geq \epsilon^2 2^j|\gamma|$ and $2^{2\ell_m} < \epsilon^2 2^j|\gamma|$.  The second case will be split further into two subcases.  We remark that both the first case and the first subcase of the second case hold for all $(X, \omega) \in \cH^{hyp}_1(\kappa)$.  Only the second subcase of the second case is an almost everywhere result requiring the Borel-Cantelli Lemma (Corollary~\ref{coro:bad-smallmeasure}).

\paragraph{Case I:} $2^{2\ell_m}\geq \epsilon^2  2^j|\gamma|$.

By Lemma~\ref{lem:badtimes} Item~\ref{lem:badtimes:AngleUpperBound}, the angle every $\rho$ makes with $\gamma$ is at most $\frac{20}{2^j|\gamma|}$.
 From Lemma~\ref{SeparatingSCSetForGammaBeta:Lemma} Item~\ref{SeparatingSCSetForGammaBeta:Lemma:sigiAngsigip1}, the angle between $\sigma_i$ and $\sigma_{i+1}$ is at most $\frac{c_2\epsilon^2 A}{2^{\ell_i+\ell_{i+1}}}$, and from Item~\ref{SeparatingSCSetForGammaBeta:Lemma:sigiLengthsigip1} $|\sigma_{i+1}| \leq  c_2|\sigma_i|$.  So if $2^{2\ell_m}\geq \epsilon^2  2^j|\gamma|$, then there is a constant $c'$ such that the angle between $\sigma_m$ and $\gamma$ is at most 
$$\frac{c' A}{2^j|\gamma|}\epsilon^2 \leq \frac{c'}{2^j|\gamma|}\epsilon^2.$$
Since the saddle connections $\sigma_m$ are $\epsilon$-isolated by Lemma~\ref{SeparatingSCSetForGammaBeta:Lemma} Item~\ref{SeparatingSCSetForGammaBeta:Lemma:sigmIso2}, the number of saddle connections $\sigma_m$ is at most 
\begin{eqnarray}\label{eq:number-sigma-m}
c' \frac{2^{2\ell_m}}{2^j|\gamma|}.
\end{eqnarray}

We modify the proof  of Lemma~\ref{Lemma:Triangles:BigSigmam} to show that
\begin{lemma}\label{lem:linear-bad}
There is a constant $c''$ such that 
the number of saddle connections $\rho$ in $\cR$ that determine $\sigma_m$ satisfying Case I, i.e., $2^{2\ell_m}\geq \epsilon^2  2^j|\gamma|$,  is bounded above by
$$c'' \frac{ 2^{2\ell}}{ 2^j |\gamma|}.$$
\end{lemma}
\begin{proof}

As we did in the proof of Lemma~\ref{Lemma:Triangles:BigSigmam},
we look at all possible sequences of saddle connections associated to $\rho$
$$\left(\rho=\sigma_0, \sigma_1, \cdots, \sigma_m\right)$$
such that the logarithmic scales of the saddle connections are specified by $\bm{\ell} =\left(\ell_0, \ell_1, \cdots, \ell_m\right)$, where $\ell_0=\ell$ and $\ell_m$ satisfies $2^{2\ell_m}\geq \epsilon^2  2^j|\gamma|$.
Instead of \eqref{eq:combine-linear}, now for each $\bm{\ell}$, the number of $\rho$ is bounded above by
\begin{equation}\label{eq:combine-linear2}
{c}\cdot  2^{\ell-\ell_1}\cdots 2^{\ell_{m-1}-\ell_m}\frac{2^{2\ell_m}}{2^{j}|\gamma|}=c\cdot \frac{2^{\ell+\ell_m}}{2^{j}|\gamma|}.
\end{equation}
This is because the logarithmic scale of $\rho=\sigma_0$ is $\ell$ and the number of  $\sigma_m$ is bounded above by
$c' \frac{2^{2\ell_m}}{2^j|\gamma|}$, as we have shown in \eqref{eq:number-sigma-m}.

As before, we have to sum over all possible tuples $\bm{\ell}$.   The number of $\rho$ over all possible $\bm{\ell}$ is bounded by
\begin{eqnarray*}
&&\sum_{\substack{0\leq i\leq m-1\\ \ell_{i+1}\leq \ell_i}} c\cdot \frac{2^{\ell+\ell_m}}{2^{j}|\gamma|} \\
&\leq& c''\frac{2^{2\ell}}{2^j|\gamma|}
\end{eqnarray*} 
for some constant $c''$ depending on the stratum.
\end{proof}

Now each $\rho \in \cR$ determines an interval of $\lambda$-measure $\frac{2^j|\gamma|^2}{2^\ell A}$ by Lemma~\ref{lem:badtimes} Item~\ref{lem:badtimes:RhoAvailable},
so the sum of the lengths of these intervals taken up by all bad $\rho$ of all logarithmic scales $\ell$ is 
$\frac{c'' 2^\ell|\gamma|}{A}$.
Together with the  property of being bad, i.e., $2^\ell\leq 2 m_0 2^j$ by Lemma~\ref{lem:badtimes} Item~\ref{lem:badtimes:nobad}, this gives
  $\frac{4 c'' m_0 2^j |\gamma|}{A}$ for the total $\lambda$-measure.  
By our choice of $m_0$, $m_0 < \frac{1}{2^7 d c''} = \frac{\epsilon^2}{2^7 d c c'}$. This  proves  Proposition~\ref{lem:badtime}  in the first case.

\paragraph{Case IIa:} Set  $M = 2 (d+1)$. Assume that
\begin{equation*}
2^{2\ell_m}\leq \epsilon^2 2^j|\gamma| \ \text{and}\  2^{\ell_m}\geq  2^{j/24M}|\gamma|.
\end{equation*}

The first inequality says there is at most a single  $\sigma_m$ of logarithmic scale $\ell_m$. 
By cutting $Y$ along $\sigma_m \cup \tau(\sigma_m)$, 
we can reduce to a stratum of lower dimension and conclude that for some $N$, by Theorem~\ref{CountAlphaSC:Theorem3} the number of $\rho$
is at most 
$$ C_0\frac{2^{\ell}}{2^{\ell_m}} \left( \log  \frac{2^{\ell}}{2^{\ell_m}}\right)^{N}.$$

Each $\rho$ determines an interval of $\lambda$-measure $\frac{2^j |\gamma|^2}{2^\ell A}$ by Lemma~\ref{lem:badtimes} Item~\ref{lem:badtimes:RhoAvailable}.
The second assumption in this case implies the sum of lengths taken up by all bad $\rho$  of horizontal logarithmic scale $\ell$ is bounded by
$$\frac{2^j |\gamma|^2}{2^\ell A}  C_0\frac{2^{\ell}}{2^{j/24M}|\gamma|} \left( \log \frac{2^{\ell}}{2^{j/24M}|\gamma|}\right)^{N} = C_0 \frac{2^{j }}{2^{j/24M} A}|\gamma|\left(\log \frac{2^{\ell}}{2^{j/24M}|\gamma|}\right)^N.$$
Summing over values of all $\ell$ with  $2^{\ell}\leq 2m_0 2^j$ and  $2^\ell \geq c_2^{-m} 2^m\geq c_2^{-m}2^{j/24M}|\gamma|$ (that is, the negative values of $\ell$ are bounded in absolute value), we see that the sum of lengths taken up by all bad $\rho$ is bounded by 
$$ C_0' j \frac{2^{j }}{2^{j/24M} A}|\gamma|\left(\log \frac{2^{j}}{2^{j/24M}|\gamma|}\right)^N.$$
Since $|\gamma|$ is bounded below by the systole of $(X, \omega)$ and $j$ is allowed to depend on $(X, \omega)$, this quantity is less than $\frac{2^j|\gamma|}{2^5 d A}$, if $j$ is sufficiently large.

\paragraph{Case IIb:} On $Y$, we have $\sigma_m$ with $
2^{\ell_m} < 2^{j/24M}|\gamma|$.
Assume that $j$ is sufficiently large such that $|\gamma|\leq 2^{j/24M}$.
Then in this case we have $2^{\ell_m} \leq 2^{j/3M}$.
That is, 
\begin{equation*}
\label{eq:remainingcase}
\ell_m \leq \frac{j}{3M}.
\end{equation*}

We have $\ell_m \leq \frac{\ell}{M}$ because $\ell\geq j/3$.
Since $\ell -\ell_m \geq \ell-\frac{\ell}{M}$ and $m\leq d-2$, 
there exists some largest index $i\leq m-1$ with the property 
\begin{equation}\label{equ:jm-small-1}
 \ell_i-\ell_{i+1}\geq \frac{\ell}{2M}. 
\end{equation}
Therefore, after the index $i$, Inequality~\eqref{equ:jm-small-1} reverses to yield
\begin{eqnarray*}
	\ell_{i+1}&=&(\ell_{i+1}-\ell_{i+2})+(\ell_{i+2}-\ell_{i+3})+\dots +(\ell_{m-1}-\ell_{m})+\ell_m \\
&<& \frac{m\ell}{2M}+\frac{\ell}{M}=\frac{m+2}{4(d+1)}\ell \\
	&\leq&  \frac{d+1}{4(d+1)}\ell = \frac{\ell}{4}.
\end{eqnarray*}
That is
\begin{equation}\label{equ:jm-small-2}
\ell_{i+1} < \frac{\ell}{4}.
\end{equation}

To conclude, we claim that any translation surface $(X, \omega)$ satisfying Case IIb is contained in $\Omega_2(\ell, c_3, M, \delta)$ for appropriate constants.  Let $\delta/2$ be the systole of $(X, \omega)$ so that it lies in the $\delta$-thick part of its stratum.  From Property~\ref{Omega2:Gammalogarithmic scale} the saddle connections $\rho \in \cR$ have logarithmic scale $\ell$.  In Property~\ref{Omega2:Sigmalogarithmic scale}, $\sigma$ is $\sigma_{i+1}$ above, which has logarithmic scale $\ell/4$ by Inequality~\eqref{equ:jm-small-2}.

Furthermore, we claim that the angle condition in Property~\ref{Omega2:Sigmalogarithmic scale} holds, and we show that the angle between $\sigma_{i+1}$ and $\rho = \sigma_0$ is at most $c_3/2^{m+\ell_{i+1}}$, where $m=\ell_{i+1}+ \frac{\ell}{2M}$ (cf. Inequality~\eqref{eq:smallangle}).  Indeed, by Lemma~\ref{SeparatingSCSetForGammaBeta:Lemma} Item~\ref{SeparatingSCSetForGammaBeta:Lemma:sigiAngsigip1}, for each $k \leq i$, the angle between $\sigma_k$ and $\sigma_{k-1}$ is at most $\frac{c_2\epsilon^2}{2^{\ell_k+\ell_{k-1}}} \leq \frac{c_3\epsilon^2}{2^{\ell_{i+1} + \ell_i}}$ for a constant $c_3$.    Therefore, taken together the angle between $\rho$ and $\sigma_{i+1}$ is at most $\frac{c}{2^{\ell_{i+1} + \ell_i}}$, for some $c > 0$.  On the other hand, $\ell_i-\ell_{i+1}\geq \ell/2M$ implies $\ell_i+\ell_{i+1}\geq 2\ell_{i+1}+\ell/2M$ and yields the desired bound.

Finally, $\beta$ in Property~\ref{Omega2:Rhologarithmic scale} is $\gamma$ in the current setting.  Thus, if $\gamma$ has logarithmic scale $p$, then the assumptions $|\gamma|\leq 2^{j/24M}$ and $j/3 \leq \ell$ above implies $p\leq \ell/8M$, and $\gamma$ makes an angle with $\rho$ at most $1/2^{\ell+p}$. This completes the proof that  $(X, \omega) \in \Omega_2(\ell, c_3, M, \delta)$.

Since $(X, \omega) \in \Omega_2(\ell, c_3, M, \delta)$, the Borel-Cantelli Lemma (Corollary \ref{coro:bad-smallmeasure})
shows that the final case would not happen for $\mu$-almost every $(X,\omega)$
if we take $j$ to be sufficiently large. 
\end{proof}

\begin{proposition}
\label{BadSaddleConnBd:Prop}
Given the setup above, we have that $\mu$-almost every $(X,\omega) \in \cH^{hyp}_{1}(\kappa)$ has the following property.  There exists a constant $T_0$ depending on $(X,\omega)$ and $|\gamma|$ such that if $T\geq T_0$, then the set of times in the interval $I = \left[\frac{T|\gamma|}{2A},\frac{T|\gamma|}{A}\right]$ for which the available saddle connections that are bad for $\gamma$ on $Y$ take up a subset of $\lambda$-measure at most
$$ \frac{T|\gamma|}{8 dA}.$$
\end{proposition}

\begin{proof}
Since we are free to choose any interval of times on which we see the desired result, we choose the interval $I$ in the statement of the Proposition.  We take $T$ sufficiently large so that the assumptions of Propositions~\ref{prop:smalllogarithmic scale:MeasureZero} and \ref{lem:badtime} are satisfied.

By Lemma~\ref{lem:badtimes} Item~\ref{lem:badtimes:nobad}, $2^\ell  <2m_0 T$, so Proposition~\ref{lem:bad-isolated} can be applied to yield the desired result for the subset of bad $\epsilon$-isolated saddle connections.

If there are bad saddle connections that are not $\epsilon$-isolated, then they necessarily have logarithmic scale at least $j/3$ by Proposition~\ref{prop:smalllogarithmic scale:MeasureZero} after restricting to a full measure subset of the stratum.  On the other hand, if the bad saddle connections have logarithmic scale at least $j/3$ and are not $\epsilon$-isolated, then we have the desired result on a full measure subset of the stratum by Proposition~\ref{lem:badtime}.  
\end{proof}

\subsection{Proofs of the Lower Bound Theorems}
\label{Sect:LowerBdProof}

Theorem~\ref{thm:lower-general} will be proven by induction on the number of nested cylinders contained in the given surface (see the definition of succession of cylinders in Section~\ref{PfLowerBdThm:sect} below).  To perform the induction, we need to be able to excise subsurfaces bounded by a pair of non-invariant saddle connections, identify the non-invariant saddle connections to an invariant one, and have linear growth on the subsurface relative to the new invariant saddle connection.  For this Theorem~\ref{thm:uniformtostart} is insufficient and we need the following strengthening of it.

\begin{theorem}
\label{thm:uniformtostart:genstatement}
Let $\cH^{hyp}(\kappa)$ be a stratum with $d = dim_\bC\,\cH^{hyp}(\kappa)$.  For $(X,\omega)\in \mathcal{H}_1^{hyp}(\kappa)$ assume
\begin{itemize}
\item Either $(X, \omega)$ has an invariant slit $\beta$ and we let $\gamma = \beta$ and $Y = (X, \omega)$ below, or
\item $(X, \omega)$ has a non-invariant saddle connection $\gamma$ in which case we cut along $\gamma \cup \tau(\gamma)$, remove a resulting component, and glue along the resulting boundaries to obtain a subsurface $Y \in \cH^{hyp}(\kappa')$ of area $A < 1$ with a slit $\gamma$ in a smaller dimensional stratum.
\end{itemize}
Let $\gamma$ have logarithmic scale $j_0$.  Then there exists $c > 0$ depending only on the stratum such that for $\mu$-almost every surface $(X,\omega)\in \mathcal{H}_1^{hyp}(\kappa)$, there exists $L_0$ depending on $(X,\omega)$ and $j_0$ such that for every $L\geq L_0$, the number of simple cylinders of length less than $L$ containing $\gamma$ on $Y$ is at least $c \frac{L}{|\gamma|}$.
\end{theorem}

\subsubsection{Producing Cylinders}\label{sec:cylinder-lowerbound}

Using the results proven above, we are finally able to prove Theorem~\ref{thm:uniformtostart:genstatement}.  

\begin{lemma}
\label{lem:covering}
Suppose that  $(X, \omega) \in \cH_A^{hyp}(\kappa)$ has a horizontal slit $\beta$ and that the set of times for which the bad saddle connections are available take up a subset of $\lambda$-measure at most
$\frac{1}{8d}\frac{T|\beta|}{A}$
in the interval $I=\left[\frac{T|\beta|}{2A},\frac{T|\beta|}{A}\right]$.
Then there exists a constant $c$ depending on $\cH^{hyp}(\kappa)$ such that there are at least $c \frac{|T|}{|\beta|}$ many disjoint intervals of length $\frac{(d-2)|\beta|^2}{m_0A}$ contained in $\mathbf{Good}$ each of which contains a good saddle connection corresponding to a simple cylinder containing $\beta$.
\end{lemma}

\begin{proof}
Let  $I_1, I_2, \cdots, I_n$ be the bad intervals such that their union is $\mathbf{Bad}$. 
We define for each $j$ the interval $\tilde{I}_j$ with the same center as $I_j$ and $|\tilde{I}_j|=d|I_j|$.  Define
$$I_{0} = \left[\frac{T|\beta|}{2A},\frac{T|\beta|}{A}\right] \setminus \bigcup_j \tilde{I}_j.$$
Since $\sum\limits_j |\tilde{I}_j| \leq \frac{T|\beta|}{4A}$ by assumption, we have   
$$\left|I_0\right| \geq \frac{T|\beta|}{4A}.$$  Therefore, we need at least $\frac{T|\beta|/4A}{  \frac{(d-2)|\beta|^2}{m_0A}}$ good intervals to cover $I_0$, each of which has length 
at most
 $ \frac{(d-2)|\beta|^2}{m_0A}$ by Lemma~\ref{TimeAvailable:Lemma} and the discussion immediately following it.

Now we claim that each point $x\in I_0$ is contained in a good interval of length $\frac{(d-2)|\beta|^2}{m_0A}$.  Otherwise, let $x$ be the center of an interval $J_x$ with length $\frac{(d-2)|\beta|^2}{m_0A}$, and assume by contradiction that $J_x$ intersects a bad interval $I_k$ for some $k$.  By Lemma~\ref{TimeAvailable:Lemma}, each bad interval $I_k$ has length at least $\frac{|\beta|^2}{m_0 A}$, so 
$$\frac{|J_x|}{2} = \frac{(d-2)|\beta|^2}{2m_0 A} \leq \frac{(d-2)|I_k|}{2}.$$
The assumption that $x \in I_0$ implies the distance from $x$ to $I_k$ is at least $\frac{|\tilde{I}_k|  - |I_k|}{2} = \frac{(d-1)|I_k|}{2}$.  However, the contradiction assumption implies $\frac{|J_x|}{2} \geq \frac{(d-1)|I_k|}{2}$, otherwise, $J_x$ would be disjoint from $I_k$.  This contradiction proves the claim.

Recall that from the NPS-algorithm, every available saddle connection does not necessarily represent a cylinder, but for every set of $(d-2)$ consecutive available saddle connections, there does exist a saddle connection corresponding to a simple cylinder.  
Thus, there are at least 
$$\frac{\frac{T|\beta|}{4 A}}{\frac{(d-2)|\beta|^2}{m_0 A}} = \frac{m_0}{4(d-2)} \frac{T}{|\beta|}$$
such intervals of times for which a good saddle connection is available corresponding to a simple cylinder containing $\beta$.  Let $c = \frac{m_0}{4(d-2)}$ to conclude that there are at least $c \frac{T}{|\beta|}$ cylinders in the interval of times $I$.
\end{proof}

\begin{remark}
The lemma  applies to any non-invariant saddle connection 
$\gamma$, considered as an invariant saddle connection 
on the surface $Y$ that is obtained by cutting and gluing along $\gamma\cup \tau(\gamma)$. See the discussion before Proposition
\ref{prop:smalllogarithmic scale:MeasureZero}.
\end{remark}

\paragraph{Converting Time to Length:}

Proposition \ref{BadSaddleConnBd:Prop} and Lemma \ref{lem:covering}
together prove that we can find the desired number of intervals of time with available saddle connections corresponding to simple cylinders containing the slit $\beta$. 
A priori, we do not know that the circumferences of those cylinders lie in the desired interval of \emph{lengths} because the above discussion only concerned the \emph{times} when those cylinders appeared under the horocycle flow.  Nevertheless, an elementary computation resolves this discrepancy between time and length.

\begin{proof}[Proof of Theorem \ref{thm:uniformtostart:genstatement}]
Proposition~\ref{BadSaddleConnBd:Prop} justifies the assumption on the $\lambda$-measure in Lemma~\ref{lem:covering} and proves that there are $cT/|\gamma|$ many simple cylinders containing the slit $\gamma$ with boundary saddle connection $\rho$ such that $v_\rho \geq \frac{m_0 A}{|\gamma|}$. 

By definition, each such $\rho$ satisfies $sv_\rho+ h_\rho=0$ for some $s \in \left[\frac{T|\gamma|}{2A}, \frac{T|\gamma|}{A}\right]$.
So by Lemma~\ref{rhoLengthTime:Lemma}, we have 
$$|\rho|\leq T + \frac{A}{|\gamma|}.$$
On the other hand, we have
  $$|\rho| \geq  |h_\rho| = s v_\rho \geq \frac{m_0 A}{|\gamma|} \cdot \frac{T|\gamma|}{2A} = \frac{m_0 T}{2}.$$
Given sufficiently large $L$, we take $T= L - \frac{1}{|\gamma|}\leq L - \frac{A}{|\gamma|}$.  
We conclude  that there are at least $c\left( \frac{L}{|\gamma|} - \frac{1}{|\gamma|^2}\right)$ cylinders of length at most  $L$.
By making the coefficient $c$ a little bit smaller,
we conclude the proof of the theorem. 
\end{proof}

\begin{remark}
The proof above gives more, namely
the cylinders that we have constructed have length at least $\frac{m_0T}{2}>\frac{m_0}{2}\left(L-\frac{1}{|\gamma|}\right)$,
which is approximately $\frac{m_0 L}{2}$. 
\end{remark}

\subsubsection{Proof of Theorems~\ref{thm:lower-general} and \ref{thm:lower-general-cyls-noninv}}
\label{PfLowerBdThm:sect}

Now we are able to complete the proof of the lower bound. We use Theorem~\ref{thm:uniformtostart:genstatement} as well as the technical results established above to prove Theorems~\ref{thm:lower-general} and \ref{thm:lower-general-cyls-noninv}.  
Exactly as was the case for the upper bound, a small modification to the proof of Theorem~\ref{thm:lower-general} will prove Theorem~\ref{thm:lower-general-cyls-noninv}.  The difference will be explicitly explained in the proof (in Step V below), so that both proofs are given simultaneously.  The following concept is the basis of the proof below.

\begin{definition}
We define \emph{a succession of cylinders} to be a finite sequence of non-invariant saddle connections 
$\beta=\gamma_{0}, \gamma_1, \gamma_2, \cdots, \gamma_{n}$ such that
\begin{itemize}
  \item $\gamma_1$ and $\tau(\gamma_1)$ bound  a cylinder containing the invariant slit $\beta$.
  \item For each $i \geq 1$, cut along $\gamma_i \cup \tau(\gamma_i)$ and glue the resulting boundaries to obtain a closed surface $X_{\gamma_i}$ with an invariant slit $\gamma_i$.  Then $\gamma_{i+1} \cup \tau(\gamma_{i+1})$ bound a cylinder on $X_{\gamma_i}$ containing $\gamma_i$.
\end{itemize}
Furthermore, given a succession of cylinders, for each $i$,  we say that \emph{$\gamma_{i+1}$ is determined by $\gamma_i$}. 
\end{definition}

\begin{lemma}
\label{SuccOfCyls:Lem}
Given $(X, \omega) \in \cH^{hyp}(\kappa)$ and a succession of cylinders $\beta=\gamma_{0}, \gamma_1, \gamma_2, \cdots, \gamma_{n}$, for each $i$, the surface $X_{\gamma_i}$ is a hyperelliptic surface in a lower dimensional stratum than that of $(X, \omega)$ with an invariant saddle connection $\gamma_i$ and $\{\gamma_i,\tau(\gamma_i),\gamma_{i-1},\tau(\gamma_{i-1})\}$ bound a parallelogram on $(X, \omega)$.
\end{lemma}

\begin{proof}
The first part is Lemma~\ref{TwoBdSC:Lemma}.  Since $\gamma_{i+1} \cup \tau (\gamma_{i+1})$ bound a cylinder on $X_{\gamma_i}$, cutting along a saddle connection, namely $\gamma_i$, interior to the cylinder, leaves the cylinder connected.  Therefore, $\{\gamma_i,\tau(\gamma_i),\gamma_{i-1},\tau(\gamma_{i-1})\}$ bound a parallelogram on $(X, \omega)$ as claimed.
\end{proof}

\begin{proof}[Proof of Theorems~\ref{thm:lower-general} and \ref{thm:lower-general-cyls-noninv}]
The proof is by induction and it is separated into five steps.  We consider a succession of cylinders up to $\gamma_n$ and prove Theorem~\ref{thm:lower-general} by induction on $n$.  Step I will cover the base case of the induction, which follows from Theorem~\ref{thm:uniformtostart}.  Step~II states the induction hypothesis and again uses Theorem~\ref{thm:uniformtostart} to obtain the desired count at each step.  Crucial to the induction is to argue that the next collection of saddle connections in the succession of cylinders are distinct from the previous ones, and this is addressed in Lemma~\ref{lem:isolate-2} in Step~III by showing that their angles are separated and are therefore distinct.  (This argument also relies on induction with Step~I addressing the base case.)  Finally, Step~IV concludes the induction.  Step~V proves the two theorems by arguing that the succession of cylinders can contain at most $d-2$ terms (after $\beta = \gamma_0$).  

\paragraph{Step I - The Base Case and an Angle Bound:}

By Theorem~\ref{thm:uniformtostart:genstatement}, we have shown that for $\mu$-almost every (unit area) $(X,\omega)$ there exists $L_0$ (depending on the surface) such that for any $2^k \geq L_0$, there exists $c > 0$ such that the number of simple cylinders $\gamma_1$ containing $\beta$ with length satisfying
$$m_1 2^k \leq |\gamma_1| \leq 2^k,$$
where $m_1=\frac{m_0}{2}$, is at least
$$c \frac{2^k}{|\beta|}.$$
As we have shown in Lemma \ref{GoodCylinderAngle:Lemma}, any two distinct $\gamma_1$
and $\gamma_1'$ are $\sqrt{m_0}$-isolated, i.e., they form an angle at least 
$\frac{m_0}{|\gamma_1||\gamma_1'|}$.  

\paragraph{Step II - The Induction Hypothesis:}

We proceed by
induction on $n$.
The base case of the induction was covered by Step~I.
Given $1 \leq n\leq d-2$,
suppose inductively for $\mu$ almost every $(X,\omega)$,
there  exists $L_0$ such that for any $2^k \geq L_0$, there are at least 
\begin{equation}\label{equ:induction}
c \frac{2^k}{|\beta|} \left( \log \frac{2^k}{|\beta|}\right)^{n-1}
\end{equation} many $\gamma_n$ that satisfy
\begin{itemize}
  \item $ m_1 2^k \leq |\gamma_n| \leq 2^k.$
  \item For each $\gamma_n$, we have a succession of cylinders $\beta, \gamma_1, \gamma_2, \cdots, \gamma_n$.
  \item Any two distinct $\gamma_n$ and $\gamma_n'$ are isolated in the following sense: if $X_{\gamma_{n-1}}$ has area $A_{n-1}$
      and $X_{\gamma_{n-1}'}$ has area $A_{n-1}'$,  then the angle between $\gamma_n$ and $\gamma_n'$ is at least 
    \begin{equation}\label{equ:isolated}
    m_0 \frac{\max\{A_{n-1}, A_{n-1}'\}}{|\gamma_n||\gamma_n'|}.
     \end{equation}
     \end{itemize}
Denote such a collection of $\gamma_n$ by $\Cyl_n(k)$. 
\paragraph{Step III - Linearly Many Distinct Cylinders at Step $n$:}
By definition, each $\gamma_n$ in $\Cyl_n(k)$ determines a subsurface $X_{\gamma_n}$ of area $A < 1$ in a lower dimensional stratum $\mathcal{H}_A^{hyp}(\kappa')$ with an invariant slit $\gamma_n$.  By Theorem~\ref{thm:uniformtostart:genstatement}, for $\mu$-almost every $(X,\omega)$, there exists $L_0$ such that for $L\geq L_0$ the number of simple cylinders containing $\gamma_n$ on $X_{\gamma_n}$ is at least 
\begin{equation}\label{equ:linear-induction}
  c \frac{L}{|\gamma_{n}|}.
\end{equation}
Let  $k$ be such that  $2^k \geq L_0$ and  $L \geq 2^{8Mk}$.   For  each $\gamma_n\in \Cyl_n(k)$  there are at least $c\frac{L}{|\gamma_n|}$
many cylinders $\gamma_{n+1}$ containing $\gamma_n$, with lengths in $[m_1L, L]$.

Now we claim that for $k_1,k_2$ such that $L\geq \max(2^{8Mk_1}, 2^{8Mk_2})$
for any distinct $\gamma_n\in \Cyl_n(k_1)$ and  $\gamma_n' \in  \Cyl_n(k_2)$,
none of the corresponding $\gamma_{n+1}$ and $\gamma_{n+1}'$ are the same.  
To see this, we prove the following lemma, which will also be used in Step IV. 

\begin{lemma}\label{lem:isolate-2}
For $i \in \{1, 2\}$, and $\max(2^{8Mk_1}, 2^{8Mk_2})\leq L$ 
suppose $\gamma_n \in \Cyl_n(k_1)$ and $\gamma_n'\in \Cyl_n(k_2)$. Then for any $\gamma_{n+1}$ and $\gamma_{n+1}'$ determined by $\gamma_n$ and $\gamma_n'$, respectively, the angle 
between $\gamma_{n+1}$ and $\gamma_{n+1}'$ is at least $$m_0 \frac{\max\{A_n, A_n'\}}{|\gamma_{n+1}||\gamma_{n+1}'|}.$$
\end{lemma}

\begin{proof}
The proof of the lemma is based on our inductive assumption in Step~II.
For the base case when $n=1$, $A_{0}=A_0'=1$ so the result holds by Lemma~\ref{GoodCylinderAngle:Lemma}.

Without loss of generality we may assume that the area of $X_{\gamma_n}$ is larger than $X_{\gamma_n'}$,
that is, $A_n \geq A_n'$. 
First note that the angle between $\gamma_{n+1}$ and $\gamma_n$ is at most $$\frac{A_n}{|\gamma_n| |\gamma_{n+1}|}.$$
Similarly, the angle between $\gamma_{n+1}'$ and $\gamma_n'$ is at most $$\frac{A'_n}{|\gamma_n'| |\gamma_{n+1}'|}.$$

If $\gamma_n=\gamma_n'$, then by Lemma~\ref{GoodCylinderAngle:Lemma}
the angle between $\gamma_{n+1}$ and $\gamma_{n+1}'$ must be at least 
$$m_0 \frac{A_n}{|\gamma_{n+1}||\gamma_{n+1}'|}.$$

Suppose $\gamma_n\neq \gamma_n'$. Then by our inductive construction of $\Cyl_n(k)$, the angle between $\gamma_{n}$ and $\gamma_{n}'$ is at least
$$\frac{\max\{A_{n-1}, A_{n-1}'\}}{|\gamma_n||\gamma_n'|}.$$
The bound in \eqref{equ:isolated} holds whether $k_1$ and $k_2$ are equal or not. 
Combining the above two inequalities, we know that the angle 
between $\gamma_{n+1}$ and $\gamma_{n+1}'$ must be at least
\begin{eqnarray*}
&& \frac{\max\{A_{n-1}, A_{n-1}'\}}{|\gamma_n||\gamma_n'|}- \frac{A_n}{|\gamma_{n}||\gamma_{n+1}|}- \frac{A_n'}{|\gamma_{n}'| |\gamma_{n+1}'|}  \\
& \geq &  \left( \frac{|\gamma_{n+1}||\gamma_{n+1}'|}{|\gamma_{n}||\gamma_{n}'|}  - \frac{|\gamma_{n+1}'|}{|\gamma_{n}|} - 
\frac{|\gamma_{n+1}|}{|\gamma_{n}'|}\right) \frac{A_n}{|\gamma_{n+1}||\gamma_{n+1}'|} \\
&\geq & m_0\frac{A_n}{|\gamma_{n+1}||\gamma_{n+1}'|} .\\
 \end{eqnarray*}
The last inequality holds because the ratios $\frac{|\gamma_{n+1}|}{|\gamma_{n}'|} $ and $\frac{|\gamma_{n+1}'|}{|\gamma_{n}|}$
are large by the assumption that $L \geq \max(2^{8Mk_1},2^{8Mk_2})$ (in particular, larger than $3$).
\end{proof}

%According to the above discussion, we conclude from \eqref{equ:induction}
%and \eqref{equ:linear-induction} that on $(X,\omega)$, there are at least 
%$$c^2 \frac{L}{|\beta|} \left( \log\frac{2^k}{|\beta|} \right)^{n-1}$$
%many $\gamma_{n+1} \in \Cyl_{n+1}(k)$, with lengths in $[m_1 L, L]$. 

\paragraph{Step IV - Concluding the Induction:}
We have to construct, for some constant $c'>0$,
$$c' \frac{L}{|\beta|} \left( \log\frac{L}{|\beta|} \right)^{n}$$
many $\gamma_{n+1}$ with lengths in $[m_1 L, L]$. 
%For this purpose, we should allow $k$ to change. 
Consider the length interval
$I=[L_0, L^{1/18M}]$, where $L_0$ is given at stage $n$. The number of disjoint sub-intervals  $[m_1 2^k, 2^k]\subset I$  is approximately $$\frac{\log L}{18M |\log m_1|}.$$

Now we apply Step~III to each $\Cyl_n(\ell)$. 
For each $\Cyl_n(k)$, we are able to construct 
$$c^2  \frac{L}{|\beta|} \left( \log\frac{2^{k}}{|\beta|} \right)^{n-1}$$
many $\gamma_{n+1}$, with lengths in $[m_1L, L]$.

Applying  Lemma~\ref{lem:isolate-2}, we 
see that for any two distinct $k_1, k_2$, if $\gamma_{n+1}$ is determined by $\gamma_n \in \Cyl_n(k_1)$ and $\gamma_{n+1}'$ is determined by $\gamma'_n \in \Cyl_n(k_2)$, then the angle between $\gamma_{n+1}$ and $\gamma_{n+1}'$ is at least 
 \begin{equation*}
   m_0 \frac{\max\{A_n, A_n'\}}{|\gamma_{n+1}||\gamma_{n+1}'|}.
 \end{equation*}
In particular, $\gamma_{n+1}$ and $\gamma_{n+1}'$ are distinct. In other words there is no over counting.
Therefore summing over all possible $k$, we have the desired lower bound for the number of $\gamma_{n+1}$, given by 
$$\sum_{k=\log\frac{M_0}{m_1}}^{\frac{\log L}{18M |\log m_1|}} c^2  \frac{L}{|\beta|} \left( \log\frac{2^{k}}{|\beta|} \right)^{n-1} = c^2  \frac{L}{|\beta|} \sum_{k=\log\frac{M_0}{m_1}}^{\frac{\log L}{18M |\log m_1|}} \left(k - \log |\beta| \right)^{n-1}  \geq c' \frac{L}{|\beta|} \left( \log\frac{L}{|\beta|} \right)^{n},$$
where the coefficient $c'$ depends on $c, m_1$ and $M$, and $L_0$ is already taken sufficiently large so that $\frac{L_0}{m_1}  > |\beta|$.

\paragraph{Step V - Total Count:}

By Lemma~\ref{Parallelogramulation:Lemma}, $(X, \omega)$ can be decomposed into at most $d-1$ parallelograms.  By Lemma~\ref{SuccOfCyls:Lem}, every quadruple of saddle connections $\{\gamma_i,\tau(\gamma_i),\gamma_{i-1},\tau(\gamma_{i-1})\}$ bound a parallelogram on $(X, \omega)$.  Therefore, $i \in \{1, \ldots, d-2\}$ because the succession of saddle connections necessarily terminates when the complement of the parallelograms produced is itself a parallelogram.  In other words, for each $\gamma_{d-2}$, $\{\gamma_{d-2} \cup \tau(\gamma_{d-2})\}$ splits the surface into two components: $X_{\gamma_{d-2}}$, which contains $\beta$ and decomposes into $d-2$ parallelograms, and its complement, which is necessarily a simple cylinder.

By the induction concluded in Step IV, for $\mu$-almost every surface $(X,\omega)$ and for any given invariant $\beta$ on $(X,\omega)$, there exists  $L_0$ depending on $(X,\omega)$ such that for any $L > L_0$, there are at least
\begin{equation*}
  c \frac{L}{|\beta|} \left( \log \frac{L}{|\beta|}\right)^{d-3}
\end{equation*}
saddle connections $\gamma_{d-2}$ contained in a succession of cylinders, for some $c$ depending on the stratum.  Furthermore, since the boundary of every one of these cylinders is a pair of non-invariant saddle connections, we have produced both lower bounds in Theorem~\ref{thm:lower-general-cyls-noninv}.

\

To prove Theorem~\ref{thm:lower-general}, we note that every cylinder with area at most one and circumference $|\gamma_{d-2}|$ contains at least $L/|\gamma_{d-2}|$ saddle connections of length at most $L$ winding around it (which can be seen by taking positive and negative Dehn twists about the cylinder).  Again, by taking all $k$ satisfying $L_0 \leq 2^k \leq L^{\frac{1}{18M}}$
and then summing over $k$, we obtain that there are 
  $$c' \frac{L}{|\beta|} \left( \log \frac{L}{|\beta|}\right)^{d-2}$$
many saddle connections contained in each simple cylinder bounded by $\gamma_{d-2}$, which are necessarily interiorly disjoint from $\beta$ and invariant under the hyperelliptic involution $\tau$ because they cross a simple cylinder, which is invariant by \cite[Lem.~2.1]{LindseyInvCompsHyp}.  Furthermore, we claim that the angle bound in Lemma~\ref{lem:isolate-2} guarantees that these saddle connections are distinct.  Consider
the saddle connection $\gamma_{d-1}$ (resp. $\gamma'_{d-1}$) and the associated boundary saddle connection $\gamma_{d-2}$ (resp. $\gamma'_{d-2}$) of the cylinder containing it.  For $L_0$ sufficiently large, the angle between $\gamma_{d-1}$ and $\gamma_{d-2}$ and the angle between $\gamma'_{d-1}$ and $\gamma'_{d-2}$ are sufficiently small, while the angle between $\gamma_{d-2}$ and $\gamma'_{d-2}$ is sufficiently large.  This completes the proof of Theorem~\ref{thm:lower-general}.
\end{proof}

\appendix

\section{Cantor-Bendixson Rank for Hyperelliptic Translation Surfaces}
\label{HypCBRank:Appendix}

We recall the setup from \cite[$\S$2]{AulicinoCBRank}.  Let $(X, \omega)$ be a translation surface.  Let
$$\Theta(X, \omega) = \{ \theta \in [0, 2\pi) | (X, e^{i \theta} \omega) \text{ admits a vertical saddle connection}\}.$$
When $\omega$ is holomorphic $\Theta(X, \omega)$ is a countable dense set (\cite[Cor.~1.2]{AulicinoCBRank}), and by \cite[Thm.~1.4]{AulicinoCBRank}, if $\omega$ is meromorphic with at least one pole of order one, then $\Theta(X, \omega)$ is closed.

The \emph{Cantor-Bendixson derivative} of a set $S$ is the set $S^* \subset S$ defined by removing all isolated points from $S$.  
The \emph{Cantor-Bendixson rank (CB-rank)} of a closed set is smallest $n$ such that $S^{*n} = S^{*(n+1)}$, where $S^{*n}$ denotes the $n$'th Cantor-Bendixson derivative.  For example, $\{0\} \cup \left\{\frac{1}{n} | n \in \bN\right\}$ has CB-rank two.

In \cite[$\S$6]{AulicinoCBRank}, a sequence of square-tiled hyperelliptic translation surfaces in $\cH^{hyp}(\kappa)$ were produced with CB-rank equal to $d = \dim_\bC \cH^{hyp}(\kappa)$.  Furthermore, \cite[Thm.~7.5]{AulicinoCBRank} shows that for all translation surfaces in $\cH^{hyp}(\kappa)$ the CB-rank is at most $d$.  Here we use the NPS-algorithm (Section~\ref{LowerBoundSetting:Section} above, \cite{NguyenPanSu2020}) to complete the picture of CB-rank for hyperelliptic translation surfaces.

\begin{theorem}
\label{CBRank:hyp}
Let $(X, \omega) \in \cH^{hyp}(\kappa)$ have a slit $\beta$ invariant under the hyperelliptic involution.  Let $d = \dim_\bC \cH^{hyp}(\kappa)$.  Then $\Theta((X, \omega) \setminus \beta)$ has CB-rank $d$.
\end{theorem}

\begin{proof}
That the CB-rank is at most $d$ is given by \cite[Thm.~7.5]{AulicinoCBRank}.  Therefore, we consider the following inductive argument to show that the CB-rank is at least $d$.

For the base case, it is clear that a simple cylinder has CB-rank two, cf. \cite[Prop.~4.5]{AulicinoCBRank}, where the direction of the boundary of the cylinder corresponds to an accumulation point of the set of directions of saddle connections contained in the cylinder (whose directions are isolated among the set of all saddle connections contained in the cylinder).

Next we consider a hyperelliptic translation surface with an invariant saddle connection $\alpha$.  Our inductive hypothesis is that if $(Y, \eta) \in \cH^{hyp}(\kappa')$ with $\dim_\bC \cH^{hyp}(\kappa') = d - 1$ and a boundary invariant saddle connection $\rho$ with direction $\theta_\rho \in [0, 2\pi)$, then $\theta_\rho \in \Theta((Y, \eta) \setminus \rho)^{*(d-2)}$, and $\Theta((Y, \eta) \setminus \rho)$ has CB-rank at least $d-1$.  Therefore, it suffices to show that the direction of $\alpha$, $\theta_\alpha$, on $(X, \omega)$ is an accumulation point of directions of saddle connections $\rho$, which will imply that the set of directions of saddle connections on $(X, \omega)$ is at least one more than $d-1$, so it is at least $d$.

However, this is simply a repeated application of Lemma~\ref{TwoBdSC:Lemma} to each of the cylinders produced by the NPS-algorithm.  By the NPS-algorithm, there exists an infinite sequence of simple cylinders $C_\rho \subset (X, \omega)$ containing $\alpha$.  Since the circumferences of the cylinders $\{C_\rho\}$ necessarily tend to infinity, their angle with $\alpha$ necessarily tends to zero.  Hence, $\theta_\alpha$ is an accumulation point of the directions $\theta_\rho$ of the core curves of the cylinders $C_\rho$.

Consider one such cylinder $C_\rho$, excise the cylinder, and join the two resulting boundary saddle connections.  By Lemma~\ref{TwoBdSC:Lemma}, this results in a hyperelliptic translation surface $Y_\rho$ with an invariant saddle connection $\rho$.  Furthermore, the dimension of this hyperelliptic translation surface is exactly $d-1$ because a cylinder is parameterized by two complex parameters and since $\rho$ is a saddle connection on the resulting translation surface, only one complex parameter was forgotten.

Then
$$\bigcup_\rho \Theta((Y, \eta) \setminus \rho) \subseteq \Theta((X, \omega) \setminus \alpha).$$
Since each set has CB-rank $d-1$ by the inductive hypothesis and $\theta_\rho$ corresponds to the accumulation point realizing the CB-rank, the argument above shows that $\theta_\alpha$ is the accumulation point of the set $\{\theta_\rho | \rho\}$, thus the set of saddle connection directions on $(X, \omega)$ is at least $d$.
\end{proof}

As noted in the introduction to Section~\ref{LowerBound:Section}, we do not know if the almost everywhere assumption on Theorem~\ref{thm:uniformtostart} is necessary.  Theorem~\ref{CBRank:hyp} indicates that the CB-rank itself is not an obstruction to a version of Theorem~\ref{thm:uniformtostart} without the almost everywhere assumption.  Nevertheless, the CB-rank is purely qualitative and if a surface could be constructed where the lengths of the simple cylinders containing $\beta$ grow too quickly, then that would produce an obstruction to the lower bound on the growth rate of saddle connections established in Theorem~\ref{thm:uniformtostart}.  Thus, we pose the following

\begin{question}
Does there exist a hyperelliptic translation surface $(X, \omega)$ with an invariant slit $\beta$ such that the number of saddle connections interiorly disjoint from $\beta$ is $o\hspace{-2pt}\left(L (\log L)^{d-2}\right)$?
\end{question}

\bibliography{fullbibliotex}{}

\providecommand{\bysame}{\leavevmode\hbox to3em{\hrulefill}\thinspace}
\providecommand{\MR}{\relax\ifhmode\unskip\space\fi MR }
\providecommand{\MRhref}[2]{%
  \href{http://www.ams.org/mathscinet-getitem?mr=#1}{#2}
}
\providecommand{\href}[2]{#2}
\begin{thebibliography}{EMM98}

\bibitem[AM24]{AthreyaMasurTransSurfText}
Jayadev~S. Athreya and Howard Masur, \emph{Translation surfaces}, first
  edition. ed., Graduate studies in mathematics, volume 242, American
  Mathematical Society, Providence, Rhode Island, 2024 (eng).

\bibitem[AM26]{AulicinoMasurSlitTorusAsymptotics}
David Aulicino and Howard Masur, \emph{Asymptotics for the number of saddle
  connections on a slit torus}, arXiv:2610.24429 (2026), 1--16.

\bibitem[Aul18]{AulicinoCBRank}
David Aulicino, \emph{The {C}antor--{B}endixson {R}ank of {C}ertain
  {B}ridgeland--{S}mith {S}tability {C}onditions}, Comm. Math. Phys.
  \textbf{357} (2018), no.~2, 791--809. \MR{3767708}

\bibitem[BS15]{BridgelandSmithStabConds}
Tom Bridgeland and Ivan Smith, \emph{Quadratic differentials as stability
  conditions}, Publ. Math. Inst. Hautes \'Etudes Sci. \textbf{121} (2015),
  155--278. \MR{3349833}

\bibitem[Doz19]{DozierEquidisSC}
Benjamin Dozier, \emph{Equidistribution of saddle connections on translation
  surfaces}, J. Mod. Dyn. \textbf{14} (2019), 87--120. \MR{3959357}

\bibitem[EM01]{EskinMasurAsymptForms}
Alex Eskin and Howard Masur, \emph{Asymptotic formulas on flat surfaces},
  Ergodic Theory Dynam. Systems \textbf{21} (2001), no.~2, 443--478.
  \MR{1827113}

\bibitem[EMM98]{EskinMargulisMozesQuantOppenheim}
Alex Eskin, Gregory Margulis, and Shahar Mozes, \emph{Upper bounds and
  asymptotics in a quantitative version of the {O}ppenheim conjecture}, Ann. of
  Math. (2) \textbf{147} (1998), no.~1, 93--141. \MR{1609447}

\bibitem[EMZ03]{EskinMasurZorich}
Alex Eskin, Howard Masur, and Anton Zorich, \emph{Moduli spaces of abelian
  differentials: the principal boundary, counting problems, and the
  {S}iegel-{V}eech constants}, Publ. Math. Inst. Hautes \'Etudes Sci. (2003),
  no.~97, 61--179. \MR{2010740 (2005b:32029)}

\bibitem[FK36]{FK36}
R.~H. {Fox} and R.~B. {Kershner}, \emph{{Concerning the transitive properties
  of geodesics on a rational polyhedron.}}, {Duke Math. J.} \textbf{2} (1936),
  147--150 (English).

\bibitem[FM14]{ForniMatheusSurvey}
Giovanni Forni and Carlos Matheus, \emph{Introduction to {T}eichm\"{u}ller
  theory and its applications to dynamics of interval exchange transformations,
  flows on surfaces and billiards}, J. Mod. Dyn. \textbf{8} (2014), no.~3-4,
  271--436. \MR{3345837}

\bibitem[GMN13]{GMNWallHitchinWKB}
Davide Gaiotto, Gregory~W. Moore, and Andrew Neitzke, \emph{Wall-crossing,
  {H}itchin systems, and the {WKB} approximation}, Adv. Math. \textbf{234}
  (2013), 239--403. \MR{3003931}

\bibitem[HKK17]{HaidenKatzarkovKontsevichFlatSurfsStabStrcts}
F.~Haiden, L.~Katzarkov, and M.~Kontsevich, \emph{Flat surfaces and stability
  structures}, Publ. Math. Inst. Hautes \'Etudes Sci. \textbf{126} (2017),
  247--318. \MR{3735868}

\bibitem[KMS86]{KMS}
Steven Kerckhoff, Howard Masur, and John Smillie, \emph{Ergodicity of billiard
  flows and quadratic differentials}, Ann. of Math. (2) \textbf{124} (1986),
  no.~2, 293--311. \MR{MR855297 (88f:58122)}

\bibitem[KZ03]{KontsevichZorichConnComps}
Maxim Kontsevich and Anton Zorich, \emph{Connected components of the moduli
  spaces of {A}belian differentials with prescribed singularities}, Invent.
  Math. \textbf{153} (2003), no.~3, 631--678. \MR{MR2000471 (2005b:32030)}

\bibitem[Lin15]{LindseyInvCompsHyp}
Kathryn~A. Lindsey, \emph{Counting invariant components of hyperelliptic
  translation surfaces}, Israel J. Math. \textbf{210} (2015), no.~1, 125--146.
  \MR{3430271}

\bibitem[Mas82]{MasFinMeasErg}
Howard Masur, \emph{Interval exchange transformations and measured foliations},
  Ann. of Math. (2) \textbf{115} (1982), no.~1, 169--200. \MR{644018
  (83e:28012)}

\bibitem[Mas88]{MasurGrowthRateTrajsLowerBd}
\bysame, \emph{Lower bounds for the number of saddle connections and closed
  trajectories of a quadratic differential}, Holomorphic functions and moduli,
  {V}ol. {I} ({B}erkeley, {CA}, 1986), Math. Sci. Res. Inst. Publ., vol.~10,
  Springer, New York, 1988, pp.~215--228. \MR{955824}

\bibitem[Mas90]{MasurGrowthRateTrajs}
\bysame, \emph{The growth rate of trajectories of a quadratic differential},
  Ergodic Theory Dynam. Systems \textbf{10} (1990), no.~1, 151--176.
  \MR{1053805}

\bibitem[Mir08]{MirzakhaniGrowthSCGeods}
Maryam Mirzakhani, \emph{Growth of the number of simple closed geodesics on
  hyperbolic surfaces}, Ann. of Math. (2) \textbf{168} (2008), no.~1, 97--125.
  \MR{2415399}

\bibitem[MS91]{MasurSmillieHausdDim}
Howard Masur and John Smillie, \emph{Hausdorff dimension of sets of nonergodic
  measured foliations}, Ann. of Math. (2) \textbf{134} (1991), no.~3, 455--543.
  \MR{1135877}

\bibitem[MT02]{MasTab}
Howard Masur and Serge Tabachnikov, \emph{Rational billiards and flat
  structures}, Handbook of dynamical systems, {V}ol.\ 1{A}, North-Holland,
  Amsterdam, 2002, pp.~1015--1089. \MR{MR1928530 (2003j:37002)}

\bibitem[Nev17]{nevo2017equidistribution}
Amos Nevo, \emph{Equidistribution in measure-preserving actions of semisimple
  groups : case of $sl_2(\mathbb{R})$}, arXiv:1708.03886 (2017).

\bibitem[Ngu11]{NguyenParallelogramH4}
Duc-Manh Nguyen, \emph{Parallelogram decompositions and generic surfaces in
  {$\mathcal{H}^{hyp}(4)$}}, Geom. Topol. \textbf{15} (2011), no.~3,
  1707--1747. \MR{2851075}

\bibitem[NPS20]{NguyenPanSu2020}
Duc-Manh Nguyen, Huiping Pan, and Weixu Su, \emph{Existence of closed geodesics
  through a regular point on translation surfaces}, Math. Ann. \textbf{376}
  (2020), no.~1-2, 583--607. \MR{4055170}

\bibitem[SW10]{SmillieWeissCharLattice}
John Smillie and Barak Weiss, \emph{Characterizations of lattice surfaces},
  Invent. Math. \textbf{180} (2010), no.~3, 535--557. \MR{2609249}

\bibitem[Tah18]{TaharCountSCHighOrd}
Guillaume Tahar, \emph{Counting saddle connections in flat surfaces with poles
  of higher order}, Geom. Dedicata \textbf{196} (2018), 145--186. \MR{3853632}

\bibitem[Tah21]{TaharTopBdCBRk}
\bysame, \emph{A topological bound on the {C}antor-{B}endixson rank of
  meromorphic differentials}, Arnold Math. J. \textbf{7} (2021), no.~2,
  213--223. \MR{4260073}

\bibitem[Vee82]{VeechFinMeasErg}
William~A. Veech, \emph{Gauss measures for transformations on the space of
  interval exchange maps}, Ann. of Math. (2) \textbf{115} (1982), no.~1,
  201--242. \MR{644019 (83g:28036b)}

\bibitem[Vee89]{VeechTeichCurvEisen}
W.~A. Veech, \emph{Teichm\"uller curves in moduli space, {E}isenstein series
  and an application to triangular billiards}, Invent. Math. \textbf{97}
  (1989), no.~3, 553--583. \MR{1005006 (91h:58083a)}

\bibitem[Vee98]{VeechSiegelMeasures}
William~A. Veech, \emph{Siegel measures}, Ann. of Math. (2) \textbf{148}
  (1998), no.~3, 895--944. \MR{1670061}

\bibitem[Wri15]{WrightSurvey}
Alex Wright, \emph{Translation surfaces and their orbit closures: {A}n
  introduction for a broad audience}, EMS Surv. Math. Sci. \textbf{2} (2015),
  no.~1, 63--108. \MR{3354955}

\bibitem[Zor06]{ZorichFlatSurfs}
Anton Zorich, \emph{Flat surfaces}, Frontiers in number theory, physics, and
  geometry. {I}, Springer, Berlin, 2006, pp.~437--583. \MR{2261104
  (2007i:37070)}

\end{thebibliography}

\end{document}